\documentclass[11pt]{article}

\usepackage{amsmath,amssymb,amsfonts,amsthm,graphicx}
\usepackage{url}
\usepackage{cite}

\newtheorem{assumption}{Assumption}[section]
\newtheorem{remark}{Remark}[section]
\newtheorem{example}{Example}[section]
\newtheorem{defin}{Definition}[section]
\newtheorem{thm}{Theorem}[section]
\newtheorem{lem}{Lemma}[section]

\makeatletter 

\def\smallunderbrace#1{\mathop{\vtop{\m@th\ialign{##\crcr  
$\hfil\displaystyle{#1}\hfil$\crcr 
\noalign{\kern3\p@\nointerlineskip}%
\tiny\upbracefill\crcr\noalign{\kern3\p@}}}}\limits} 

\makeatother 

\title{Optimal control for a nonlinear stochastic PDE model\\ of cancer growth} 

\author{Sakine Esmaili$^{a}$ 
\hspace{.5cm} M. R. Eslahchi$^{a,}$\footnote{Corresponding author (M. R. Eslahchi). 
$\newline$ {\em E-mail addresses:} sakine.esmaili@modares.ac.ir (Sakine Esmaili),
eslahchi@modares.ac.ir (M. R. Eslahchi), 
delfim@ua.pt (Delfim F. M. Torres).}
\hspace{.5cm} Delfim F. M. Torres$^{b}$
\vspace{.5cm}$$\\
\small{\em $$\em  $^{a}$Department of Applied Mathematics, 
Faculty of Mathematical Sciences,}\vspace{-1mm}\\
\small{\em Tarbiat Modares University, P.O. Box 14115-134,
Tehran, Iran}\\
\small{\em $^{b}$Center for Research and Development 
in Mathematics and Applications (CIDMA),}\vspace{-1mm}\\
\small{\em Department of Mathematics, University of Aveiro, 3810-193 Aveiro, Portugal}}

\date{\small (This is a preprint of a paper whose final and definite 
form is published in 'Optimization' 
at \url{https://doi.org/10.1080/02331934.2023.2232141})}

\begin{document}

\maketitle 

\begin{abstract}
We study an optimal control problem for a stochastic model of tumour growth with drug application. This model consists of three stochastic hyperbolic equations describing the evolution of tumour cells. It also includes two stochastic parabolic equations describing the diffusions of nutrient   and drug concentrations. Since all systems are subject to many uncertainties, we have added stochastic terms to the deterministic model to consider the random perturbations. Then, we have added control variables to the model according to the medical concepts to control the concentrations of drug and nutrient. In the optimal control problem, we have defined the stochastic and deterministic cost functions and we have proved  the problems have unique optimal controls. For deriving the necessary conditions for optimal control variables, the stochastic adjoint equations are derived. We have proved the stochastic model of tumour growth and the stochastic adjoint equations have unique solutions. For proving the theoretical results, we have used a change of variable which changes the stochastic model and adjoint equations (a.s.) to deterministic equations. Then we have employed the techniques used for deterministic ones to prove the existence and uniqueness of optimal control.   
\end{abstract}




\section{Introduction}

Cancer is one of the most leading  causes  of death  throughout the world. 
For this, researchers in different fields,  have studied tumours from different points of view.  Mathematicians, who are active in the cancer research field,  study mathematical models of tumour growth with different approaches (for instance see \cite{2.,7.,8.1,8.2,4.,8.}). In most of the models, it is assumed that the tumours grow radially symmetric (for instance, see \cite{2.,8.}) because in vitro observations suggest that in early stages  the solid tumour growth  is approximately radially symmetric \cite{4.}.  Available evidences suggest that low concentrations of glucose
and oxygen in the inner regions of spheroids may contribute to the formation of quiescent, hypoxic, anoxic and necrotic cell sub-populations \cite{11.} (see Figure~\ref{pqdcells}). 
\begin{figure}[!t]
\begin{center}
\hspace{-0cm}\includegraphics[width=14cm]{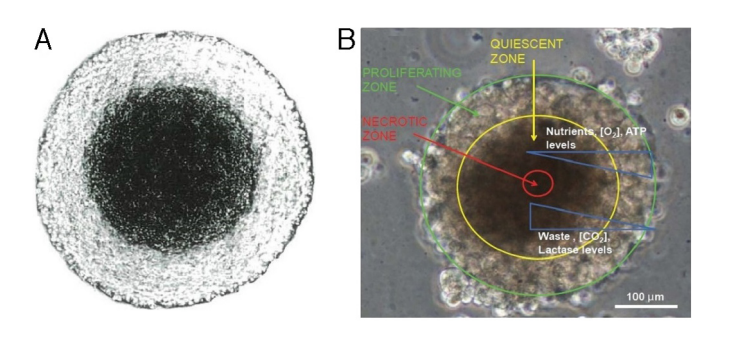}
\vspace{-0.1cm}\caption{\small{Unstained equatorial cryosection of 10 \textup{$\mu$}m thickness through a WiDr human colon adenocarcinoma spheroid as visualized by a phase contrast
microscope (bar, 250 \textup{$\mu$}m) \cite{im1} (A). The inner necrotic core is surrounded by a layer of quiescent cells, which, in turn, is surrounded by a layer of viable
cells \cite{im2}; Pathophysiological gradients in a 3D tumour spheroid \cite{im3} (B). }}\label{pqdcells}
\vspace{-1cm}
\end{center}
\end{figure}
Therefore, in many  tumour growth models  alive cells are divided into  proliferative cells and quiescent cells (for instance, see \cite{2.,8.}).
It is worth mentioning that all systems are affected by  many uncertainties  arisen from environment, experimental variations and so on. Therefore many researchers have studied the stochastic models of tumour growth  to consider  random perturbations and uncertainties     (for instance, see \cite{stoch1,stoch2,stoch3}). On the other hand, due to the importance of the treatment of cancers, many researchers have investigated the models  in which drug therapy of tumours is studied. For example in \cite{2.}, the author studied a model of tumour growth, which is under treatment. In this model, it is assumed that the tumour grows radially symmetric and contains three types of cells including proliferative cells, quiescent cells and dead cells. The evolution of cells are modelled by three nonlinear first-order hyperbolic equations. The effects of nutrient and drug on the tumour growth are also modelled and the diffusions of nutrient and drug concentrations are described by two coupled nonlinear parabolic equations.

 The treatment of tumours is of enormous importance and the tumours can be treated   controlling some parameters such as drug and nutrient concentrations but changing these parameters  may have negative effects on healthy tissues. Therefore, studying the optimal control problem for  the mathematical models of tumour growth   is necessary and is the subject of
many studies. For instance, the authors of \cite{30.} studied an optimal control problem  for a model  of tissue invasion
by solid tumours presented in \cite{31.}.   Calzada et al.  \cite{17.} also  studied the optimal control problem for a free boundary model of tumour growth  that is a slight simplification of the model proposed   by Greenspan \cite{27.,28.}  and   Byrne and Chaplain \cite{29.}. The authors of \cite{29.5} studied the dynamics of breast cancer disease  in the presence of two control strategies,
anti-cancer drugs and ketogenic-diet, against the tumour cells. They analysed the necessary
and sufficient conditions, optimality and transversality conditions using Pontryagin
Maximum Principle. In another investigation  \cite{29.4},  an optimal control approach is presented to analyse
some  treatments for bone metastasis. They considered an ODE model and focus on denosumab
treatment, an anti-resorptive therapy, and radiotherapy treatment and provide proofs of existence and uniqueness of solutions to the corresponding optimal control problems.  

Real environments are stochastic and  in biological systems,
 birth rates, competition coefficients, carrying capacities and other parameters characterizing natural biological systems exhibit random fluctuations \cite{Stability and Complexity in Model Ecosystems}. For instance, in terrestrial ecosystems  the environmental noise tends to be white \cite{the color of}. Even weak noise can  result in unexpected qualitative shifts
in dynamics of nonlinear systems \cite{Bashkirtseva}. Therefore many  researchers have studied the stochastic models of cancers to take into account  uncertainties and deal with more reliable models \cite{stoch1,stoch2,stoch3,29.3,29.1,29.2,Bashkirtseva,stochtumour1,stochtumour2,stochtumour3,stochtumour4}. In paper \cite{29.3},  the total
number of tumours is minimized subject to a stochastic model
in the form of a nonlinear system of four stochastic differential equations (SDEs). The model  describes
tumour-immune dynamics after BCG instillations. The existence and the
stability results are also studied. A model of cancer based upon stochastic controlled versions of the classical Lotka--Volterra equations are studied in \cite{29.1}. The authors considered from a
control point of view the utility of employing  ultrahigh dose-rate flash irradiation. 
In another work \cite{29.2},  a stochastic differential equation model of the evolution of bone metastasis is studied. In this work, the existence and uniqueness of a nonnegative  solution for the stochastic differential equation model  are proved.

In paper
 \cite{8.2}, we have studied two optimal
controls for a free boundary problem modelling tumour
growth with drug application which is deterministic and presented in \cite{2.}. We  have also solved the deterministic optimal control problems numerically and proved the convergence and stability of the used methods \cite{8.3,8.4}. In this paper, we have considered  the  nonlinear deterministic parabolic-hyperbolic free boundary  problem modelling the growth of tumour given in \cite{2.}.  In order to have a more reliable model, we have taken into account the random perturbations in the model by adding some random terms. Most of the researchers study the stochastic ordinary differential equations to study cancer growth in which the effects of uncertainties are considered. But in this paper, we have studied  stochastic partial  differential equations including both parabolic and hyperbolic equations to study cancer growth with more details. Then,  we have controlled the concentrations of drug and nutrient using some control variables to control the growth of  tumour. We have also studied an optimal control problem for the stochastic model of tumour growth  to be able to make correct decisions for destroying the tumour. 
 \noindent  For the reader's convenience, we  summarize our main contributions as follows:
\begin{itemize}
\item Adding stochastic terms to the deterministic partial differential equation model of tumour growth to take into account the random perturbations
and showing that the obtained stochastic model of tumour growth has a unique solution.
\item Constructing a stochastic optimal control problem named "SOCP" by defining a stochastic cost function according to the medical concepts.
\item Obtaining stochastic adjoint equations  and showing that the stochastic adjoint equations have unique solutions.
\item Presenting the explicit  forms of the stochastic control variables in terms of the stochastic adjoint states.
\item Proving that the obtained stochastic optimal control optimizes the deterministic cost function obtained by taking the expectation of the stochastic cost function.
\item Showing  the deterministic cost function defined by expectation also has a unique deterministic optimal control.
\end{itemize}


\section{Preliminaries}

Before presenting the main results of this paper, we first present some  theorems  and lemmas.

\begin{lem}
\label{lem}
Let $\{\mathfrak{M}_n\}_{n=0}^\infty$ be a sequence of nonnegative continuous functions on $[0,T]$ such that
\begin{equation}\label{induc11}
\mathfrak{M}_{n+1}(t)\leq \mathfrak{K}\Big(\int_t^T\int_0^s \mathfrak{M}_n(\xi)d\xi ds\Big),~~ \end{equation}
where $\mathfrak{K}$ is a positive constant. Then there exists $\tilde{T}>0$ such that
\begin{equation*}
\lim_{n\rightarrow\infty}\mathfrak{M}_n(t)=0,~~t\in[0,\tilde{T}].
 \end{equation*}
\end{lem}
\noindent{\it Proof} From \eqref{induc11}, we conclude that 
\begin{equation*}
{M}_{n+1}\leq \mathfrak{K}\Big(\int_0^T\int_0^s M_n d\xi ds\Big),~~ \end{equation*}
where
\[M_k=\max_{0<t<T}| \mathfrak{M}_k(t)|, k\in \mathbb{N}_0.\]
Therefore
\begin{equation*}
{M}_{n+1}\leq \dfrac{\mathfrak{K}T^2}{2}{M}_{n},~~ \end{equation*}
So, the theorem  can be proved easily if we consider $\tilde{T}<\min\left\{\sqrt{\dfrac{2}{\mathfrak{K}}},T\right\}$.
\qed

\begin{defin}\label{def1}
 \textup{(Definition 1.2.2 in \cite{21.}) Let $0<\alpha<1$ and $\Omega\subset \mathbb{R}^n$ be bounded. Then, $f\in C^{\alpha,\frac{\alpha}{2}}(\overline{\Omega}\times[0,T])$ (where  $\overline{\Omega}$ is the closure of $\Omega$  with respect to Euclidean norm) if there exists a positive constant $l$ such that
\[|f(x^1,t_1)-f(x^2,t_2)|\leq l\Big(|x^1-x^2|^2+|t_1-t_2|\Big)^{\frac{\alpha}{2}},~\forall x^1,x^2\in\overline{\Omega},~\forall t_1,t_2\in[0,T].\] 
Furthermore, for any nonnegative integer $k$
\[ C^{2k+\alpha,k+\frac{\alpha}{2}}(\overline{\Omega}\times[0,T]):=\{ f\in C^{\alpha,\frac{\alpha}{2}}(\overline{\Omega}\times[0,T]): \partial_x^{\beta} \partial_t^{i} f\in C^{\alpha,\frac{\alpha}{2}}(\overline{\Omega}\times[0,T]),~~   |\beta |+2i\leq 2k\},\]}
where $\beta = \left( {{\beta_1},{\beta_2},\ldots,{\beta_n}} \right)$, $\left| \beta \right| = \sum\limits_{j = 1}^n {{\beta_j}} $
and
\begin{equation*}
\partial _x^\beta f = \frac{{{\partial ^{{\beta_1} + {\beta_2} + \cdots + {\beta_n}}}f}}{{\partial x_1^{{\beta_1}}\partial x_2^{{\beta_2}} \cdots \partial x_n^{{\beta_n}}}}.
\end{equation*}
\end{defin}

\begin{thm}\label{v principle} \textup{(Theorem 3.2 in \cite{24.})}
Let $X$ be a complete metric space and let $f : X \rightarrow ]-\infty, +\infty]$
be  lower semicontinuous and bounded from below and $\not\equiv+\infty$. Let $\epsilon > 0$ and $x_\epsilon\in X$ be such that
$f(x_\epsilon)\leq\inf\Big\{f(x): x\in X\Big\}+\epsilon.$
Then, there exists $y_\epsilon\in X$ such that
\[f(y_\epsilon)\leq f(x_\epsilon),~~~d(x_\epsilon,y_\epsilon)\leq\epsilon^{\frac{1}{2}},\]
\[f(y_\epsilon) < f(x) + \epsilon^{\frac{1}{2}}d(y_\epsilon, x),~~ \forall x \neq y_\epsilon.\]
\end{thm}

\begin{defin} 
\textup{(Section 1 in \cite{2.})Let $R(t)~(0\leq t\leq T)$ be a positive continuous function. Then 
 we define 
\[W_{p}^{2,1 }(Q_T^R ):=\left\{u\in L^{p} (Q_T^R ) : \partial_x^\alpha \partial_t^k u\in L^{p} (Q_T^R ),~~   |\alpha |+2k\leq 2\right\},\]
with the norm
$||u||_{{W_{p}^{2,1}} (Q_T^R ) }:=\displaystyle{\sum _{|\alpha|+2k\leq 2}}||\partial_x^\alpha \partial
_t^k u||_{L^{p} (Q_T^R )}$, where 
\[Q_T^R:=\Big\{(x,t)\in \mathbb{R}^{3}\times \mathbb{R}: |x|< R(t)  ,0<t<T\Big\}.\]}
\end{defin}

\begin{defin}
\textup{(Section 1 in \cite{2.}) Let $\Omega\subseteq \mathbb{R}^3$ be an open set and $p>\dfrac{5}{2}$. Then we define $D_{p}(\Omega)$, 
the trace space of $W_{p}^{2,1}(\Omega\times]0,T[)$ at $t=0$, as follows:
\begin{equation*}
 D_{p}(\Omega):=\Big\{\varphi :\exists u\in W_{p}^{2,1}(\Omega\times]0,T[)~s.t.~ u(\cdot, 0) = \varphi \Big\},
\end{equation*}
which is equipped with the norm
\[||\varphi ||_{D_p (\Omega)}  := \inf\Big\{T^{-\frac{1}{p}} ||u||_{W_{p}^{2,1}(\Omega\times]0,T[)} : u\in W_{p}^{2,1} (\Omega\times]0,T[),~  u(.,0)=\varphi\Big\}.\]}
\end{defin}

\begin{defin}\label{def33333}
\textup{ Let $p>1$, $\Omega\subset \mathbb{R}^n$ be a bounded  set and $f_1,\ldots,f_k\in L^p(\Omega)$. Then we define 
\begin{equation*}
\|(f_1,\ldots,f_k)\|_{L^p(\Omega)}=\|f_1\|_{L^p(\Omega)}+\cdots+\|f_k\|_{L^p(\Omega)}.
\end{equation*}}
\end{defin}
\noindent We  have merged  the results of  Lemma 3.1 in \cite{2.} and Lemma 3.3 in \cite{22.} in the form of the following lemma.
\begin{lem}\label{Lemma 1} Let $f\left(\rho ,\tau \right),~\psi(\rho,\tau)$ and $\varphi (\tau )$ be bounded continuous functions on  $  [0, 1]\times [0, T]$ and $\ [0, T]\ (T>0),$ respectively.
Let $\overline{c}$ be a constant and ${c}_{ 0}$ be a function on $[0,1]$ such that $c_0(\left|x\right|)\in D_p(\mathfrak{B})$  for some $p>5$, where $\mathfrak{B}$ is a unit ball in $\mathbb{R}^3$. Then the following problem
\[\dfrac{\partial c }{\partial\tau }=\dfrac{1}{{\rho }^2}\dfrac{\partial }{\partial\rho }\left({\rho }^2\dfrac{\partial c }{\partial\rho }\right)+\varphi \left(\tau \right)\rho \dfrac{\partial c }{\partial\rho }+\psi\left(\rho ,\tau \right)c+f(\rho,\tau) ,~~ 0<\rho <1,~ 0< \tau \leq T,\] 
\[ \dfrac{\partial c}{\partial\rho}(0,\tau)=0,~~c \left(1,\tau \right)=\overline{c},~~ 0\leq\tau\leq T,~c \left(\rho ,0\right)=c_0(\rho),~~ 0\leq\rho \leq 1,\] 
has a unique solution such that $c(|x|,\tau )\in W^{2,1}_p(Q^1_T)$, where $Q^1_T:=\Big\{(x,\tau)\in \mathbb{R}^{3}\times \mathbb{R}: |x|< 1,~0<\tau\leq T\Big\}.$ Also, there exists a positive constant $\mu$ depending only on $p, ~T$, ${||\psi (|x|,\tau )||}_{L^{ \infty}(Q^1_T)}$ and ${||\varphi(|x|,\tau ) ||}_{L^{\infty }(Q^1_T)}$ such that
\begin{equation*}
 {||c (|x|,\tau )||}_{W^{2,1}_p(Q^1_T)}\le  \mu\left(\left|\overline{c }\right|+{\left|\left|{c }_0\left(|x|\right)\right|\right|}_{D_p\left(\mathfrak{B}\right)}+||f(|x|,\tau ) ||_{L^p(Q^1_T)}\right),
\end{equation*}
where $\mu$ is bounded for $T$ in any bounded set. Moreover, there exists a positive constant $d$   such that
\begin{equation*}
{||\partial_x^\beta c (|x|,\tau ) ||}_{L^{\infty }(Q^1_T)}\le s_1T^{\frac{1}{2}-\frac{5}{2p}} {||c (|x|,\tau )||}_{W^{2,1}_p(Q^1_T)}+s_2\delta^{-1-\frac{5}{p}}||c (|x|,\tau )||_{L^p(Q^1_T)},
\end{equation*} 
where $|\beta|=1$ and $\delta=\min\{d,\sqrt{T}\}$ and  $s_1$, $s_2$ are positive constants depended on $p$. Also 
\begin{equation*}
{||c (|x|,\tau )||}_{L^{\infty }(Q^1_T)}\le e^{\mu_0T}\Big(\max\{\left|\overline{c }\right|+{\left|\left|{c }_0(|x| )\right|\right|}_{L^{\infty }(Q^1_T)}\}+T||f(|x|,\tau ) ||_{L^\infty(Q^1_T)}\Big),
\end{equation*} 
where $\mu_0=0$ if $\psi\left(\rho ,\tau \right)\leq 0$ and  $\mu_0=\displaystyle \max_{\overline{Q^1_{T}}}\psi$ otherwise.
\end{lem}

\begin{lem}\label{Lemma2}\textup{(Lemma 3.2 in \cite{2.})} Let $z\left(\rho ,\tau \right)$, $h_{ij}(\rho ,\tau )\ (i, j=1,2,3)$ and $g_i(\rho ,\tau)\ (i= 1, 2, 3)$ be bounded continuous functions on $[0, 1]\times [0, T]$, $z\left(\rho ,\tau \right)$ be continuously differentiable with respect to $\rho$ and $z\left(0,\tau \right)=z\left(1,\tau \right)=0$. Then, for every ${\alpha }_0,\ {\beta }_0,\ {\gamma }_0\in \ C[0, 1]$, the problem
\[\dfrac{\partial\alpha }{\partial\tau }+z\left(\rho ,\tau \right)\dfrac{\partial\alpha }{\partial\rho }=h_{11}\left(\rho ,\tau \right)\alpha +h_{12}\left(\rho ,\tau \right)\beta +h_{13}\left(\rho ,\tau \right)\gamma+g_1(\rho,\tau),~0\le \rho \le 1, \ 0<\tau \le T, \]
\[\dfrac{\partial\beta }{\partial\tau }+z\left(\rho ,\tau \right)\dfrac{\partial\beta }{\partial\rho }=h_{21}\left(\rho ,\tau \right)\alpha +h_{22}\left(\rho ,\tau \right)\beta +h_{23}\left(\rho ,\tau \right)\gamma+g_2(\rho,\tau),~0\le \rho \le 1, \ 0<\tau \le T,\] 
\[\dfrac{\partial\gamma }{\partial\tau }+z\left(\rho ,\tau \right)\dfrac{\partial\gamma }{\partial\rho }=h_{31}\left(\rho ,\tau \right)\alpha +h_{32}\left(\rho ,\tau \right)\beta +h_{33}\left(\rho ,\tau \right)\gamma+g_3(\rho,\tau),~0\le \rho \le 1, \ 0<\tau \le T, \] 
\[\alpha \left(\rho ,0\right)={\alpha }_0\left(\rho \right),\ \beta \left(\rho ,0\right)={\beta }_0\left(\rho \right),\ \gamma \left(\rho ,0\right)={\gamma }_0\left(\rho \right),~\ 0\le \rho \le 1, \] 
has a unique weak solution which is continuous with respect to $(\rho ,\tau )$ and 
\[{\|(\alpha ,\beta ,\gamma) \|}_{L^{\infty }(Q^*_T)}\le e^{TA_0\left(T\right)}\left({\|({\alpha }_0,{\beta }_0,{\gamma }_0)\|}_{L^{\infty }(Q^*_T)}+T{\|(g_1,g_2,g_3)\|}_{L^{\infty }(Q^*_T)}\right),\] 
where $\|.\|_{L^\infty(Q^*_T)}$ is defined in Definition \ref{def33333}, $Q^*_T=]0,1[\times ]0,T[$ and $A_0(T)=2{\max \{{\|h_{ij}\|}_{L^{\infty }(Q^*_T)}i,j=1,2,3\}}$. If $h_{ij}(\rho ,\tau )\ (i, j=1, 2, 3)$ and $g_{i}(\rho ,\tau)\ (i= 1, 2, 3)$ are continuously differentiable with respect to $\rho$ and ${\alpha }_0,\ {\beta }_0,\ {\gamma }_0\in \ C^1[0, 1]$, then the weak solution of problem is a classical solution and we have
\begin{eqnarray*}
&{\left\|\left(\dfrac{\partial \alpha}{\partial\rho },\dfrac{\partial \beta}{\partial\rho },\dfrac{\partial \gamma}{\partial\rho }\right)\right\|}_{L^{\infty }(Q^*_T)}&\nonumber\\
 &\le e^{\left(A_0\left(T\right)+A\left(T\right)\right)T}\Big({\|({\alpha }_0',{\beta }_0',\gamma_0')\|}_{L^{\infty }(Q^*_T)}+TA_1(T)e^{TA\left(T\right)}{\|({\alpha }_0,{\beta }_0,{\gamma }_0)\|}_{L^{\infty }(Q^*_T)}&\\
 &+Te^{TA\left(T\right)}{\|(\dfrac{\partial g_1}{\partial\rho },\dfrac{\partial g_2}{\partial\rho },\dfrac{\partial g_3}{\partial\rho })\|}_{L^{\infty }(Q^*_T)}\Big),~~~~~&
\end{eqnarray*}
where $A\left(T\right)={\|\frac{\partial z}{\partial\rho }\|}_{L^{\infty }(Q^*_T)}$ and  $A_1(T)=2{\max \left\{{\|\frac{\partial h_{ij}}{\partial\rho} \|}_{L^{\infty }(Q^*_T)},~i,j=1,2,3\right\}}.$
\end{lem}


\section{Model}
\label{OCP1111}

\noindent In this section, first we present a deterministic mathematical model of tumour growth introduced in \cite{2.}. Then, we add the stochastic terms to the deterministic model to consider random perturbations and have a more reliable model of tumour growth. After that we show  the stochastic model has a unique solution.

 Now,  we  present  the following deterministic parabolic-hyperbolic free boundary problem modelling tumour growth with drug application, which is  introduced in  \cite{2.} 
 \begin{equation} \label{1*Yzzzz} 
\frac{\partial C(x,t)}{\partial t}=D_1\Delta C(x,t)-(K_1\left(C\right)P+K_2\left(C\right)Q),~~~\ \ \ \ \ x\in \Omega(t)\subset\mathbb{R}^3,\ \ t>0, 
\end{equation} 
\begin{equation} \label{GrindEQ__3_2_Yzzzz} 
 C\left(x,t\right)=\overline{C}(t) ~~on ~\partial\Omega(t),  \ t>0,~~~  C\left(x,0\right)=C_0(x) ~~x\in  \Omega(0),
\end{equation} 
\begin{equation} \label{2*Yzzzz} 
\frac{\partial W(x,t)}{\partial t}=D_2\Delta W(x,t)-(K_3\left(W\right)P+K_4\left(W\right)Q 
),\ \ \ \ \ x\in \Omega(t),\ \ t>0, 
\end{equation} 
\begin{equation} \label{3*Yzzzz} 
 W\left(x,t\right)=\overline{W}(t) ~~on~ \partial\Omega(t),  \ t>0,~~~  W\left(x,0\right)=W_0(x) ~~x\in \Omega(0),
\end{equation} 
\[\frac{\partial P(x,t)}{\partial t}+\mathrm{div}(P(x,t)\overrightarrow{v}(x,t) )\]
\begin{equation} \label{6*Yzzzz} 
=\left[K_B\left(C\right)-K_Q\left(C\right)-K_A\left(C\right)-G_1\left(W\right)\right]P+K_P\left(C\right)Q ~\ x\in \Omega(t),\ t>0,
\end{equation} 
\begin{equation} \label{GrindEQ__3_6_Yzzzz} 
\frac{\partial Q(x,t)}{\partial t}+\mathrm{div}( Q(x,t)\overrightarrow{v}(x,t))=K_Q\left(C\right)P-\left[K_P\left(C\right)+K_D\left(C\right)+G_2\left(W\right)\right]Q~ \ x\in \Omega(t),\ t>0,
\end{equation} 
\begin{equation} \label{7*Yzzzz} 
\frac{\partial D(x,t)}{\partial t}+\mathrm{div}( D(x,t)\overrightarrow{v}(x,t))=[K_A\left(C\right)+G_1\left(W\right)]P+[K_D\left(C\right)+G_2\left(W\right)]Q-K_RD ~\ x\in \Omega(t),\ t>0,
\end{equation} 
\begin{equation} \label{yfza} 
P(x,t)+Q(x,t)+D(x,t)=N~\ x\in \Omega(t),\ t>0,
\end{equation} 
\begin{equation} \label{GrindEQ__3_11_Yzzzz} 
 P\left(x,0\right)=P_0\left(x\right),~~Q\left(x,0\right)=\ Q_0\left(x\right),\ \ \ D\left(x,0\right)=\ D_0\left(x\right),\ x \in\Omega(0),
\end{equation} 
\noindent where $\Omega(t)\subset\mathbb{R}^3$ is the tumour domain at time $t$ and $\overrightarrow{v}$ is the velocity of tumour cells' movement.    The concentrations of  nutrient (e.g., oxygen and glucose) and drug are shown by $C$ and $W$, respectively.  The densities of proliferative cells, quiescent cells and dead cells are presented by $P$, $Q$, and $D$, respectively.  The consumption rates of nutrient by  proliferative cells and quiescent cells are shown by $K_1 (C)$ and $K_2(C)$, respectively.  The consumption rates of drug by  proliferative cells and quiescent cells are presented by $K_3(W)$ and $K_4(W)$, respectively.  The dead rates of  proliferative cells and quiescent cells due to the drug are considered in the model by $G_1(W)$ and $G_2(W)$, respectively.  The transferring rate of quiescent cells to proliferative cells and the rate of  transferring proliferative cells to quiescent cells are presented by $K_P (C)$ and $K_Q(C)$, respectively.   The mitosis rate of proliferative cells that is dependent on nutrient level $C$,  the death rates of proliferative cells and quiescent cells are shown by $K_B(C)$, $K_A(C)$ and $K_D(C)$, respectively.   The constant rate of removing dead cells from the tumour is presented by $K_R$. It is assumed that the tumour tissue is a porous medium so that by Darcy's law, we have  
\begin{equation}\label{Ak6Yzzzz} 
\overrightarrow{v}=-\nabla\sigma~\ in~\Omega(t),\ t>0, 
\end{equation} 
and the boundary conditions for $\sigma$ are as follows:
\[\sigma=\theta\kappa_1~on~\partial\Omega(t),\ t>0,~ \dfrac{\partial \sigma}{\partial n}=-V_n ~on~\partial\Omega(t),\ t>0,\]
where $\sigma$ is the pressure in the tumour, $\theta$ is the surface tension, $\kappa_1$ is the mean curvature of the tumour
surface, $\dfrac{\partial}{\partial n}$ is the derivative in the direction $n$ of the outward normal, and $V_n$ is
the velocity of the free boundary $\partial\Omega(t)$ in the direction $n$. If we sum up (\ref{6*Yzzzz})--(\ref{7*Yzzzz}) then using  (\ref{yfza}) we have
\begin{equation}\label{ererer}
\mathrm{div}(\overrightarrow{v})=\frac{1}{N}\left[K_B\left(C\right)P-K_RD\right]~\ in~\Omega(t),\ t>0.
\end{equation} 
 Since, the tumour considered in this model grows radially symmetric, we can write
\begin{equation}\label{r45678800}
C=C(r,t),~W=W(r,t),~P=P(r,t),~Q=Q(r,t),~D=D(r,t),~~r\leq R(t),~t>0,
\end{equation} 
where $r=|x|$ and $R(t)$ is the radius of tumour at time $t$. Because $\sigma$ is  radially symmetric in the space variable, using (\ref{Ak6Yzzzz}) it is easy to derive that there exists a scalar function $\dot{v}=\dot{v}(|x|,t)$ such that \cite{7.}
\begin{equation}\label{veldar}
\overrightarrow{v}=\dfrac{x}{|x|}\dot{v}(|x|,t),~|x|\leq R(t),~t\geq0.
\end{equation}
Therefore, the model (\ref{1*Yzzzz})--(\ref{ererer})  becomes
\begin{equation} \label{1*} 
\frac{\partial C}{\partial t}=D_1\frac{1}{r^2}\frac{\partial }{\partial r}\left(r^2\frac{\partial C}{\partial r}\right)-f\left(C,P,Q\right),\ \ \ \ \ \ 0<r<R\left(t\right),\ \ t>0, 
\end{equation} 
\begin{equation} \label{GrindEQ__3_2_} 
\frac{\partial C}{\partial r}\left(r,t\right)=0\ at\ r=0,\ C\left(r,t\right)=\overline{C}(t)\ at\ r=R\left(t\right),\ \ t>0, 
\end{equation} 
\begin{equation} \label{2*1}  
 C\left(r,0\right)=C_0(r), ~~ 0\leq r\leq R_0, 
\end{equation} 
\begin{equation} \label{2*} 
\frac{\partial W}{\partial t}=D_2\frac{1}{r^2}\frac{\partial }{\partial r}\left(r^2\frac{\partial W}{\partial r}\right)-g\left(W,P,Q\right),\ \ \ \ \ \ 0<r<R\left(t\right),\ \ t>0, 
\end{equation} 
\begin{equation}
\frac{\partial W}{\partial r}\left(r,t\right)=0\ at\ r=0,\ W\left(r,t\right)=\overline{W}(t)\ at\ r=R\left(t\right),\  \ \ \ t>0, 
\end{equation} 
\begin{equation} \label{3*} 
 W\left(r,0\right)=W_0(r), ~~ 0\leq r\leq R_0, 
\end{equation} 
\begin{equation} \label{6*} 
\frac{\partial P}{\partial t}+\dot{v}\frac{\partial P}{\partial r}=g_{11}\left(C,W,P,Q,D\right)P+g_{12}\left(C,W,P,Q,D\right)Q+g_{13}\left(C,W,P,Q,D\right)D, 
\end{equation} 
\begin{equation} \label{GrindEQ__3_6_} 
\frac{\partial Q}{\partial t}+\dot{v}\frac{\partial Q}{\partial r}=g_{21}\left(C,W,P,Q,D\right)P+g_{22}\left(C,W,P,Q,D\right)Q+g_{23}\left(C,W,P,Q,D\right)D, 
\end{equation} 
\begin{equation} \label{7*} 
\frac{\partial D}{\partial t}+\dot{v}\frac{\partial D}{\partial r}=g_{31}\left(C,W,P,Q,D\right)P+g_{32}\left(C,W,P,Q,D\right)Q+g_{33}\left(C,W,P,Q,D\right)D, 
\end{equation} 
\[0\le r\le R\left(t\right),\ \ t>0,\] 
\begin{equation}\label{Ak6} 
\frac{1}{r^2}\frac{\partial }{\partial r}\left(r^2\dot{v}\right)=h\left(C,W,P,Q,D\right),\ \ 0<r\leq R\left(t\right),\ \ t>0, 
\end{equation} 
\begin{equation}
\dot{v}\left(0,t\right)=0,\ \ t>0, 
\end{equation} 
\begin{equation} \label{Ak7} 
\frac{dR(t)}{dt}=\dot{v}\left(R\left(t\right),t\right),\ \ t>0, 
\end{equation} 
\begin{equation} \label{GrindEQ__3_11_} 
 P\left(r,0\right)=P_0\left(r\right),~~Q\left(r,0\right)=\ Q_0\left(r\right),\ \ \ D\left(r,0\right)=\ D_0\left(r\right),\ \ \ R(0)=R_0,\ \ \ \ 0\le r\le R_0, 
\end{equation} 
where (\ref{Ak6}) comes from (\ref{ererer})--(\ref{veldar}),  also
\[f\left(C,P,Q\right)=K_1\left(C\right)P+K_2\left(C\right)Q,~~~g\left(W,P,Q\right)=K_3\left(W\right)P+K_4\left(W\right)Q,\] 
\[g_{11}\left(C,W,P,Q,D\right)=\left[K_B\left(C\right)-K_Q\left(C\right)-K_A\left(C\right)-G_1\left(W\right)\right]-\frac{1}{N}\left[K_B\left(C\right)P-K_RD\right],\] 
\[g_{12}\left(C,W,P,Q,D\right)=K_P\left(C\right),~~~g_{13}\left(C,W,P,Q,D\right)=0,~g_{21}\left(C,W,P,Q,D\right)=K_Q\left(C\right),\]
\[g_{22}\left(C,W,P,Q,D\right)=-\left[K_P\left(C\right)+K_D\left(C\right)+G_2\left(W\right)\right]-\frac{1}{N}\left[K_B\left(C\right)P-K_RD\right],~g_{23}\left(C,W,P,Q,D\right)=0,\] 
\[g_{31}\left(C,W,P,Q,D\right)=K_A\left(C\right)+G_1\left(W\right),~~~g_{32}\left(C,W,P,Q,D\right)=K_D\left(C\right)+G_2\left(W\right),\] 
\[g_{33}\left(C,W,P,Q,D\right)=-K_R-\frac{1}{N}\left[K_B\left(C\right)P-K_RD\right],~~~h\left(C,W,P,Q,D\right)=\frac{1}{N}\left[K_B\left(C\right)P-K_RD\right].\] 
In order to change the domain of the model from a domain with moving boundary $R(t)$ to a domain with fixed boundary, we have used  the following change of variables  \cite{2.} 
\begin{equation}\label{changevar1}
\rho =\frac{r}{R\left(t\right)},~\tau =\int^t_0{\frac{ds}{R^{{\rm 2}}\left(s\right)}}~,\eta \left(\tau \right){\rm =}R\left(t\right),~c\left(\rho ,\tau \right){\rm =}C\left(r,t\right),~w\left(\rho ,\tau \right){\rm =}W\left(r ,t \right),
\end{equation}
\begin{equation}\label{changevar2}
p\left(\rho ,\tau \right){\rm =}P\left(r,t\right),~q\left(\rho ,\tau \right){\rm =}Q\left(r,t\right),~ a\left(\rho ,\tau \right){\rm =}D\left(r,t\right),~u\left(\rho ,\tau \right)=R\left(t\right)\dot{v}\left(r,t\right){\rm}.
\end{equation}
 Using the change of variables \eqref{changevar1}--\eqref{changevar2}, Equations \eqref{1*}--\eqref{2*1} change to the following parabolic equation
 \begin{equation*}
\dfrac{\partial c(\rho,\tau)}{\partial\tau}=D_{1}\dfrac{1}{\rho^2}\dfrac{\partial }{\partial\rho }\left({\rho }^{2}\dfrac{\partial c(\rho,\tau)}{\partial\rho }\right)+u\left(1, \tau \right)\rho \dfrac{\partial c(\rho,\tau)}{\partial\rho }-{\eta } ^{2}f\left(c,p,q\right),~~{\rm \ 0}<\rho<1,\ \tau {\rm >}0, 
\end{equation*}
\begin{equation*}
\dfrac{\partial c}{\partial\rho }\left(0,\tau \right){\rm =0, }~c\left({\rm 1,}\tau \right){\rm =}c_1(\tau), ~~\ \tau {\rm >}0,
\end{equation*} 
\begin{equation*}
c\left(\rho ,0\right){\rm =}c_0\left(\rho \right),~{\rm  \ 0}\le \rho \le {\rm 1,}
\end{equation*}
and Equations \eqref{2*}--\eqref{3*} change to
\begin{equation*}
\dfrac{\partial w(\rho,\tau)}{\partial\tau}=D_{2}\dfrac{1}{{\rho }^{2}}\dfrac{\partial }{\partial\rho }\left({\rho }^{2}\dfrac{\partial w(\rho,\tau)}{\partial\rho }\right)+u\left(1, \tau \right)\rho \dfrac{\partial w(\rho,\tau)}{\partial\rho }-{\eta }^{2}g\left(w,p,q\right),~~{\rm \ 0}<\rho<1,\ \tau {\rm >}0,
\end{equation*}
\begin{equation*}
\dfrac{\partial w}{\partial\rho }\left(0,\tau \right){\rm =0, }~w\left({\rm 1,}\tau \right){\rm =}w_1(\tau), ~~\ \tau {\rm >}0,
\end{equation*}
\begin{equation*}
w\left(\rho ,0\right){\rm =}w_0(\rho),~{\rm \ \ 0}\le \rho \le {\rm 1,}
\end{equation*}
and Equations \eqref{6*}--\eqref{7*} change to the following hyperbolic equations
\begin{eqnarray*}
&\dfrac{\partial p(\rho,\tau)}{\partial\tau}{\rm +}v(\rho,\tau)\dfrac{\partial p(\rho,\tau)}{\partial\rho }{\rm =}
{\eta^2(g_{{\rm 11}}\left(c,w,p,q,a\right)p{\rm +}g_{{\rm 12}}\left(c,w,p,q,a\right)q{\rm +}g_{{\rm 13}}\left(c,w,p,q,a\right)a)},&\\  
&\dfrac{\partial q(\rho,\tau)}{\partial\tau}{\rm +}v(\rho,\tau)\dfrac{\partial q(\rho,\tau)}{\partial\rho }{\rm =} 
\eta^2(g_{{\rm 21}}\left(c,w,p,q,a\right)p{\rm +}g_{{\rm 22}}\left(c,w,p,q,a\right)q{\rm +}g_{{\rm 23}}\left(c,w,p,q,a\right)a),&\\  
&\dfrac{\partial a(\rho,\tau)}{\partial\tau}{\rm +}v(\rho,\tau)\dfrac{\partial a(\rho,\tau)}{\partial\rho }{\rm =}
\eta^2(g_{{\rm 31}}\left(c,w,p,q,a\right)p{\rm +}g_{{\rm 32}}\left(c,w,p,q,a\right)q{\rm +}g_{{\rm 33}}\left(c,w,p,q,a\right)a),&
\end{eqnarray*}
\[{\rm 0}\le \rho \le {\rm 1,\ }\tau {\rm >}0,\]
\begin{equation*}
~~~~~~~~~~~~~~~~~~~p\left(\rho,0\right)=\ p_0\left(\rho\right),~~~q\left(\rho,0\right)=\ q_0\left(\rho\right),~~~ a\left(\rho,0\right)=a_0\left(\rho\right),~~~~{\rm 0}\le \rho \le {\rm 1,\ }
\end{equation*}
and the equations \eqref{Ak6}--\eqref{Ak7} change to the following  equations
\begin{eqnarray*} 
&\dfrac{{\rm 1}}{{\rho }^{{\rm 2}}}\dfrac{\partial }{\partial\rho }\left({\rho }^{{\rm 2}}u\right){\rm =}\eta^2(\tau)h\left(c,w,p,q,a\right),~{\rm 0<}\rho {\rm \leq}1,\ \tau {\rm >}0,~u\left(0,\tau \right){\rm =0,}~\tau> 0,&\nonumber\\  
&\dfrac{d \eta}{d\tau}=\eta \left(\tau \right)u\left( 1,\tau \right),~\tau> 0,~~\eta(0)=\eta_0,&\\
&v\left(\rho ,\tau \right){\rm =}u\left(\rho ,\tau \right){\rm -}\rho u\left({\rm 1,}\tau \right),~~{\rm  0}\le \rho \le {\rm 1,\ }\tau {\rm >}0,&\nonumber 
\end{eqnarray*}
where 
\[c_1(\tau){\rm =}\overline{C}(t){\rm ,\ \ }w_1(\tau){\rm =}\overline{W}(t){\rm ,\ \ }c_0\left(\rho \right){\rm =}C_0\left(\rho R_0\right){\rm ,\ \ }w_0\left(\rho \right){\rm =}W_0\left(\rho R_0\right),\]
\[p_0\left(\rho \right){\rm =}P_0\left(\rho R_0\right){\rm ,\ \ }q_0\left(\rho \right){\rm =}Q_0\left(\rho R_0\right){\rm ,\ \ }a_0\left(\rho \right){\rm =}D_0\left(\rho R_0\right),~{\eta }_0{\rm =}R_0.\]
 Also, it is assumed that the initial data satisfy the following conditions
\begin{eqnarray}\label{initial1}
&0\le c_0\left(\rho\right)\le c_1(0),~ 0\le w_0\left(\rho\right)\le w_1(0),~ 0\le \rho\le 1,\nonumber&\\
&\dfrac{\partial c_0(\rho)}{\partial \rho}=0~ at~\rho=0,~ c_0(1)=c_1(0),~~\dfrac{\partial w_0(\rho)}{\partial \rho}=0~ at~\rho=0,~w_0(1)=w_1(0),&\\
&p_0\left(\rho\right)\ge 0,\ \ q_0\left(\rho\right)\ge 0,\ \ a_0\left(\rho\right)\ge 0,~~{p}_0\left(\rho\right)+q_0\left(\rho\right)+a_0\left(\rho\right)=N,~ 0\le \rho\le 1\nonumber.&
\end{eqnarray}
In the following, we have added stochastic terms to the deterministic model of tumour growth to consider random perturbations and uncertainties. The obtained stochastic model is as follows
\begin{equation}\label{YAS0o0o0o}
d c=D_{1}\dfrac{1}{\rho^2}\dfrac{\partial }{\partial\rho }\left({\rho }^{2}\dfrac{\partial c}{\partial\rho }\right)d\tau+u\left(1, \tau \right)\rho \dfrac{\partial c}{\partial\rho }d\tau-{\eta } ^{2}f\left(c,p,q\right)d\tau+c\sum_{i=0}^mh^1_id B_i,~ 
\end{equation}
\[{\rm \ 0}<\rho<1,\ \tau {\rm >}0,\]
\begin{equation}\label{jadid10o0o0o}
\dfrac{\partial c}{\partial\rho }\left(0,\tau \right){\rm =0, }~c\left({\rm 1,}\tau \right){\rm =}c_1(\tau), ~~\ \tau {\rm >}0,
\end{equation} 
\begin{equation}\label{opb10o0o0o}
c\left(\rho ,0\right){\rm =}c_0\left(\rho \right),~{\rm  \ 0}\le \rho \le {\rm 1,}
\end{equation}
\begin{equation}\label{YAS6767670o0o0o}
{d w}=D_{2}\dfrac{1}{{\rho }^{2}}\dfrac{\partial }{\partial\rho }\left({\rho }^{2}\dfrac{\partial w}{\partial\rho }\right)d\tau+u\left(1, \tau \right)\rho \dfrac{\partial w}{\partial\rho }d\tau-{\eta }^{2}g\left(w,p,q\right)d\tau+w\sum_{i=0}^mh^2_id B_i,~
\end{equation}
\[{\rm \ 0}<\rho<1,\ \tau {\rm >}0,\]
\begin{equation}\label{jadid20o0o0o}
\dfrac{\partial w}{\partial\rho }\left(0,\tau \right){\rm =0, }~w\left({\rm 1,}\tau \right){\rm =}w_1(\tau), ~~\ \tau {\rm >}0,
\end{equation}
\begin{equation}\label{YAS10o0o0o}
w\left(\rho ,0\right){\rm =}w_0(\rho),~{\rm \ \ 0}\le \rho \le {\rm 1,}
\end{equation}
\begin{eqnarray}\label{pqd0o0o0o}
&{d p}{\rm +}v\dfrac{\partial p}{\partial\rho }d\tau{\rm =}
{\smallunderbrace{\eta^2(g_{{\rm 11}}\left(c,w,p,q,a\right)p{\rm +}g_{{\rm 12}}\left(c,w,p,q,a\right)q{\rm +}g_{{\rm 13}}\left(c,w,p,q,a\right)a)}_{G_1\left(\eta,c,w,p,q,a\right)}d\tau}+p\displaystyle\sum_{i=0}^mh^3_id B_i{\rm , }~&\\  
&{d q}{\rm +}v\dfrac{\partial q}{\partial\rho }d\tau{\rm =} 
\smallunderbrace{\eta^2(g_{{\rm 21}}\left(c,w,p,q,a\right)p{\rm +}g_{{\rm 22}}\left(c,w,p,q,a\right)q{\rm +}g_{{\rm 23}}\left(c,w,p,q,a\right)a)}_{G_2\left(\eta,c,w,p,q,a\right)}d\tau+q\displaystyle\sum_{i=0}^mh^4_id B_i,&\\  
&{da}{\rm +}v\dfrac{\partial a}{\partial\rho }d\tau{\rm =}
\smallunderbrace{\eta^2(g_{{\rm 31}}\left(c,w,p,q,a\right)p{\rm +}g_{{\rm 32}}\left(c,w,p,q,a\right)q{\rm +}g_{{\rm 33}}\left(c,w,p,q,a\right)a)}_{G_3\left(\eta,c,w,p,q,a\right)}d\tau+a\displaystyle\sum_{i=0}^mh^5_id B_i,&
\end{eqnarray}
\[{\rm 0}\le \rho \le {\rm 1,\ }\tau {\rm >}0,\]
\begin{equation}\label{YAS20o0o0o}
~~~~~~~~~~~~~~~~~~~p\left(\rho,0\right)=\ p_0\left(\rho\right),~~~q\left(\rho,0\right)=\ q_0\left(\rho\right),~~~ a\left(\rho,0\right)=a_0\left(\rho\right),~~~~{\rm 0}\le \rho \le {\rm 1,\ }
\end{equation}
\begin{eqnarray}\label{HM3450o0o0o} 
&\dfrac{{\rm 1}}{{\rho }^{{\rm 2}}}\dfrac{\partial }{\partial\rho }\left({\rho }^{{\rm 2}}u\right){\rm =}\eta^2(\tau)h\left(c,w,p,q,a\right),~{\rm 0<}\rho {\rm \leq}1,\ \tau {\rm >}0,~u\left(0,\tau \right){\rm =0,}~\tau> 0,&\nonumber\\  
&{d\eta  (\tau)}=\eta \left(\tau \right)u\left( 1,\tau \right)d\tau+\eta d B_*,~\tau> 0,~~\eta(0)=\eta_0,&\\
&v\left(\rho ,\tau \right){\rm =}u\left(\rho ,\tau \right){\rm -}\rho u\left({\rm 1,}\tau \right),~~{\rm  0}\le \rho \le {\rm 1,\ }\tau {\rm >}0,&\nonumber 
\end{eqnarray}
\[B_*=\sum_{i=0}^mr_iB_i,\]
where  $B(\tau)=(B_0(\tau),\ldots, B_m(\tau))$ is an $m-$dimensional   Brownian motion on $(\mathbf{\Omega},\mathcal{F},\mathbb{P})$ with associated natural filtration $\{\mathcal{F}(\tau)\}_{\tau\geq 0}$ and $r_i$ $(i=0,\ldots,m)$ is constant.
\begin{assumption}
\label{assumpt}
\noindent In  stochastic model (\ref{YAS0o0o0o})--(\ref{HM3450o0o0o}), it is assumed that

\noindent \textbf{A.~} The rates $K_1\left(c\right)$, $K_2\left(c\right)$, $K_3\left(w\right)$, $K_4\left(w\right)$,   $K_A\left(c\right),~K_B\left(c\right),~K_D\left(c\right),~K_P\left(c\right)$, $K_Q(c)$, $G_1(w)$ and $G_2(w)$ are  $C^2$-smooth functions.

\noindent \textbf{B.~} Initial data $p_0,{q}_0$ and $a_0$ are non-negative $C^1$-smooth functions on $\left[0,R_0\right].$

\noindent \textbf{C.~}  Functions $c(|x|,\tau)$ and $w(|x|,\tau)$ on $\partial Q^1_T\setminus(\mathfrak{B}\times\{t=T\})$ are non-negative, also $c(|x|,\tau)=\Psi(x,\tau)$ and $w(|x|,\tau)=\Psi_1(x,\tau)$ on $\partial Q^1_T\setminus(\mathfrak{B}\times\{t=T\})$, where $Q^1_T:=\{(x,\tau)\in \mathbb{R}^{3}\times \mathbb{R}:~ |x|< 1,~0<\tau\leq T\}$, $\mathfrak{B}=\left\{ x\in {\rm \ \mathbb{R}^3:~}| x|\le 1\right\}{\rm}$  and $\Psi(x,\tau),~\Psi_1(x,\tau)\in C^{2+\alpha,1+\frac{\alpha}{2}}(\overline{ Q^1_T})$ (which is defined in Definition \ref{def1}) for some $0<\alpha<1$.

\noindent \textbf{D.~}  For $i=0,1,\ldots,m$ and $j=1,2,\ldots,5$, $h_i^j$ is $C^2$-smooth function and $\dfrac{\partial h_i^j}{\partial \rho}(0)=h_i^j(1)=0$. \\
\end{assumption}

\begin{lem}\label{Lemma 1ito} Let $f\left(\rho ,\tau \right),~\psi(\rho,\tau)$ and $\varphi (\tau )$ be bounded continuous functions on  $ [0, 1]\times [0, T] $ and $\ [0, T]\ (T>0),$ respectively and  $Q^*_T=]0,1[\times ]0,T[$.
Let $\overline{c}$ be  constant and ${c}_{ 0}$ be a function on $[0,1]$ such that $c_0(\left|x\right|)\in D_p(\mathfrak{B})$  for some $p>5$, where $\mathfrak{B}$ is a unit ball in $\mathbb{R}^3 $. Then, the following problem
\begin{equation}\label{ito3}
{d c }=\dfrac{1}{{\rho }^2}\dfrac{\partial }{\partial\rho }\left({\rho }^2\dfrac{\partial c }{\partial\rho }\right)d\tau+\varphi \left(\tau \right)\rho \dfrac{\partial c }{\partial\rho }d\tau+\psi\left(\rho ,\tau \right)cd\tau+f(\rho,\tau)d\tau+c\sum_{i=0}^mh_i(\rho)dB_i ,~~ 0<\rho <1,~ 0< \tau \leq T,
\end{equation} 
\[ \dfrac{\partial c}{\partial\rho}(0,\tau)=0,~~c \left(1,\tau \right)=\overline{c},~~ 0\leq\tau\leq T,~c \left(\rho ,0\right)=c_0(\rho),~~ 0\leq\rho \leq 1,\] 
 for almost every $\omega\in \mathbf{\Omega}$ has a unique continuous solution $c=e^{\sum_{i=0}^mh_i(\rho)B_i}\mathfrak{C}$  such that $\mathfrak{C}(|x|,\tau )\in W^{2,1}_p(Q^1_T)$ ($x\in \mathfrak{B}$) and $c$ is $\mathcal{F}(\tau)-$adapted and
\begin{equation}\label{GH1}
 {||\mathfrak{C} (|x|,\tau )||}_{W^{2,1}_p(Q^1_T)}\le  \mu_1'(\omega)\left(|\overline{c }|+{\left|\left|{c}_0\left(|x|\right)\right|\right|}_{D_p\left(\mathfrak{B}\right)}+||f(|x|,\tau ) ||_{L^p(Q^1_T)}\right),
\end{equation}
\begin{equation}\label{GH2}
{||\partial_x^\beta \mathfrak{C} (|x|,\tau ) ||}_{L^{\infty }(Q^1_T)}\le s'_1(\omega)T^{\frac{1}{2}-\frac{5}{2p}} {||\mathfrak{C} (|x|,\tau )||}_{W^{2,1}_p(Q^1_T)}+s_2'(\omega)\delta^{-1-\frac{5}{p}}||\mathfrak{C}(|x|,\tau )||_{L^p(Q^1_T)},
\end{equation}
and 
\begin{equation}\label{GH3}{||\mathfrak{C}(|x|,\tau ) ||}_{L^{\infty }(Q^1_T)}\le e^{\mu_2'(\omega)T}\Big(\max\{|\overline{c}|+{\left|\left|c_0(|x|,\tau )\right|\right|}_{L^{\infty }(Q^1_T)}\}+T||f(|x|,\tau ) ||_{L^\infty(Q^1_T)}\Big),
\end{equation} 
where $\mu_1'(\omega)$, $s'_1(\omega)$, $s'_2(\omega)$ and $\mu_2'(\omega)$ are positive functions of $\omega$, $|\beta|=1$, $\delta=\min\{d,\sqrt{T}\}$ and $d$ is a positive constant. Moreover, if  $\overline{c}=c_0(\rho)=0$, then
\begin{equation}\label{eqGkkp}
  \int_{\mathfrak{B}} (\mathfrak{C}(|x|,\tau))^2dx\leq e^{\lambda_1^*(\omega)\tau}\int_0^\tau \int_{\mathfrak{B}}  (f(|x|,s))^2dxds~a.s.,\end{equation}
where  
 $\lambda_1^*(\omega)$ depends on $T,~p, ~\|\varphi\|_{L^\infty(Q^*_T)},~\|\psi\|_{L^\infty(Q^*_T)},~\|\dfrac{\partial h_i}{\partial \rho}\|_{L^\infty(Q^*_T)},~\|\dfrac{\partial^2 h_i}{\partial \rho^2}\|_{L^\infty(Q^*_T)}$ and $\|h_i\|_{L^\infty(Q^*_T)}$.
  
\end{lem}
\noindent{\it Proof } See Appendix.\qed\\
\begin{remark}
In Lemma \ref{Lemma 1ito}, we have assumed that $p>5$ to use $t-$Anisotropic Embedding Theorem (Theorem 1.4.1 in \cite{21.}) and Lemma 3.3 in \cite{22.} to be able to prove inequality \eqref{GH2}. Since Lemma \ref{Lemma 1ito} is used to study the solutions of the parabolic equations \eqref{YAS} and \eqref{YAS676767} describing the concentrations of nutrient and drug, therefore condition $p>5$ guarantees the boundedness of $\frac{\partial c}{\partial \rho}$ and $\frac{\partial w}{\partial \rho}$. It means that the concentrations of drug and nutrient are of bounded variation.  Since $\frac{\partial c}{\partial \rho}$ and $\frac{\partial w}{\partial \rho}$ are used in $f^*$ and $g^*$ (see equations \eqref{YAS*1e1} and \eqref{YIH}), boundedness of $\frac{\partial c}{\partial \rho}$ and $\frac{\partial w}{\partial \rho}$ helps us to prove Theorems \ref{AMHZ} and \ref{???} that result in the existence and uniqueness of optimal control. So,  the density of tumour cells can be limited by controlling the control variables affecting the concentrations of drug and nutrient on the boundary and inside the tumour.
\end{remark}

\begin{lem}\label{Lemma2ito} Let $z\left(\rho ,\tau \right)$ , $h_{ij}(\rho ,\tau )\ (i, j=1,2,3)$ and $g_i(\rho ,\tau)\ (i= 1, 2, 3)$ be bounded continuous functions on $[0, 1]\times [0, T]$, $z\left(\rho ,\tau \right)$ be continuously differentiable with respect to $\rho$ and $z\left(0,\tau \right)=z\left(1,\tau \right)=0$. Then, for every ${\alpha }_0,\ {\beta }_0,\ {\gamma }_0\in \ C[0, 1]$, the problem
\begin{equation}\label{ito12}
{d\alpha }+z\left(\rho ,\tau \right)\dfrac{\partial\alpha }{\partial\rho }d\tau=\Big(h_{11}\left(\rho ,\tau \right)\alpha +h_{12}\left(\rho ,\tau \right)\beta +h_{13}\left(\rho ,\tau \right)\gamma+g_1(\rho,\tau)\Big)d\tau+\alpha dW_1,
\end{equation}
\begin{equation}
{d\beta }+z\left(\rho ,\tau \right)\dfrac{\partial\beta }{\partial\rho }d\tau=\Big(h_{21}\left(\rho ,\tau \right)\alpha +h_{22}\left(\rho ,\tau \right)\beta +h_{23}\left(\rho ,\tau \right)\gamma+g_2(\rho,\tau)\Big)d\tau+\beta dW_2,
\end{equation} 
\begin{equation}\label{ito13}
{d\gamma }+z\left(\rho ,\tau \right)\dfrac{\partial\gamma }{\partial\rho }d\tau=\Big(h_{31}\left(\rho ,\tau \right)\alpha +h_{32}\left(\rho ,\tau \right)\beta +h_{33}\left(\rho ,\tau \right)\gamma+g_3(\rho,\tau)\Big)d\tau+\gamma dW_3,
\end{equation} 
\[~0\le \rho \le 1, \ 0<\tau \le T, \]
\[\alpha \left(\rho ,0\right)={\alpha }_0\left(\rho \right),\ \beta \left(\rho ,0\right)={\beta }_0\left(\rho \right),\ \gamma \left(\rho ,0\right)={\gamma }_0\left(\rho \right),~\ 0\le \rho \le 1, \] 
\[W_j=\sum_{i=0}^mh^j_i(\rho)B_i,~~j=1,2,3,\]
for almost every $\omega\in \mathbf{\Omega}$, has a unique weak solution which is continuous with respect to $(\rho ,\tau )$ and 
\begin{equation}\label{khfgt}
{\|(\alpha ,\beta ,\gamma) \|}_{L^{\infty }([0,1])}\le \iota(\omega)\left({\|({\alpha }_0,{\beta }_0,{\gamma }_0)\|}_{L^{\infty }([0,1])}+\int_0^\tau{\|(g_1,g_2,g_3)\|}_{L^{\infty }([0,1])}ds\right),
\end{equation} 
and
\begin{equation}\label{khfgt1}
{\|(\alpha ,\beta ,\gamma) \|}^2_{L^{2 }(Q^*_\tau)}\le \iota_1(\omega)\left({\|(\alpha_0 ,\beta_0 ,\gamma_0)\|}^2_{L^{2 }(Q^*_\tau)}+\int_0^\tau{\|(g_1,g_2,g_3)\|}^2_{L^{2 }(Q^*_s)}ds\right),
\end{equation}
where $\iota$  and $\iota_1$ are positive functions of $\omega$, $Q^*_\tau=]0,1[\times]0,\tau[$ and the first norm, $\|.\|_{L^\infty([0,1])}$, is taken with respect to $\rho\in [0,1]$. If $h_{ij}(\rho ,\tau )\ (i, j=1, 2, 3)$ and $g_{i}(\rho ,\tau)\ (i= 1, 2, 3)$ are continuously differentiable with respect to $\rho$ and ${\alpha }_0,\ {\beta }_0,\ {\gamma }_0\in \ C^1[0, 1]$, then the   solution of the problem is  continuously differentiable with respect to $\rho$ for almost every $\omega\in \mathbf{\Omega}$ and we have
\begin{eqnarray}\label{nemi}
&{\left\|\left(\dfrac{\partial \alpha}{\partial\rho },\dfrac{\partial \beta}{\partial\rho },\dfrac{\partial \gamma}{\partial\rho }\right)\right\|}_{L^{\infty }(Q^*_T)}&\nonumber\\
 &\le\iota_2(\omega)\left({\|({\alpha }_0',{\beta }_0',\gamma_0')\|}_{L^{\infty }(Q^*_T)}+T{\|({\alpha }_0,{\beta }_0,{\gamma }_0)\|}_{L^{\infty }(Q^*_T)}+T{\|(\dfrac{\partial g_1}{\partial\rho },\dfrac{\partial g_2}{\partial\rho },\dfrac{\partial g_3}{\partial\rho })\|}_{L^{\infty }(Q^*_T)}\right),~~~~~&
\end{eqnarray}
where $\iota_2$ is a positive function of $\omega.$ Also $\alpha$, $\beta$ and $\gamma$ are $\mathcal{F}(\tau)-$adapted.
\end{lem}
\noindent{\it Proof } See Appendix.\qed\\

\noindent In the following theorem the existence and uniqueness of the solution of (\ref{YAS0o0o0o})--(\ref{HM3450o0o0o}) are presented. We formally gather the results of  Lemmas \ref{Lemma 1ito} and \ref{Lemma2ito}  in the form of the following theorem.
\begin{thm}\label{mainthm1}
Let Assumption~\ref{assumpt} and initial condition \textup{(\ref{initial1})} be satisfied. Then for almost every $\omega\in \mathbf{\Omega}$ the problem \textup{(\ref{YAS0o0o0o})--(\ref{HM3450o0o0o})} has a unique solution and, for every $T>0$, $\eta\left(\tau\right)\in C\left[0,T\right],\ c,\ w\in C\left({Q^1_T}\right)$  and $p,q,a\in C\left([0,1]\times[0,T]\right)$. Moreover, $c$, $w$, $p$, $q$, $a$ and $\eta$ are $\mathcal{F}(t)$-adapted.
\end{thm}


\section{Optimal control problem}
\label{newopde}

 In this section, we present the following stochastic optimal control problem  (SOCP) in which we add the control variables to the stochastic model to control the concentrations of drug and nutrient on the boundary and inside the tumour to destroy the tumour cells. The optimal control problem is studied to be able to make accurate decisions in order to treat the tumour cells.  We control the concentrations of drug and nutrient using control variables $u_1$, $u_2$, $\overline{c}$ and $\overline{w}$, and for almost every $\omega\in\mathbf{\Omega}$  we minimize
\begin{equation*}\label{OP55557777}
J(u_1,u_2,\overline{c},\overline{w}):=\int_0^1\int_0^T\rho^2(\lambda_1p^2+\lambda_2q^2+(u_1-r_1^*)^2+(u_2-r_2^*)^2)d\tau d\rho+\int_0^T(\overline{c}-r_3^*)^2+(\overline{w}-r_4^*)^2d\tau,
\end{equation*}
\[~~~~(u_1,u_2,\overline{c},\overline{w})\in \mathcal{C}_{ad}\times\mathcal{W}_{ad}\times\mathcal{C}^1_{ad}\times\mathcal{W}^1_{ad},\]
 such that
$J(u^*_1,u^*_2,\overline{c}^*,\overline{w}^*)=\min \Big\{J(u_1,u_2,\overline{c},\overline{w}):(u_1,u_2,\overline{c},\overline{w})\in \mathcal{C}_{ad}\times\mathcal{W}_{ad}\times\mathcal{C}^1_{ad}\times\mathcal{W}^1_{ad}\Big\},$
where 
\begin{equation}\label{YAK*15555}
\mathcal{C}_{ad}:=\Big\{ v\in C([0, 1]\times [0, T])~a.s.: l^{*}_1(\rho,\tau)\leq v(\rho,\tau)\leq l^{**}_1(\rho,\tau)~a.s.\Big\},
\end{equation}
\begin{equation}
 \mathcal{W}_{ad}:=\Big\{ v\in C([0, 1]\times [0, T])~a.s.: l^{*}_2(\rho,\tau) \leq v(\rho,\tau)\leq l^{**}_2(\rho,\tau)~a.s.\Big\},
\end{equation}
\begin{equation}
\mathcal{C}^1_{ad}:=\Big\{ v\in C [0, T]~a.s.:~\dfrac{\partial v}{\partial \tau}\in L^p[0,T],~ p>5, ~v(0)=0,~ l^{*}_3(\tau)\leq v(\tau)\leq l^{**}_3(\tau)~a.s.\Big\},
\end{equation}
\begin{equation}\label{YAK*255555}
\mathcal{W}^1_{ad}:=\Big\{ v\in C [0, T]~a.s.:~\dfrac{\partial v}{\partial\tau}\in L^p[0,T],~ p>5,~v(0)=0,~ l^{*}_4(\tau)\leq v(\tau)\leq l^{**}_4(\tau)~a.s.\Big\},
\end{equation} 
\[r^*_1,r^*_2,l_1^*,l_1^{**},l_2^*,l_2^{**}\in C^{2+\alpha,1+\alpha/2}([0, 1]\times [0, T]) ,\]
\[r_3^*,r_4^*,l_3^{*},l_3^{**},l_4^{*},l_4^{**}\in C^{1+\alpha/2} [0, T],~l_3^{*}(0)=l_3^{**}(0)=l_4^{*}(0)=l_4^{**}(0)=0,~\]
subject to 
\begin{equation}\label{YAS}
d c=D_{1}\dfrac{1}{\rho^2}\dfrac{\partial }{\partial\rho }\left({\rho }^{2}\dfrac{\partial c}{\partial\rho }\right)d\tau+u\left(1, \tau \right)\rho \dfrac{\partial c}{\partial\rho }d\tau-{\eta } ^{2}f\left(c,p,q\right)d\tau+u_1d\tau+c\sum_{i=0}^mh^1_id B_i,~ 
\end{equation}
\[{\rm \ 0}<\rho<1,\ \tau {\rm >}0,\]
\begin{equation}\label{jadid1}
\dfrac{\partial c}{\partial\rho }\left(0,\tau \right){\rm =0, }~c\left({\rm 1,}\tau \right){\rm =}c_1(\tau)+\overline{c}(\tau), ~~\ \tau {\rm >}0,
\end{equation} 
\begin{equation}\label{opb1}
c\left(\rho ,0\right){\rm =}c_0\left(\rho \right),~{\rm  \ 0}\le \rho \le {\rm 1,}
\end{equation}
\begin{equation}\label{YAS676767}
{d w}=D_{2}\dfrac{1}{{\rho }^{2}}\dfrac{\partial }{\partial\rho }\left({\rho }^{2}\dfrac{\partial w}{\partial\rho }\right)d\tau+u\left(1, \tau \right)\rho \dfrac{\partial w}{\partial\rho }d\tau-{\eta }^{2}g\left(w,p,q\right)d\tau+u_2d\tau+w\sum_{i=0}^mh^2_id B_i,~
\end{equation}
\[{\rm \ 0}<\rho<1,\ \tau {\rm >}0,\]
\begin{equation}\label{jadid2}
\dfrac{\partial w}{\partial\rho }\left(0,\tau \right){\rm =0, }~w\left({\rm 1,}\tau \right){\rm =}w_1(\tau)+\overline{w}(\tau), ~~\ \tau {\rm >}0,
\end{equation}
\begin{equation}\label{YAS1}
w\left(\rho ,0\right){\rm =}w_0(\rho),~{\rm \ \ 0}\le \rho \le {\rm 1,}
\end{equation}
\begin{eqnarray}\label{pqd}
&{d p}{\rm +}v\dfrac{\partial p}{\partial\rho }d\tau{\rm =}
{\smallunderbrace{\eta^2(g_{{\rm 11}}\left(c,w,p,q,a\right)p{\rm +}g_{{\rm 12}}\left(c,w,p,q,a\right)q{\rm +}g_{{\rm 13}}\left(c,w,p,q,a\right)a)}_{G_1\left(\eta,c,w,p,q,a\right)}d\tau}+p\displaystyle\sum_{i=0}^mh^3_id B_i{\rm , }~&\\  
&{d q}{\rm +}v\dfrac{\partial q}{\partial\rho }d\tau{\rm =} 
\smallunderbrace{\eta^2(g_{{\rm 21}}\left(c,w,p,q,a\right)p{\rm +}g_{{\rm 22}}\left(c,w,p,q,a\right)q{\rm +}g_{{\rm 23}}\left(c,w,p,q,a\right)a)}_{G_2\left(\eta,c,w,p,q,a\right)}d\tau+q\displaystyle\sum_{i=0}^mh^4_id B_i,&\\  
&{d a}{\rm +}v\dfrac{\partial a}{\partial\rho }d\tau{\rm =}
\smallunderbrace{\eta^2(g_{{\rm 31}}\left(c,w,p,q,a\right)p{\rm +}g_{{\rm 32}}\left(c,w,p,q,a\right)q{\rm +}g_{{\rm 33}}\left(c,w,p,q,a\right)a)}_{G_3\left(\eta,c,w,p,q,a\right)}d\tau+a\displaystyle\sum_{i=0}^mh^5_id B_i,&
\end{eqnarray}
\[{\rm 0}\le \rho \le {\rm 1,\ }\tau {\rm >}0,\]
\begin{equation}\label{YAS2}
~~~~~~~~~~~~~~~~~~~p\left(\rho,0\right)=\ p_0\left(\rho\right),~~~q\left(\rho,0\right)=\ q_0\left(\rho\right),~~~ a\left(\rho,0\right)=a_0\left(\rho\right),~~~~{\rm 0}\le \rho \le {\rm 1,\ }
\end{equation}
\begin{eqnarray}\label{HM345} 
&\dfrac{{\rm 1}}{{\rho }^{{\rm 2}}}\dfrac{\partial }{\partial\rho }\left({\rho }^{{\rm 2}}u\right){\rm =}\eta^2(\tau)h\left(c,w,p,q,a\right),~{\rm 0<}\rho {\rm \leq}1,\ \tau {\rm >}0,~u\left(0,\tau \right){\rm =0,}~\tau> 0,&\nonumber\\  
&{d\eta  (\tau)}=\eta \left(\tau \right)u\left( 1,\tau \right)d\tau+\eta d B_*,~\tau> 0,~~\eta(0)=\eta_0,&\\
&v\left(\rho ,\tau \right){\rm =}u\left(\rho ,\tau \right){\rm -}\rho u\left({\rm 1,}\tau \right),~~{\rm  0}\le \rho \le {\rm 1,\ }\tau {\rm >}0,&\nonumber 
\end{eqnarray}
\[B_*=\sum_{i=0}^mr_iB_i.\]
 Note that (\ref{YAS})--(\ref{HM345})  are obtained from (\ref{YAS0o0o0o})--(\ref{HM3450o0o0o}) by adding  the control variables $u_1$, $u_2$, $\overline{c}$ and $\overline{w}$. 

\begin{thm}\label{??}
Let  Assumption~\ref{assumpt} and initial condition \textup{(\ref{initial1})} be satisfied. Also assume that \\$(c_1,w_1,p_1,q_1,a_1,\eta_1)$ and $(c_2,w_2,p_2,q_2,a_2,\eta_2)$ are the  solutions of  the problem \textup{(\ref{YAS})--(\ref{HM345})} corresponding to $(u^1_1,u^1_2,\overline{c}^1,\overline{w}^1)\in \mathcal{C}_{ad}\times\mathcal{W}_{ad}\times\mathcal{C}^1_{ad}\times\mathcal{W}^1_{ad}$ and $(u^2_1,u^2_2,\overline{c}^2,\overline{w}^2)\in \mathcal{C}_{ad}\times\mathcal{W}_{ad}\times\mathcal{C}^1_{ad}\times\mathcal{W}^1_{ad}$, respectively, where $\mathcal{C}_{ad}\times\mathcal{W}_{ad}\times\mathcal{C}^1_{ad}\times\mathcal{W}^1_{ad}$ is the admissible control set for SOCP. Then, for almost every $\omega\in \mathbf{\Omega}$, there exist positive  $\lambda_0(\omega)$ and $\lambda(\omega)$, which are independent of $(u^1_1,u^1_2,\overline{c}^1,\overline{w}^1)$ and $(u^2_1,u^2_2,\overline{c}^2,\overline{w}^2)$,  such that 
\[{\|(c_1-c_2,w_1-w_2,p_1-p_2,q_1-q_2,a_1-a_2,\eta_1-\eta_2) \|}_{L^{\infty }(Q^*_\tau)} \]
\begin{equation*}
\le \lambda_0(\omega){\|(u^1_1-u^2_1,u^1_2-u^2_2,\overline{c}^1-\overline{c}^2,\overline{w}^1-\overline{w}^2) \|}_{L^{\infty }(Q^*_\tau)},
\end{equation*}
and
\[\|(c_1-c_2,w_1-w_2)\|^2_{L^2(Q^1_\tau)}+\|(p_1-p_2,q_1-q_2,a_1-a_2,\eta_1-\eta_2) \|^2_{L^2(Q^*_\tau)}\]
\[\leq\lambda(\omega){\|(u^1_1-u^2_1,u^1_2-u^2_2,\overline{c}^1-\overline{c}^2,\overline{w}^1-\overline{w}^2) \|}^2_{L^2(Q^*_\tau)},\]
\end{thm}
where  $Q^*_\tau=]0,1[\times]0,\tau[$.
 
\begin{proof}
Since, $(c_1,w_1,p_1,q_1,a_1,\eta_1)$ and $(c_2,w_2,p_2,q_2,a_2,\eta_2)$ are the  solutions of  the problem \textup{(\ref{YAS})--(\ref{HM345})} corresponding to $(u^1_1,u^1_2,\overline{c}^1,\overline{w}^1)\in \mathcal{C}_{ad}\times\mathcal{W}_{ad}\times\mathcal{C}^1_{ad}\times\mathcal{W}^1_{ad}$ and $(u^2_1,u^2_2,\overline{c}^2,\overline{w}^2)\in \mathcal{C}_{ad}\times\mathcal{W}_{ad}\times\mathcal{C}^1_{ad}\times\mathcal{W}^1_{ad}$, respectively, therefore from Lemma \ref{Lemma 1ito}, the problem 
\[d (c_1-c_2)=D_{1}\dfrac{1}{\rho^2}\dfrac{\partial }{\partial\rho }\left({\rho }^{2}\dfrac{\partial (c_1-c_2)}{\partial\rho }\right)d\tau+u^1\left(1, \tau \right)\rho \dfrac{\partial c_1}{\partial\rho }d\tau-u^2\left(1, \tau \right)\rho \dfrac{\partial c_2}{\partial\rho }d\tau\]
\begin{equation*}
-((\eta_1 ) ^{2}f\left(c_1,p_1,q_1\right)-(\eta_2 ) ^{2}f\left(c_2,p_2,q_2\right))d\tau+(u^1_1-u_1^2)d\tau+(c_1-c_2)\sum_{i=0}^mh^1_id B_i,~ 
\end{equation*}
\[{\rm \ 0}<\rho<1,\ \tau {\rm >}0,\]
\begin{equation*}
\dfrac{\partial (c_1-c_2)}{\partial\rho }\left(0,\tau \right){\rm =0, }~(c_1-c_2)\left({\rm 1,}\tau \right){\rm =}\overline{c}^1(\tau)-\overline{c}^2(\tau), ~~\ \tau {\rm >}0,
\end{equation*} 
\begin{equation*}
(c_1-c_2)\left(\rho ,0\right){\rm =}0,~{\rm  \ 0}\le \rho \le {\rm 1,}
\end{equation*}
where 
\[\dfrac{{\rm 1}}{{\rho }^{{\rm 2}}}\dfrac{\partial }{\partial\rho }\left({\rho }^{{\rm 2}}u^1\right){\rm =}(\eta_1)^2h\left(c_1,w_1,p_1,q_1,a_1\right),~{\rm 0<}\rho {\rm \leq}1,\ \tau {\rm >}0,~u^1\left(0,\tau \right){\rm =0,}~\tau> 0,\]
and
\[\dfrac{{\rm 1}}{{\rho }^{{\rm 2}}}\dfrac{\partial }{\partial\rho }\left({\rho }^{{\rm 2}}u^2\right){\rm =}(\eta_2)^2h\left(c_2,w_2,p_2,q_2,a_2\right),~{\rm 0<}\rho {\rm \leq}1,\ \tau {\rm >}0,~u^2\left(0,\tau \right){\rm =0,}~\tau> 0,\]
for almost every $\omega\in \mathbf{\Omega}$, has a unique continuous solution $c_1-c_2=e^{\sum_{i=0}^mh^1_i(\rho)B_i}\mathfrak{C}$  such that $\mathfrak{C}(|x|,\tau )\in W^{2,1}_p(Q^1_T)$ and
\[{||\mathfrak{C} (|x|,\tau )||}_{W^{2,1}_p(Q^1_T)}\]
\begin{equation}\label{GH1**}
\le \iota_1'(\omega)\left(\|\overline{c}^1-\overline{c}^2 \|_{W^{2,1}_p(Q^1_T)}+\|u^1_1-u_1^2\|_{L^p(Q^*_T)}+||(\eta_1 ) ^{2}f\left(c_1,p_1,q_1\right)-(\eta_2 ) ^{2}f\left(c_2,p_2,q_2\right) ||_{L^p(Q^1_T)}\right)~a.s.,
\end{equation}
where $\iota_1'(\omega)$ is positive and  independent of $(u^1_1,u^1_2,\overline{c}^1,\overline{w}^1)$ and 
$(u^2_1,u^2_2,\overline{c}^2,\overline{w}^2)$. Thus, from \eqref{GH2} and \eqref{GH1**}, 
one can obtain that for almost every $\omega\in \mathbf{\Omega}$ 
\[ 
{||c_1 (|x|,\tau )-c_2(|x|,\tau )||}^p_{L^\infty(\mathfrak{B})}
\]
\begin{equation}
\label{GH1****}
\le\iota_2'(\omega)\left(\|\overline{c}^1-\overline{c}^2 \|^p_{L^\infty(Q^*_\tau)}
+\|u^1_1-u_1^2\|^p_{L^\infty(Q^*_\tau)}+\int_0^\tau||(\eta_1)^{2}
f\left(c_1,p_1,q_1\right)-(\eta_2 ) ^{2}f\left(c_2,p_2,q_2\right)||^p_{L^\infty([0,1])}ds\right),
\end{equation}
where $\iota_2'(\omega)$ is positive. Employing the inequalities presented in \eqref{khfgt} 
and \eqref{GH1****} and the Gronwall inequality, one can conclude that there exists positive 
$\lambda_0(\omega)$, which is independent of $(u^1_1,u^1_2,\overline{c}^1,\overline{w}^1)$ 
and $(u^2_1,u^2_2,\overline{c}^2,\overline{w}^2)$, such that
\[
{\|(c_1-c_2,w_1-w_2,p_1-p_2,q_1-q_2,a_1-a_2,\eta_1-\eta_2) \|}_{L^{\infty }(Q^*_\tau)}
\]
\[ 
\le \lambda_0(\omega){\|(u^1_1-u^2_1,u^1_2-u^2_2,\overline{c}^1-\overline{c}^2,\overline{w}^1-\overline{w}^2) \|}_{L^{\infty }(Q^*_\tau)}~a.s.
\]
On the other hand, applying the inequalities presented in \eqref{eqGkkp} and \eqref{khfgt1} and the Gronwall inequality, we deduce that there exists positive $\lambda(\omega)$, which is independent of $(u^1_1,u^1_2,\overline{c}^1,\overline{w}^1)$ and $(u^2_1,u^2_2,\overline{c}^2,\overline{w}^2)$, such that
\[{\|(c_1-c_2,w_1-w_2)\|^2_{L^2(Q^1_\tau)}+\|(p_1-p_2,q_1-q_2,a_1-a_2,\eta_1-\eta_2) \|}^2_{L^2(Q^*_\tau)}\]
\[
\le\lambda(\omega){\|(u^1_1-u^2_1,u^1_2-u^2_2,\overline{c}^1-\overline{c}^2,\overline{w}^1-\overline{w}^2) \|}^2_{L^2(Q^*_\tau)}~a.s.
\]
The proof is complete.
\end{proof}


\subsection{Adjoint equations}
\label{EUOP7777}

In this subsection,  we present the following stochastic adjoint equations (related to equations \eqref{YAS}--\eqref{HM345}) 
corresponding to the control  $(u_1^*,u_2^*,{\overline{c}}^*,{\overline{w}}^*)$ which are instrumental in presenting the necessary conditions and proving the existence and  uniqueness of optimal control variables for almost every $\omega\in\mathbf{\Omega}$. The following adjoint equation is related to the solution $c$ of \eqref{YAS}--\eqref{opb1},
\[d z_c=-D_{1}\dfrac{1}{\rho^2}\dfrac{\partial }{\partial\rho }\left({\rho }^{2}\dfrac{\partial z_c}{\partial\rho }\right){d\tau }+u^*\left(1, \tau \right)\rho \dfrac{\partial z_c}{\partial\rho }d\tau -z_c\displaystyle\sum_{i=0}^mh^1_id B_i+\smallunderbrace{z_c\sum_{i=0}^m(h_i^1)^2}_{n^c}d\tau\]
\begin{equation}\label{YAS*1e1}
+3u^*\left(1, \tau \right) z_cd\tau -f^*\left(z_\eta,z_c,z_w,z_p,z_q,z_a,\eta^*,c^*,w^*,p^*,q^*,a^*\right)d\tau,~{\rm \ 0}<\rho <1,\ \tau {\rm >}0,
\end{equation}
\begin{equation}
\dfrac{\partial z_c}{\partial\rho }\left(0,\tau \right){\rm =0, }~z_c\left({\rm 1,}\tau \right){\rm =}0, ~~\ \tau {\rm >}0,
\end{equation} 
\begin{equation}
z_c\left(\rho ,T\right){\rm =}0,~{\rm  \ 0}\le \rho \le {\rm 1,}
\end{equation}
where  $(c^*,w^*,p^*,q^*,a^*,\eta^*,u^*,v^*)$ is the solution of  (\ref{YAS})--(\ref{HM345}) corresponding to the control  $(u_1^*,u_2^*,{\overline{c}}^*,{\overline{w}}^*)$ and
\[f^*\left(z_\eta,z_c,z_w,z_p,z_q,z_a,\eta^*,c^*,w^*,p^*,q^*,a^*\right)=-\dfrac{\partial\overbrace{{\eta}^2f\left(c,p,q\right)}^{f_1(\eta,c,p,q)}}{\partial c}\Big|_{(\eta,c,p,q)=(\eta^*,c^*,p^*,q^*)}z_c\]
\[+\dfrac{\partial G_1}{\partial c}(\smallunderbrace{\eta^*,c^*,w^*,p^*,q^*,a^*}_{X^*})z_p+\dfrac{\partial G_2}{\partial c}(X^*)z_q+\dfrac{\partial G_3}{\partial c}(X^*)z_a +\dfrac{\partial {\eta^*}^3 h }{\partial c}({Y^*})z_\eta\]
\[+\dfrac{\partial {\eta^*}^2 h }{\partial c}(\smallunderbrace{c^*,w^*,p^*,q^*,a^*}_{Y^*})\Big(\smallunderbrace{\int_0^1\rho^3(\dfrac{\partial c^*}{\partial \rho}z_c+\dfrac{\partial w^*}{\partial \rho}z_w)d\rho}_{F_1^*}\]
\[+\smallunderbrace{\int_0^\rho (\dfrac{\partial 
p^*}{\partial s}z_p+\dfrac{\partial q^*}{\partial s}z_q+\dfrac{\partial a^*}{\partial s}z_a)ds-\int_0^1(1-\rho^3)(\dfrac{\partial p^*}{\partial \rho}z_p+\dfrac{\partial q^*}{\partial \rho}z_q+\dfrac{\partial a^*}{\partial \rho}z_a)d\rho}_{A_1^*}\Big),\]
 and the adjoint equation  related to $w$, the solution of \eqref{YAS676767}--\eqref{YAS1}, is
\[d z_w=-D_{2}\dfrac{1}{{\rho }^{2}}\dfrac{\partial }{\partial\rho }\left({\rho }^{2}\dfrac{\partial z_w}{\partial\rho }\right)d\tau+u^*\left(1, \tau \right)\rho \dfrac{\partial z_w}{\partial\rho }d\tau-z_w\displaystyle\sum_{i=0}^mh^2_id B_i+\smallunderbrace{z_w\sum_{i=0}^m(h_i^2)^2}_{n^w}d\tau\]
\begin{equation}\label{YIH}
+3u^*\left(1, \tau \right) z_wd\tau-g^*\left(z_\eta,z_c,z_w,z_p,z_q,z_a,\eta^*,c^*,w^*,p^*,q^*,a^*\right)d\tau, ~{\rm \ 0}<\rho <1,\ \tau {\rm >}0,
\end{equation}
\begin{equation}
\dfrac{\partial z_w}{\partial\rho }\left(0,\tau \right){\rm =0, }~z_w\left({\rm 1,}\tau \right){\rm =}0, ~~\ \tau {\rm >}0,
\end{equation}
\begin{equation}\label{YAS1*e}
z_w\left(\rho ,T\right){\rm =}0,~{\rm \ \ 0}\le \rho \le {\rm 1,}
\end{equation}
where
\[g^*\left(z_\eta,z_c,z_w,z_p,z_q,z_a,\eta^*,c^*,w^*,p^*,q^*,a^*\right)\]
\[=-\dfrac{\partial\overbrace{{\eta}^2g\left(w,p,q\right)}^{g_1(\eta,w,p,q)}}{\partial w}\Big|_{(\eta,w,p,q)=(\eta^*,w^*,p^*,q^*)}z_w+\dfrac{\partial G_1}{\partial w}(X^*)z_p\]
\[+\dfrac{\partial G_2}{\partial w}(X^*)z_q+\dfrac{\partial G_3}{\partial w}(X^*)z_a+ \dfrac{\partial {\eta^*}^3 h }{\partial w}(Y^*)z_\eta+\dfrac{\partial {\eta^*}^2 h }{\partial w}(Y^*)\Big(F_1^*+A_1^*\Big),\]
 and the  adjoint equation   related to $p$ is
\begin{eqnarray}\label{pqd*e}
&d z_p{\rm +}v^*\dfrac{\partial z_p}{\partial\rho }d\tau+\dfrac{\partial \rho^2 v^*}{\rho^2\partial\rho }z_pd\tau{\rm =}-z_p\displaystyle\sum_{i=0}^mh^3_id B_i+\smallunderbrace{z_p\sum_{i=0}^m(h_i^3)^2}_{n^p}d\tau&\nonumber\\
&-\Big(-\dfrac{\partial f_1}{\partial p}(\eta^*,c^*,p^*,q^*)z_c-\dfrac{\partial g_1}{\partial p}(\eta^*,w^*,p^*,q^*)z_w+\dfrac{\partial G_1}{\partial p}(X^*)z_p&\nonumber\\
&+\dfrac{\partial G_2}{\partial p}(X^*)z_q+\dfrac{\partial G_3}{\partial p}(X^*)z_a+\dfrac{\partial {\eta^*}^3 h }{\partial p}(Y^*)z_\eta+\dfrac{\partial {\eta^*}^2 h }{\partial p}(Y^*)\Big(F_1^*+A_1^*\Big)+\lambda_1p^*\Big)d\tau,&\\
&{\rm 0}\le \rho \le {\rm 1,\ }\tau {\rm >}0,& \nonumber\\
&z_p\left(\rho,T\right)=0,~~~~{\rm 0}\le \rho \le {\rm 1,\ }&
\end{eqnarray} 
 and the following adjoint equation  is related to $q$ 
\begin{eqnarray}
&d z_q{\rm +}v^*\dfrac{\partial z_q}{\partial\rho }d\tau+\dfrac{\partial \rho^2 v^*}{\rho^2\partial\rho }z_qd\tau=-z_q\displaystyle\sum_{i=0}^mh^4_id B_i+\smallunderbrace{z_q\sum_{i=0}^m(h_i^4)^2}_{n^q}d\tau &\nonumber\\
&-\Big(-\dfrac{\partial f_1}{\partial q}(\eta^*,c^*,p^*,q^*)z_c-\dfrac{\partial g_1}{\partial q}(\eta^*,w^*,p^*,q^*)z_w+\dfrac{\partial G_1}{\partial q}(X^*)z_p&\nonumber\\
&+\dfrac{\partial G_2}{\partial q}(X^*)z_q+\dfrac{\partial G_3}{\partial q}(X^*)z_a+\dfrac{\partial {\eta^*}^3 h }{\partial q}(Y^*)z_\eta+\dfrac{\partial {\eta^*}^2 h }{\partial q}(Y^*)\Big(F_1^*+A_1^*\Big)+\lambda_2q^*\Big)d\tau,&\\
&{\rm 0}\le \rho \le {\rm 1,\ }\tau {\rm >}0,& \nonumber\\
&z_q\left(\rho,T\right)=0,~~~~{\rm 0}\le \rho \le {\rm 1,\ }&
\end{eqnarray} 
 and the  adjoint equation   related to $a$ is
\begin{eqnarray} 
&{d z_a}{\rm +}v^*\dfrac{\partial z_a}{\partial\rho }d\tau+\dfrac{\partial \rho^2 v^*}{\rho^2\partial\rho }z_a d\tau{\rm =}-z_a\displaystyle\sum_{i=0}^mh^5_id B_i+\smallunderbrace{z_a\sum_{i=0}^m(h_i^5)^2}_{n^a}d\tau&\nonumber\\
&-\left(\dfrac{\partial G_1}{\partial a}(X^*)z_p+\dfrac{\partial G_2}{\partial a}(X^*)z_q+\dfrac{\partial G_3}{\partial a}(X^*)z_a+\dfrac{\partial {\eta^*}^3 h}{\partial a}(Y^*) z_\eta+\dfrac{\partial {\eta^*}^2 h}{\partial a}(Y^*) \Big(F_1^*+A_1^*\Big)\right)d\tau,&\\  
&{\rm 0}\le \rho \le {\rm 1,\ }\tau {\rm >}0,& \nonumber
\end{eqnarray}
\begin{equation} 
 z_a\left(\rho,T\right)=0,~~~~{\rm 0}\le \rho \le {\rm 1,\ }
\end{equation}
and the adjoint equation related to $\eta$, the solution of \eqref{HM345},  is as follows
\begin{equation}\label{HM345**ee} 
{dz_\eta (\tau)}=-z_\eta dB_*+\smallunderbrace{z_\eta\sum_{i=0}^m(r_i)^2}_{n^\eta}d\tau -h^*\left(z_\eta,z_c,z_w,z_p,z_q,z_a,\eta^*,c^*,w^*,p^*,q^*,a^*\right){d\tau },~~\tau>0,~~z_\eta(T)=0, 
\end{equation}
where 
\[h^*(z_\eta,z_c,z_w,z_p,z_q,z_a,\eta^*,c^*,w^*,p^*,q^*,a^*)\]
\[=u^*(1,\tau)z_\eta+{\eta^*}^2\int_0^12\rho^2h(Y^*)d\rho z_\eta+{\eta^*}\int_0^12\rho^2 h(Y^*)\Big(F_1^*+A_1^*\Big)d\rho\]
\[+\int_0^1\rho^2\Big(-\dfrac{\partial f_1}{\partial \eta}(\eta^*,c^*,p^*,q^*)z_c-\dfrac{\partial g_1}{\partial \eta}(\eta^*,w^*,p^*,q^*)z_w+\dfrac{\partial G_1}{\partial \eta}(X^*)z_p+\dfrac{\partial G_2}{\partial \eta}(X^*)z_q+\dfrac{\partial G_3}{\partial \eta}(X^*)z_a\Big)d\rho.\] 

\noindent The next theorem  shows that there exists a unique solution  for the problem (\ref{YAS*1e1})--(\ref{HM345**ee}). 

\begin{thm}\label{AMHZ}
Let Assumption~\ref{assumpt} be satisfied. Then for  almost every $\omega\in \mathbf{\Omega}$  the adjoint system  \textup{(\ref{YAS*1e1})--(\ref{HM345**ee})} has a unique solution $\left(z_c,z_w,z_p,z_q,z_a,z_\eta\right)$ and for every $T>0$  we have  
 \[z_c(|x|,\tau),z_w(|x|,\tau)\in C\left(Q^1_T\right),~z_\eta\left(\tau\right)\in C\left[0,T\right],~z_p,z_q,z_a\in C\left([0,1]\times[0,T]\right),\]
 where $Q^1_T:=\Big\{(x,\tau)\in \mathbb{R}^{3}\times \mathbb{R}: |x|< 1,~0<\tau\leq T\Big\}.$ Moreover, $z_c$, $z_w$, $z_p$, $z_q$, $z_a$, and $z_\eta$ are  $\mathcal{F}(t)-$adapted.
\end{thm}

\begin{proof}
Similar to the proof of  Lemmas \ref{Lemma 1ito} and \ref{Lemma2ito}, we can show that
\begin{equation}\label{adjchv}
\begin{cases}
z_c=e^{\sum_{i=0}^mh^1_i(\rho)B_i}z_{\mathfrak{C}},~~z_w=e^{\sum_{i=0}^mh^2_i(\rho)B_i}z_{\mathfrak{W}},\\
z_p=e^{\sum_{i=0}^mh^3_i(\rho)B_i}z_{\mathfrak{P}},~~z_q=e^{\sum_{i=0}^mh^4_i(\rho)B_i}z_{\mathfrak{Q}},~~z_d=e^{\sum_{i=0}^mh^5_i(\rho)B_i}z_{\mathfrak{D}}.
\end{cases}
\end{equation}
Then, using the change of variable $t = T -\tau$ for the equations (which are final-boundary value problems) obtained for
\[z_{\mathfrak{C}},~~z_{\mathfrak{W}},~~z_{\mathfrak{P}},~~z_{\mathfrak{Q}},~~z_{\mathfrak{D}},\]
 these equations become initial-boundary value problems. Therefore,  employing Lemmas \ref{Lemma 1ito} and \ref{Lemma2ito} and Theorem \ref{mainthm1},  in a way similar  to the proof of the main heorem of \cite{2.}, we can prove this theorem.
\end{proof}

\begin{thm}
\label{???}
Let  Assumption~\ref{assumpt} and initial condition   \textup{(\ref{initial1})} be satisfied and $k>5$. Also, assume that $(z_c^1,z_w^1,z_p^1,z_q^1,z_a^1,z_\eta^1)$ and $(z_c^2,z_w^2,z_p^2,z_q^2,z_a^2,z_\eta^2)$ are the  solutions of the adjoint system \textup{(\ref{YAS*1e1})--(\ref{HM345**ee})}  corresponding to $(u^1_1,u^1_2,\overline{c}^1,\overline{w}^1)\in \mathcal{C}_{ad}\times\mathcal{W}_{ad}\times\mathcal{C}^1_{ad}\times\mathcal{W}^1_{ad}$ and $(u^2_1,u^2_2,\overline{c}^2,\overline{w}^2)\in \mathcal{C}_{ad}\times\mathcal{W}_{ad}\times\mathcal{C}^1_{ad}\times\mathcal{W}^1_{ad}$,  respectively, where $\mathcal{C}_{ad}\times\mathcal{W}_{ad}\times\mathcal{C}^1_{ad}\times\mathcal{W}^1_{ad}$ is the admissible control set for  SOCP. Then, for almost every $\omega\in\mathbf{\Omega}, $ there exists positive   $\lambda^*_2(\omega)$, which is independent of $(u^1_1,u^1_2,\overline{c}^1,\overline{w}^1)$ and $(u^2_1,u^2_2,\overline{c}^2,\overline{w}^2)$, such that 
\[{\left\|(z_c^1-z_c^2,z_w^1-z_w^2,z_p^1-z_p^2,z_q^1-z_q^2,z_a^1-z_a^2,z_\eta^1-z_\eta^2) \right\|}^k_{L^{\infty }([0,1])}\]
\[ \le\lambda_2^*(\omega)\int_\tau^T\int_0^\xi{\|(u^1_1-u^2_1,u^1_2-u^2_2,\overline{c}^1-\overline{c}^2,\overline{w}^1-\overline{w}^2) \|}^k_{L^{\infty }([0,1])}dsd\xi.\]
\end{thm}

\begin{proof}
Using \eqref{adjchv} and the change of variable $t=T-\tau$, the problem (\ref{YAS*1e1})--(\ref{HM345**ee}) becomes an initial-boundary value problem similar to the problem \textup{(\ref{YAS0o0o0o})--(\ref{HM3450o0o0o})}. So, in a way similar to the proof of Theorem \ref{??} (see equation \eqref{GH1****}) and  equation \eqref{iton7},using equations \eqref{HM345} and \eqref{HM345**ee}, and applying Assumption~\ref{assumpt}, we show that the following inequality holds
\[ {||z_c^{1,T} -z_c^{2,T}\|}^k_{L^{\infty }([0,1])}+{||z_w^{1,T} -z_w^{2,T}||}^k_{L^{\infty }([0,1])}\]
\[+\|(z_p^{1,T}-z_p^{2,T},z_q^{1,T}-z_q^{2,T},z_a^{1,T}-z_a^{2,T},z_\eta^{1,T}-z_\eta^{2,T}) \|^k_{L^{\infty }([0,1])}\]
\[\le\kappa_1(\omega)\Big(\int_0^t\|(c^T_1-c^T_2,w^T_1-w^T_2,p^T_1-p^T_2,q^T_1-q^T_2,a^T_1-a^T_2,\eta^T_1-\eta^T_2) \|^k_{L^{\infty }([0,1])}ds\]
 \begin{equation}
 \label{iniquality-adjoint1}
 +\int_0^t{\|(z_c^{1,T}-z_c^{2,T},z_w^{1,T}-z_w^{2,T},z_p^{1,T}-z_p^{2,T},z_q^{1,T}-z_q^{2,T},z_a^{1,T}-z_a^{2,T},z_\eta^{1,T}-z_\eta^{2,T}) \|}^k_{L^{\infty }([0,1])}ds\Big)~a.s.,
\end{equation}
where for each $A\in\{z_c,z_w,z_p,z_q,z_a,z_\eta,c,w,p,q,a,\eta\}$
\[A^T(\cdot,t)=A(\cdot,T-t)=A(\cdot,\tau),\]
 and $\kappa_1(\omega)$ is positive.
Using the Gronwall inequality, the inequality \eqref{iniquality-adjoint1} results in
\[
\|(z_c^{1,T}-z_c^{2,T},z_w^{1,T}-z_w^{2,T},z_p^{1,T}-z_p^{2,T},z_q^{1,T}-z_q^{2,T},z_a^{1,T}-z_a^{2,T},z_\eta^{1,T}-z_\eta^{2,T}) \|^k_{L^{\infty }([0,1])}
\]
\begin{equation*}
\le \kappa_2(\omega)\int_0^t\|(c^T_1-c^T_2,w^T_1-w^T_2,p^T_1-p^T_2,q^T_1-q^T_2,a^T_1-a^T_2,\eta^T_1-\eta^T_2) \|^k_{L^{\infty }([0,1])}ds~a.s.,
\end{equation*}
where $\kappa_2(\omega)$ is positive. Therefore,
\[\|(z_c^{1}-z_c^{2},z_w^{1}-z_w^{2},z_p^{1}-z_p^{2},z_q^{1}-z_q^{2},z_a^{1}-z_a^{2},z_\eta^{1}-z_\eta^{2}) \|^k_{L^{\infty }([0,1])}\]
 \begin{equation}\label{iniquality-adjoint3}
\le \kappa_2(\omega)\int_\tau^T\|(c_1-c_2,w_1-w_2,p_1-p_2,q_1-q_2,a_1-a_2,\eta_1-\eta_2) \|^k_{L^{\infty }([0,1])}ds~a.s.
\end{equation}
Then, by the use of equations \eqref{iton7}, \eqref{GH1****} and \eqref{iniquality-adjoint3}, and Gronwall inequality, we conclude that
\[\|(z_c^{1}-z_c^{2},z_w^{1}-z_w^{2},z_p^{1}-z_p^{2},z_q^{1}-z_q^{2},z_a^{1}-z_a^{2},z_\eta^{1}-z_\eta^{2}) \|^k_{L^{\infty }([0,1])}\]
 \begin{equation}\label{iniquality-adjoint4}
\le\lambda_2^*(\omega)\int_\tau^T\int_0^\xi\|(u^1_1-u^2_1,u^1_2-u^2_2,\overline{c}^1-\overline{c}^2,\overline{w}^1-\overline{w}^2) \|^k_{L^{\infty }}([0,1]) dsd\xi~a.s.,
\end{equation}
where $\lambda_2^*(\omega)$ is positive. 
\end{proof}


\subsection{Existence and uniqueness of stochastic optimal control}
\label{EUOP1111}

In this subsection, we give the necessary conditions together with the existence and  uniqueness of stochastic optimal control. 
   In the next theorem,  the necessary conditions for stochastic optimal control variables of  SOCP are presented.
\begin{thm}\label{ncon} \textbf{\textup{(Necessary Conditions)}}
Let $u_1^*$, $u_2^*$, ${\overline{c}}^*$ and ${\overline{w}}^*$ be the optimal control variables of SOCP. Then, there exists $T>0$ such that for almost every $\omega\in\mathbf{\Omega}$ we have
\begin{equation}\label{YA1**e5e}
u_1^*=\mathcal{F}_1(-z_c+r_1^*),~u_2^*=\mathcal{F}_2(-z_w+r_2^*),
\end{equation}
\begin{equation}\label{YA1**e}
{\overline{c}}^*(\tau)=\mathcal{F}_3\Big(D_1\dfrac{\partial z_c}{\partial\rho}(1,\tau)+r_3^*\Big),~{\overline{w}}^*(\tau)=\mathcal{F}_4\Big(D_2\dfrac{\partial z_w}{\partial\rho}(1,\tau)+r_4^*\Big),
\end{equation}
where  $z_c$ and $z_w$ are the adjoint states corresponding to $(u^*_1,u^*_2,{\overline{c}}^*,{\overline{w}}^*)$ and 
\begin{equation}\label{YA2**}
\mathcal{F}_i(x)=\begin{cases}
l_i^*,~~x<l_i^*,\\
x,~~l_i^*\leq x\leq l^{**}_i,\\
l^{**}_i,~~l^{**}_i< x ,\\
\end{cases}
i=1,2,3,4,
\end{equation}
where $l_i^*$ and $l_i^{**}$ are defined in \textup{(\ref{YAK*15555})--(\ref{YAK*255555})}. 
\end{thm}

\begin{proof}
See Appendix.
\end{proof}

\begin{remark}
To prove the existence and uniqueness of the optimal control variables, the first step is proving Lemma \ref{Lemma 1ito}, which is used to study the parabolic equations \eqref{YAS} and \eqref{YAS676767}. In Lemma \ref{Lemma 1ito}, we need the continuity of function $f(\rho ,\tau)$. So the optimal control variables are assumed to be continuous functions to be sure the problems \eqref{YAS} and \eqref{YAS676767} satisfy the assumptions presented in Lemma \ref{Lemma 1ito}.  
\end{remark}
\begin{remark}
For almost every $\omega\in\mathbf{\Omega}$, using the $t$-Anisotropic Embedding Theorem (Theorem 1.4.1 in \cite{21.}) and Lemma 3.3 in \cite{22.}, we conclude that  $z_c$, $z_w$, $\dfrac{\partial z_c}{\partial\rho}(1,\tau)$ and $\dfrac{\partial z_c}{\partial\rho}(1,\tau)$, are  continuous. Since $r_i^*$,  $l_i^*$, $l_i^{**}$ and $\mathcal{F}_i$ $(i=1,2,3,4)$ are continuous functions, therefore $u_1^*$, $u_2^*$, $\overline{c}^*$ and $\overline{w}^*$  are continuous.
\end{remark}
\noindent The following theorem shows the existence and  uniqueness of stochastic optimal control variables for  SOCP.
\begin{thm}\label{YKHFKH}\textbf{\textup{(Existence and Uniqueness)}}
Let Assumption~\ref{assumpt} be satisfied. Then, there exists $T>0$ such that $(u_1^*,u_2^*,{\overline{c}}^*,{\overline{w}}^*)$ is a unique optimal control  for SOCP which satisfies \textup{(\ref{YA1**e5e})--(\ref{YA1**e})} and
\[J(u_1^*,u_2^*,{\overline{c}}^*,{\overline{w}}^*)=\min \Big\{J(u_1,u_2,{\overline{c}},{\overline{w}}):(u_1,u_2,\overline{c},\overline{w})\in \mathcal{C}_{ad}\times\mathcal{W}_{ad}\times\mathcal{C}^1_{ad}\times\mathcal{W}^1_{ad}\Big\}.\]
Moreover  $u_1^*$, $u_2^*$, ${\overline{c}}^*$, ${\overline{w}}^*$ are $\mathcal{F}(\tau)-$adapted.
\end{thm}

\begin{proof}
See Appendix.
\end{proof}

\begin{remark}
\label{remcoop}
From Theorem \ref{YKHFKH}, $u_1^*$, $u_2^*$, ${\overline{c}}^*$, ${\overline{w}}^*$ are $\mathcal{F}(\tau)-$adapted and
\[J(u_1^*,u_2^*,{\overline{c}}^*,{\overline{w}}^*)\leq J(u_1,u_2,{\overline{c}},{\overline{w}}),~~\forall(u_1,u_2,{\overline{c}},{\overline{w}})\in \mathcal{C}_{ad}\times\mathcal{W}_{ad}\times\mathcal{C}^1_{ad}\times\mathcal{W}^1_{ad}~~a.s.\]
Therefore,
\[E[J(u_1^*,u_2^*,{\overline{c}}^*,{\overline{w}}^*)]\leq E[J(u_1,u_2,{\overline{c}},{\overline{w}})],~~\forall(u_1,u_2,\overline{c},\overline{w})\in \mathcal{C}_{ad}\times\mathcal{W}_{ad}\times\mathcal{C}^1_{ad}\times\mathcal{W}^1_{ad}.\]
\end{remark}

\begin{thm}\label{concorrelatednoise} 
Let $u_1^*$, $u_2^*$, ${\overline{c}}^*$ and ${\overline{w}}^*$ be the optimal control variables of SOCP in which the Brownian motions are correlated with
\[\operatorname{Corr}[B_i(\tau),B_j(\tau)]=\rho_{ij},~~\rho_{ii}=1,~~i,j\in\{0,1,\ldots,m\}.\]
Then, for almost every $\omega\in\mathbf{\Omega}$, we have
\begin{equation*}
u_1^*=\mathcal{F}_1(-z_c+r_1^*),~u_2^*=\mathcal{F}_2(-z_w+r_2^*),
\end{equation*}
\begin{equation*}
{\overline{c}}^*(\tau)=\mathcal{F}_3\Big(D_1\dfrac{\partial z_c}{\partial\rho}(1,\tau)+r_3^*\Big),~{\overline{w}}^*(\tau)=\mathcal{F}_4\Big(D_2\dfrac{\partial z_w}{\partial\rho}(1,\tau)+r_4^*\Big),
\end{equation*}
where the adjoint states $z_c$, $z_w$, $z_p$, $z_q$, $z_a$ and $z_{\eta}$ are the solutions of equations \eqref{YAS*1e1}--\eqref{HM345**ee} in which 
\begin{equation}\label{ncc}
n^c=z_c\sum_{i=0}^m\sum_{j=0}^mh^1_ih^1_j\rho_{ij},~~n^w=z_w\sum_{i=0}^m\sum_{j=0}^mh^2_ih^2_j\rho_{ij},
\end{equation}
\begin{equation}
n^p=z_p\sum_{i=0}^m\sum_{j=0}^mh^3_ih^3_j\rho_{ij},~~n^q=z_q\sum_{i=0}^m\sum_{j=0}^mh^4_ih^4_j\rho_{ij},~~n^a=z_a\sum_{i=0}^m\sum_{j=0}^mh^5_ih^5_j\rho_{ij},
\end{equation}
\begin{equation}\label{npp}
n^\eta=z_\eta\sum_{i=0}^m\sum_{j=0}^mr_ir_j\rho_{ij},
\end{equation}
and 
\begin{equation*}
\mathcal{F}_i(x)=\begin{cases}
l_i^*,~~x<l_i^*,\\
x,~~l_i^*\leq x\leq l^{**}_i,\\
l^{**}_i,~~l^{**}_i< x ,\\
\end{cases}
i=1,2,3,4,
\end{equation*}
 where $l_i^*$ and $l_i^{**}$ are defined in \textup{(\ref{YAK*15555})--(\ref{YAK*255555})}. 
\end{thm}
\begin{proof}
This theorem can be proved similar to the proof of Theorem \ref{ncon}, but in \eqref{correneed1}--\eqref{correneed3} the adjoint system  \eqref{YAS*1e1}--\eqref{HM345**ee} (see Appendix) is considered for $n^c$, $n^w$, $n^p$, $n^q$, $n^a$ and $n^\eta$ defined in \eqref{ncc}--\eqref{npp}.
The adjoint system for SOCP with correlated Brownian motions is changed, because from Figure 4.3 in \cite{brnotioncal}, for correlated Brownian motion with 
\[\operatorname{Corr}[B_i(\tau),B_j(\tau)]=\rho_{ij},~~\rho_{ii}=1,~~i,j\in\{0,1,\ldots,m\},\]
we have $dB_idB_j=\rho_{ij} d\tau$.
\end{proof}


\section{Existence and uniqueness of deterministic optimal control}

 In this section, we study the following  optimal control problem (OCP)  in which we control the concentrations of drug and nutrient using control variables $u_1$, $u_2$, $\overline{c}$ and $\overline{w}$, which are deterministic control variables, and  we minimize
\begin{equation*}
\mathbb{J}(u_1,u_2,\overline{c},\overline{w}):=E\Big[\int_0^1\int_0^T\rho^2(\lambda_1p^2+\lambda_2q^2+(u_1-r_1^*)^2+(u_2-r_2^*)^2)d\tau d\rho+\int_0^T(\overline{c}-r_3^*)^2+(\overline{w}-r_4^*)^2d\tau\Big],
\end{equation*}
\[~~~~(u_1,u_2,\overline{c},\overline{w})\in \mathbb{C}_{ad}\times\mathbb{W}_{ad}\times\mathbb{C}^1_{ad}\times\mathbb{W}^1_{ad},\]
 such that
$\mathbb{J}(u^*_1,u^*_2,\overline{c}^*,\overline{w}^*)=\min \Big\{\mathbb{J}(u_1,u_2,\overline{c},\overline{w}):(u_1,u_2,\overline{c},\overline{w})\in \mathbb{C}_{ad}\times\mathbb{W}_{ad}\times\mathbb{C}^1_{ad}\times\mathbb{W}^1_{ad}\Big\},$
where 
\begin{equation}\label{ocp1}
\mathbb{C}_{ad}:=\Big\{ v\in C([0, 1]\times [0, T]): l^{*}_1(\rho,\tau)\leq v(\rho,\tau)\leq l^{**}_1(\rho,\tau)\Big\},
\end{equation}
\begin{equation}
 \mathbb{W}_{ad}:=\Big\{ v\in C([0, 1]\times [0, T]): l^{*}_2(\rho,\tau) \leq v(\rho,\tau)\leq l^{**}_2(\rho,\tau)\Big\},
\end{equation}
\begin{equation}
\mathbb{C}^1_{ad}:=\Big\{ v\in C [0, T]:~\dfrac{\partial v}{\partial \tau}\in L^p[0,T],~ p>5, ~v(0)=0,~ l^{*}_3(\tau)\leq v(\tau)\leq l^{**}_3(\tau)\Big\},
\end{equation}
\begin{equation}\label{ocp2}
\mathbb{W}^1_{ad}:=\Big\{ v\in C [0, T]:~\dfrac{\partial v}{\partial\tau}\in L^p[0,T],~ p>5,~v(0)=0,~ l^{*}_4(\tau)\leq v(\tau)\leq l^{**}_4(\tau)\Big\},
\end{equation} 
\[r^*_1,r^*_2,l_1^*,l_1^{**},l_2^*,l_2^{**}\in C^{2+\alpha,1+\alpha/2}([0, 1]\times [0, T]) ,\]
\[r_3^*,r_4^*,l_3^{*},l_3^{**},l_4^{*},l_4^{**}\in C^{1+\alpha/2} [0, T],~l_3^{*}(0)=l_3^{**}(0)=l_4^{*}(0)=l_4^{**}(0)=0,~\]
subject to \eqref{YAS}--\eqref{HM345}.\\
In the following theorem we show that this problem has a unique optimal control. We also present the explicit forms of control variables.

\begin{thm}
\label{dexun}
Let Assumption~\ref{assumpt} be satisfied. Then, there exists $T>0$ such that the OCP has a unique optimal control $(u_1^*,u_2^*,{\overline{c}}^*,{\overline{w}}^*)$, which is as follows
\begin{equation*}
u_1^*=\mathcal{F}_1(E[-z_c+r_1^*]),~u_2^*=\mathcal{F}_2(E[-z_w+r_2^*]),
\end{equation*}
\begin{equation*}
{\overline{c}}^*(\tau)=\mathcal{F}_3\Big(E[D_1\dfrac{\partial z_c}{\partial\rho}(1,\tau)+r_3^*]\Big),~{\overline{w}}^*(\tau)=\mathcal{F}_4\Big(E[D_2\dfrac{\partial z_w}{\partial\rho}(1,\tau)+r_4^*]\Big),
\end{equation*}
where  $z_c$ and $z_w$ are the adjoint states corresponding to $(u^*_1,u^*_2,{\overline{c}}^*,{\overline{w}}^*)$ and 
\begin{equation*}
\mathcal{F}_i(x)=\begin{cases}
l_i^*,~~x<l_i^*,\\
x,~~l_i^*\leq x\leq l^{**}_i,\\
l^{**}_i,~~l^{**}_i< x ,\\
\end{cases}
i=1,2,3,4,
\end{equation*}
 where $l_i^*$ and $l_i^{**}$ are defined in \textup{(\ref{ocp1})--(\ref{ocp2})}. 
\end{thm}

\begin{proof}
Similar to the Proof of Theorem \ref{ncon}, we can arrive at
\[E\Bigg[\int_0^1\int_0^T\rho^2\Big(( z_c+u_1^*-r_1^*)u_1^0+( z_w+u_2^*-r_2^*)u_2^0\Big)d\rho d\tau+\int_0^T\left(\overline{c}^*-r_3^*- D_1\dfrac{\partial z_c}{\partial\rho}(1,\tau)\right){\overline{c}}_0(\tau)d\tau\]
\begin{equation*}
+\int_0^T\left(\overline{w}^*-r_4^*- D_2\dfrac{\partial z_w}{\partial\rho}(1,\tau)\right){\overline{w}}_0(\tau)d\tau\Bigg] \geq 0.
\end{equation*}
Since $u_1,u_2,\overline{c},\overline{w}$ are deterministic, $u_1^0,u^0_2,\overline{c}_0,\overline{w}_0$ are deterministic. Therefore
\[
\int_0^1\int_0^T\rho^2\Big(( E[z_c-r_1^*]+u_1^*)u_1^0+(E[ z_w-r_2^*]+u_2^*)u_2^0\Big)d\rho d\tau+\int_0^T\left(\overline{c}^*-E[r_3^*+ D_1\dfrac{\partial z_c}{\partial\rho}(1,\tau)]\right){\overline{c}}_0(\tau)d\tau
\]
\begin{equation*}
+\int_0^T\left(\overline{w}^*-E[r_4^*+ D_2\dfrac{\partial z_w}{\partial\rho}(1,\tau)]\right){\overline{w}}_0(\tau)d\tau \geq 0.
\end{equation*}
Therefore, using tangent-normal cone techniques (see subsection 5.3 in \cite{25.}) and  (\ref{ocp1})--(\ref{ocp2}), one can conclude that
\[
u_1^*=\mathcal{F}_1(E[-z_c+r_1^*]),~u_2^*=\mathcal{F}_2(E[-z_w+r_2^*]),~
{\overline{c}}^*=\mathcal{F}_3\left(E[D_1\dfrac{\partial z_c}{\partial\rho}(1,\tau)+r_3^*]\right),
\]
\[
~{\overline{w}}^*=\mathcal{F}_4\left(E[D_2\dfrac{\partial z_w}{\partial\rho}(1,\tau)+r_4^*]\right).
\]
Then, in a way similar to the proof of Theorem \ref{YKHFKH}, one can conclude that the obtained optimal control is unique.
\end{proof}


\section{Numerical experiments}

In this section, we have solved a stochastic optimal control problem to illustrate the effects of random terms and optimal control variables on the evolution of tumour cells.\\

\begin{example}
In this example, we have solved SOCP for  
\[r_1^*=-4,~r_2^*=5,~r_3^*=-4,~r^*_4=5,~l_1^*=-5,l_1^{**}=0,~l_2^*=0,l_2^{**}=10,\]
\[~l_3^*=-5,l_3^{**}=0,~l_4^*=0,l_4^{**}=10,\]
\[D_1=D_2=\dfrac{1}{10},\]
\[K_B(c)=15c,~K_Q(c)=15-c,K_A(c)=13-c,~K_P(c)=c,~K_D(c)=13-c,~\]
\[G_1(w)=7w,~G_2(w)=3w,\]
\[c_0(\rho)=12,~w_0(\rho)=14, p_0(\rho)=2\rho+e^{\rho^2}+1,~q_0(\rho)=2\rho^2+2,~\eta_0=\dfrac{1}{10}.\]
The problem is solved using a combination of euler-maruyama and collocation methods. In this method,
the problem on the space domain is discretized using Legendre--Gauss--Lobatto nodes, which results in
a SDE. Then, the SDE is solved using the Euler--Maruyama method. In Figure~\ref{fig2a}, for $N=19$ and $r_i=h_i^j=0.5 ~(i\in\{0,1\},j\in\{1,\ldots,5\})$, the effect of stochastic terms on the dynamic  and density of alive tumour cells is illustrated. 
In Figures~\ref{fig2b} and \ref{fig2c}, for $N=13$ and $r_i=h_i^j=0.1 ~(i\in\{0,1\},j\in\{1,\ldots,5\})$, the effects of optimal control variables on the density of tumour cells are illustrated. It is shown that the density of alive cells and proliferative cells decreases under the effects of optimal control. In Figure~\ref{fig2d}, for $N=13$ and $r_i=h_i^j=0.1 ~(i\in\{0,1\},j\in\{1,\ldots,5\})$, sample paths of 
\[\dfrac{\eta^C(\tau)-\eta(\tau)}{\eta_0}\]
are shown, where
$\eta$ and $\eta^C$ are the radii of tumour without control and under the effect of control, respectively. It is illustrated that the optimal control variables result in a relative decrease of tumour radius (Figure~\ref{fig2d}).
\begin{figure}[!t]
\begin{center}
\hspace{-1cm}\includegraphics[width=18cm]{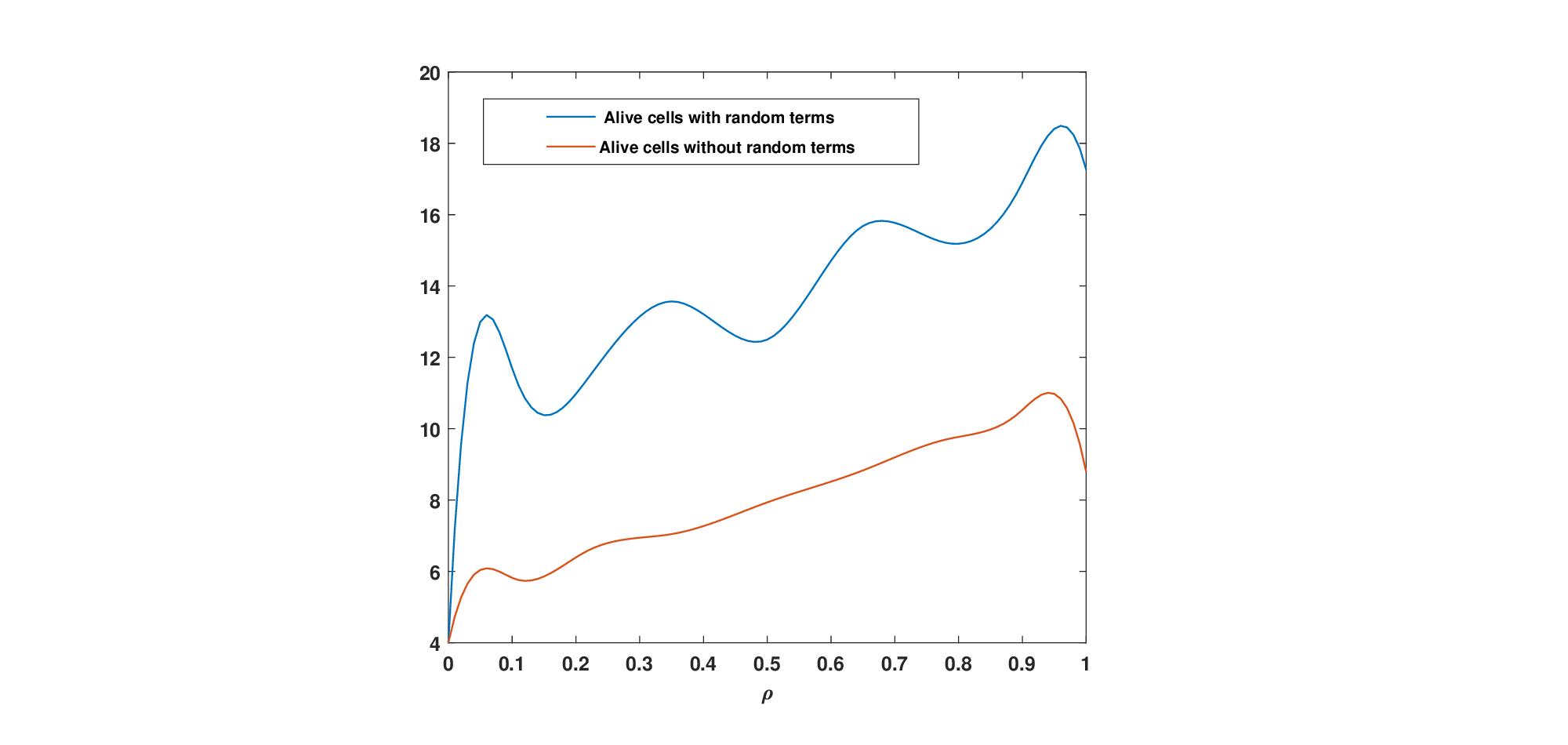}
\vspace{-1cm}\caption{\small{Density of alive cells without stochastic terms  and under the effects of stochastic terms after 1 unit of time. }}\label{fig2a}
\vspace{-.5cm}
\end{center}
\end{figure}
\begin{figure}[!t]
\begin{center}
\hspace{-1.88cm}\includegraphics[width=17.6cm]{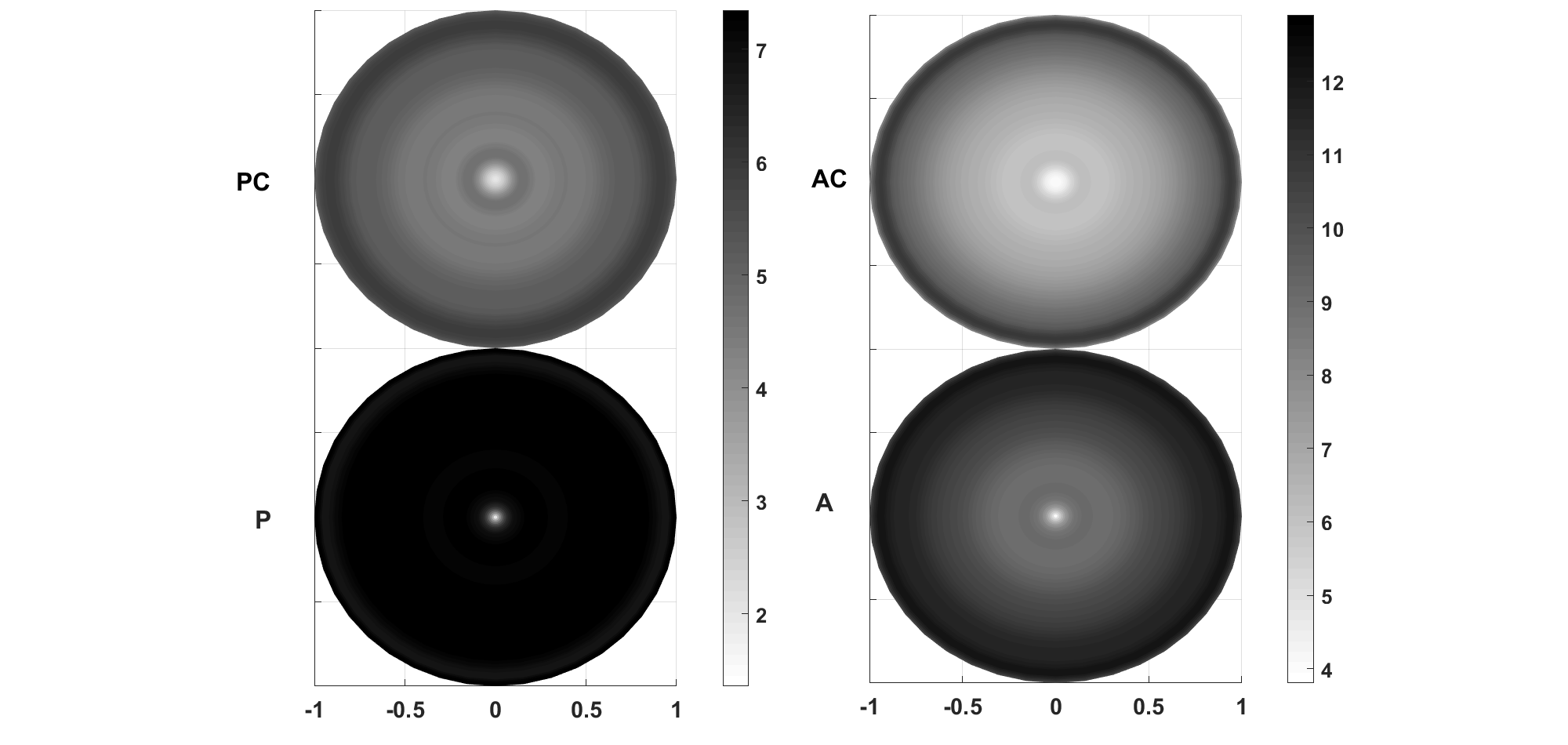}
\vspace{0cm}\caption{\small{Density of proliferative cells  with control (\textbf{PC}) and without control (\textbf{P}) and density of alive cells  with control (\textbf{AC}) and without control (\textbf{A}) on the cross section of tumour passing through the center of the tumour after 1 unit of time.}}\label{fig2b}
\vspace{-.5cm}
\end{center}
\end{figure}
\begin{figure}[!t]
\begin{center}
\hspace{-1.3cm}\includegraphics[width=17.6cm]{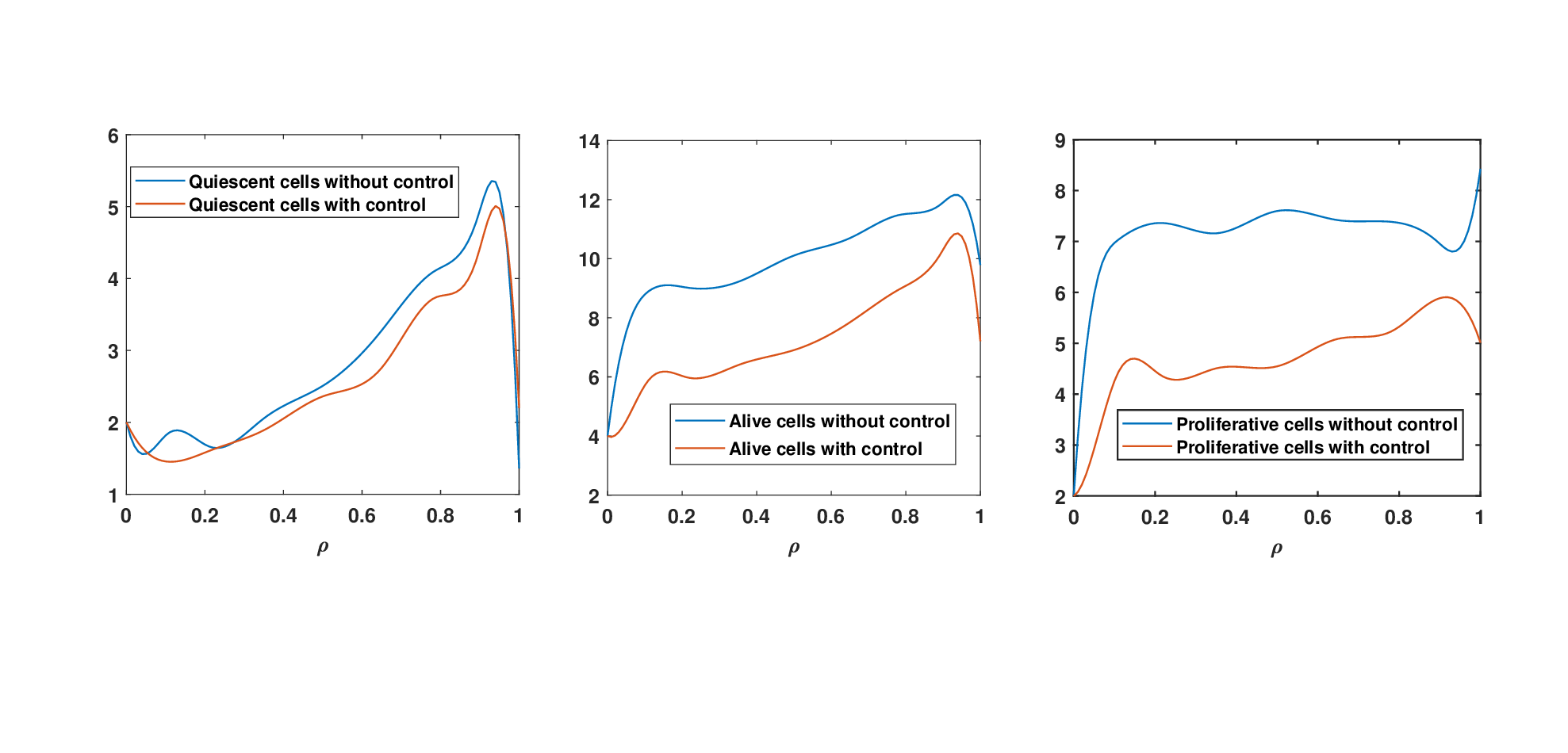}
\vspace{-2.5cm}\caption{\small{ Density of cells without control and under the effects of optimal control after 1 unit of time.}}\label{fig2c}
\vspace{-.5cm}
\end{center}
\end{figure}
\begin{figure}[!t]
\centering
\includegraphics[scale=0.6]{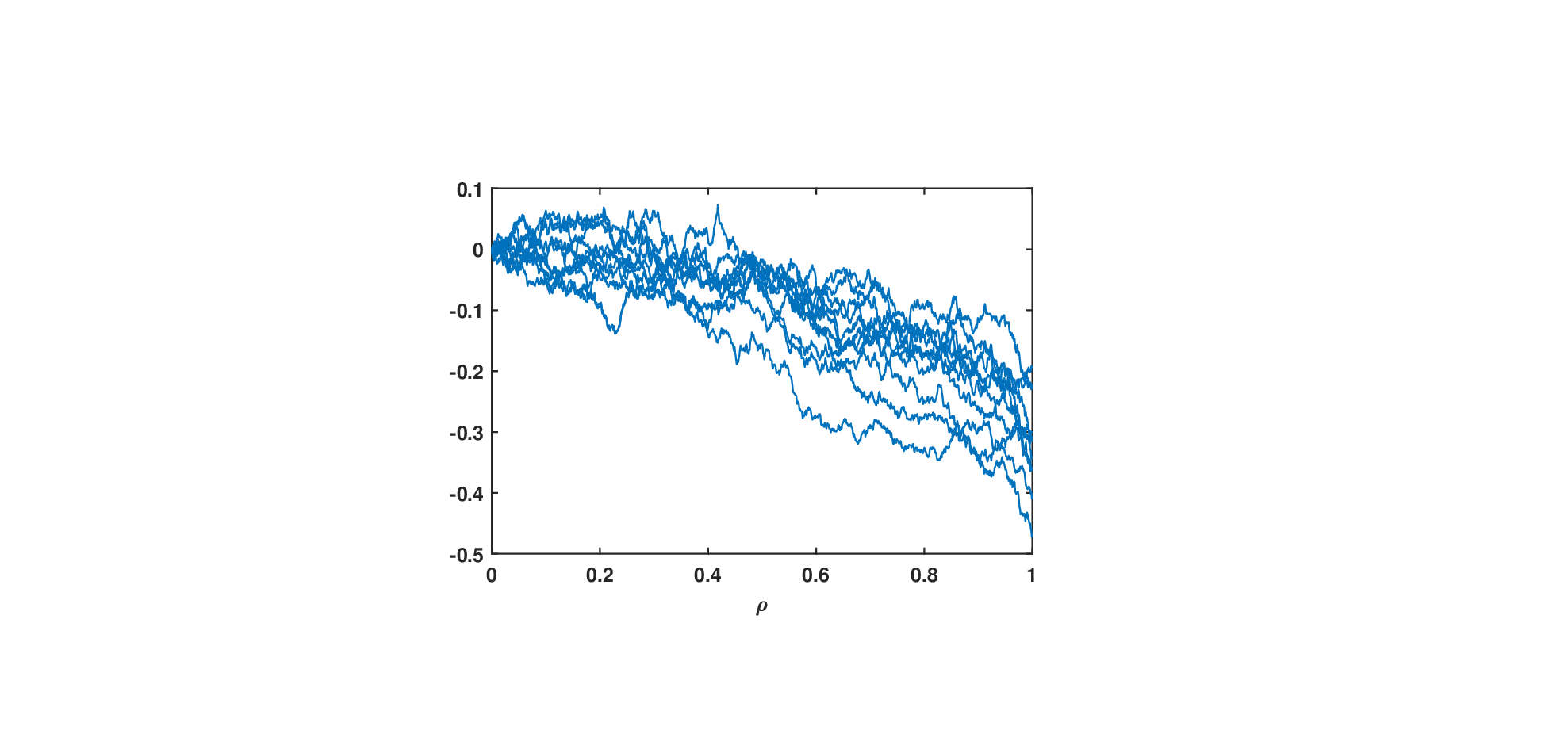}
\caption{Sample paths of relative decrease of tumour radius.}\label{fig2d}
\end{figure}
\end{example}


\section{Conclusions}
\label{cremar}

In this paper, we have studied an optimal control problem for a stochastic model 
of tumour growth. Real environments are stochastic and, in biological systems,
birth rates, competition coefficients, carrying capacities and other parameters 
characterizing natural biological systems exhibit random fluctuations 
\cite{Stability and Complexity in Model Ecosystems}. 
Even weak noise can result in unexpected qualitative shifts
in the dynamics of nonlinear systems \cite{Bashkirtseva}. Therefore, we have added 
the stochastic terms to the deterministic model to work with a more reliable model. 
By providing an example, we have shown the effect of random terms and optimal control 
variables on the density of alive tumour cells. In Figure~\ref{fig2a}, 
it is illustrated that the random terms can increase the density of alive cells 
and affect the dynamic of system. In Figures~\ref{fig2b} and \ref{fig2c}, 
we have shown how the control variables decrease the density of alive cells 
and control the uncertainties caused by the stochastic terms. 
We have also studied the problem when the noises are correlated and we have presented 
the explicit forms of the optimal control variables in terms of adjoint 
states for correlated noises.


\subsection*{Acknowledgements}

The authors are very grateful to the editor and the referees for their valuable comments
and suggestions which improved the original submission of this paper.


\subsection*{Disclosure statement}

No potential conflict of interest was reported by the authors.


\subsection*{Funding}

Torres was supported by the Portuguese Foundation for Science
and Technology (FCT -- Fundação para a Ciência e a Tecnologia) 
through CIDMA, reference UIDB/04106/2020.


\subsection*{Data availability statement}

No datasets were generated or analysed during the current study.


\subsection*{ORCID}

Delfim F. M. Torres (\url{https://orcid.org/0000-0001-8641-2505})



\subsection*{Appendix}

\setcounter{equation}{0}
\renewcommand{\theequation}{A\arabic{equation}}

\begin{appendix}
	
We provide here the proof of some lemmas and theorems.
	
\begin{proof}[Proof of Lemma \ref{Lemma 1ito}] 
First, we consider the following problem
\[d \mathfrak{C} =\dfrac{1}{{\rho }^2}\dfrac{\partial }{\partial\rho }\left({\rho }^2\dfrac{\partial \mathfrak{C} }{\partial\rho }\right)d\tau+\Big(\smallunderbrace{\varphi \left(\tau \right)\rho+2 \sum_{i=0}^m B_i\left(\tau \right) \dfrac{\partial {h_i(\rho)} }{\partial\rho }}_{\varphi_1}\Big)\dfrac{\partial \mathfrak{C} }{\partial\rho }d\tau+{\psi_1}\mathfrak{C}d\tau\]
\begin{equation}\label{ito2}
+e^{-\sum_{i=0}^mh_i(\rho)B_i}f(\rho,\tau)d\tau,~~ 0<\rho <1,~ 0< \tau \leq T,
\end{equation} 
\[ \dfrac{\partial \mathfrak{C}}{\partial\rho}(0,\tau)=0,~~\mathfrak{C} \left(1,\tau \right)=e^{-\sum_{i=0}^m h_i(1)B_i}\overline{c},~~ 0\leq\tau\leq T,~\mathfrak{C} \left(\rho ,0\right)=c_0(\rho),~~ 0\leq\rho \leq 1,\]
where
\[\psi_1=\psi\left(\rho ,\tau \right)+\varphi\sum_{i=0}^m B_i\left(\tau \right)\rho \dfrac{\partial {h_i(\rho)} }{\partial\rho }+\sum_{i=0}^m B_i\left(\tau \right)\dfrac{2}{\rho} \dfrac{\partial {h_i(\rho)} }{\partial\rho }-\dfrac{1}{2}\sum_{i=0}^m(h_i(\rho))^2\]
\[+e^{-\sum_{i=0}^mh_i(\rho)B_i}\left(\dfrac{\partial^2 e^{\sum_{i=0}^mh_i(\rho)B_i}}{\partial \rho^2}\right).\]
Since  $B_i(\tau)$ $(i=0,1,\ldots,m)$ is continuous for almost every $\omega\in \mathbf{\Omega}$, using Lemma \ref{Lemma 1}, for almost every $\omega\in \mathbf{\Omega}$, the problem \eqref{ito2} has a unique solution such that $\mathfrak{C}(|x|,\tau )\in W^{2,1}_p(Q^1_T)$. Now, we define $A_n$ and $\mathcal{F}_n(\tau)$ as follows
\begin{equation}\label{adap}
A_n=\{\omega\in\Omega: |B_k(\tau)|<n+1,k=0,1,\ldots,m\},~~
\mathcal{F}_n(\tau)=\{A\cap A_n:~A\in\mathcal{F}(\tau)\}.
\end{equation}
Clearly $B_k(\tau)\Big|_{ A_n}$ is a measurable function with respect to  $( A_n,\mathcal{F}_n(\tau))$.
 Therefore, from Lemma \ref{Lemma 1} there exists  positive  $\mu_n$ 
 such that
\begin{equation}\label{ito4}
 {||\mathfrak{C}||}_{W^{2,1}_p(Q_T)}\le  \mu_n\left(\left|\overline{c }\right|+{\left|\left|{c}_0\left(|x|\right)\right|\right|}_{D_p\left(\mathfrak{B}\right)}+||e^{-\sum_{k=0}^mh_k(x)B_k}f(|x|,\tau) ||_{L^p(Q^1_T)}\right),~~ \forall\omega\in A_n,
\end{equation}
where $\mu_n$ depends on $T,~p, ~\|\varphi_1\|_{L^\infty(Q^1_T)}$ and $\|\psi_1\|_{L^\infty(Q^1_T)}$. Moreover, if $f\in C(\overline{Q^1_T})$, then there exists positive $\mu^n_0$ such that
\begin{equation}\label{ito6}
{||\mathfrak{C} ||}_{L^{\infty }(Q_T)}\le e^{\mu^n_0T}\Big(\max\{|\overline{c}|+{\left|\left|c_0(|x|,\tau )\right|\right|}_{L^{\infty }(Q^1_T)}\}+T||f(|x|,\tau) ||_{L^\infty(Q^1_T)}\Big),~~\forall\omega\in A_n.
\end{equation} 
It is also clear that the solution of \eqref{ito2}, $\mathfrak{C}$, is a function of $B_k$ $(k=0,\ldots,m)$. From \eqref{ito2} and \eqref{ito6},  $\mathfrak{C}\Big|_{A_n}=\mathfrak{G}(B_0\Big|_{A_n},\ldots,B_m\Big|_{A_n})$ is a continuous function with respect to $B_k\Big|_{A_n}$ $(k=0,\ldots,m)$. Therefore, from Theorem 14.3.1 in \cite{measure},  $\mathfrak{C}\Big|_{A_n}$ is $\mathcal{F}_n(\tau)-$adapted. Thus $\mathfrak{C}\Big|_{A_n}$ is $\mathcal{F}(\tau)-$adapted. Also,   $\mathfrak{C}$ the solution of \eqref{ito2} can be defined as follows
\[\mathfrak{C}=\mathfrak{C}\Big|_{A_n},~~ \omega\in A_n, ~~n\in\mathbb{N}_0.\]
Therefore, for every borel set $B$, $\mathfrak{C}^{-1}(B)=\displaystyle\cup_{n=0}^\infty \mathfrak{C}\Big|_{A_n}^{-1}(B)\in\mathcal{F}(\tau).$ So, $\mathfrak{C}$ is  $\mathcal{F}(\tau)-$adapted.  \\
Also, from Lemma \ref{Lemma 1} we can conclude that there exists a positive constant $d$   such that
\begin{equation}\label{ito5}
{||\partial_x^\beta \mathfrak{C} (|x|,\tau ) ||}_{L^{\infty }(Q^1_T)}\le s'_1(\omega)T^{\frac{1}{2}-\frac{5}{2p}} {||\mathfrak{C} (|x|,\tau )||}_{W^{2,1}_p(Q^1_T)}+s'_2(\omega)\delta^{-1-\frac{5}{p}}||\mathfrak{C}(|x|,\tau )||_{L^p(Q^1_T)}~a.s.,
\end{equation} 
where $|\beta|=1$ and $\delta=\min\{d,\sqrt{T}\}$ and  $s'_1(\omega)$, $s'_2(\omega)$ are positive  and depend on $p$.  Moreover, from \eqref{ito2}, it can be seen that $\mathfrak{C}(|x|,\tau)$ is the solution of 
\[d \mathfrak{C} =\Delta \mathfrak{C} d\tau+\Big(\smallunderbrace{\varphi \left(\tau \right)+2 \sum_{i=0}^m B_i\left(\tau \right)\dfrac{1}{\rho} \dfrac{\partial {h_i(\rho)} }{\partial\rho }}_{\varphi^*_1(\rho,\tau)}\Big)x.\nabla \mathfrak{C} d\tau+\psi_1(\rho,\tau)\mathfrak{C}d\tau+e^{-\sum_{i=0}^mh_i(\rho)B_i}f(\rho,\tau)d\tau,\]
\[~~ 0<\rho=|x| <1,~ 0< \tau \leq T,\] 
\[ ~~\mathfrak{C} \left(1,\tau \right)=e^{-\sum_{i=0}^m h_i(1)B_i}\overline{c},~~ 0\leq\tau\leq T,~\mathfrak{C} \left(|x| ,0\right)=c_0(|x|),~~ 0\leq|x| \leq 1.\]
Since $\mathfrak{C}$ is $\mathcal{F}(\tau)-$adapted, applying the Ito product rule (Corollary 4.6.3 in \cite{steve345}), we have
\[d(\mathfrak{C}^2)=2\mathfrak{C}d\mathfrak{C}=2\mathfrak{C}\Delta \mathfrak{C}d\tau+2\mathfrak{C}  {\varphi_1^*}x.\nabla \mathfrak{C}d\tau+2{\psi_1}\mathfrak{C}^2d\tau+e^{-\sum_{i=0}^mh_i(\rho)B_i}f(|x|,\tau)\mathfrak{C}d\tau,\]
\begin{equation*}
~ 0<\rho=|x| <1,~ 0< \tau \leq T,
\end{equation*} 
\[ ~~\mathfrak{C} \left(1,\tau \right)=e^{-\sum_{i=0}^m h_i(1)B_i}\overline{c},~~ 0\leq\tau\leq T,~\mathfrak{C} \left(|x| ,0\right)=c_0(|x|),~~ 0\leq|x| \leq 1.\]
Therefore, if  $\overline{c}=c_0(|x|)=0$,  we arrive at 
\[\int_{\mathfrak{B}} (\mathfrak{C}(|x|,\tau))^2dx\]
\[\leq\lambda^*(\omega)\left(\int_0^\tau\int_{\mathfrak{B}}\mathfrak{C}^2dxds+\int_0^\tau\int_{\mathfrak{B}} ( e^{-2\sum_{i=0}^mh_i(\rho)B_i}(f(|x|,s))^2)dxds\right)~a.s.,\]
where  
 $\lambda^*(\omega)$  depends on $T,~p, ~\|\varphi\|_{L^\infty(Q^*_T)},~\|\psi\|_{L^\infty(Q^*_T)},~\left\|\dfrac{\partial h_i}{\partial \rho}\right\|_{L^\infty(Q^*_T)},~\left\|\dfrac{\partial^2 h_i}{\partial \rho^2}\right\|_{L^\infty(Q^*_T)}$ and $\|h_i\|_{L^\infty(Q^*_T)}$. Employing the Gronwall inequality, we have
  \begin{equation}\label{eqG}
  \int_{\mathfrak{B}} (\mathfrak{C}(|x|,\tau))^2dx\leq e^{\lambda_1^*(\omega)\tau}\int_0^\tau \int_{\mathfrak{B}}  (f(|x|,s))^2dxds~a.s.,\end{equation}
where  
 $\lambda_1^*(\omega)$ depends on $T,~p, ~\|\varphi\|_{L^\infty(Q^*_T)},~\|\psi\|_{L^\infty(Q^*_T)},~\left\|\dfrac{\partial h_i}{\partial \rho}\right\|_{L^\infty(Q^*_T)},~\left\|\dfrac{\partial^2 h_i}{\partial \rho^2}\right\|_{L^\infty(Q^*_T)}$ and $\|h_i\|_{L^\infty(Q^*_T)}$.

\noindent From the Ito product rule (Corollary 4.6.3 in \cite{steve345}) and general Ito formula (Theorem 4.6 in \cite{oksendal}), we have
\[d(e^{\sum_{i=0}^mh_i(\rho)B_i}\mathfrak{C})=d(e^{\sum_{i=0}^mh_i(\rho)B_i})\mathfrak{C}+e^{\sum_{i=0}^mh_i(\rho)B_i}d\mathfrak{C}\]
\begin{equation}
\label{ito1**99900}
=\mathfrak{C}e^{\sum_{i=0}^mh_i(\rho)B_i}\sum_{i=0}^mh_i(\rho)dB_i+\dfrac{\mathfrak{C}}{2}e^{\sum_{i=0}^mh_i(\rho)B_i}\sum_{i=0}^m(h_i(\rho))^2dt+e^{\sum_{i=0}^mh_i(\rho)B_i}d\mathfrak{C}.\end{equation}
Thus, from \eqref{ito2} and \eqref{ito1**99900}, we can obtain
\[d(\smallunderbrace{e^{\sum_{i=0}^mh_i(\rho)B_i}\mathfrak{C}}_{\mathfrak{C}^1})=\dfrac{1}{{\rho }^2}\dfrac{\partial }{\partial\rho }\left({\rho }^2\dfrac{\partial \mathfrak{C}^1 }{\partial\rho }\right)d\tau\]
\begin{equation*}
+\varphi \left(\tau \right)\rho \dfrac{\partial \mathfrak{C}^1 }{\partial\rho }d\tau+\psi\left(\rho ,\tau \right)\mathfrak{C}^1d\tau+f(\rho,\tau)d\tau+\mathfrak{C}^1\sum_{i=0}^mh_i(\rho)dB_i.
\end{equation*}
Hence,   we conclude that, for almost every $\omega\in \mathbf{\Omega}$, \eqref{ito3} has a solution $\mathfrak{C}^1$ such that  the inequalities \eqref{GH1}--\eqref{eqGkkp} can be obtained from \eqref{ito4}--\eqref{eqG}.
  Now, we want to show that the problem \eqref{ito3} has a unique solution. Let $\mathfrak{C}^1$ and $\mathfrak{C}^2$ be the solutions of \eqref{ito3}. Thus, $u=\mathfrak{C}^1-\mathfrak{C}^2$ is the solution of
\begin{equation*}
{d u }=\dfrac{1}{{\rho }^2}\dfrac{\partial }{\partial\rho }\left({\rho }^2\dfrac{\partial u }{\partial\rho }\right)d\tau+\varphi \left(\tau \right)\rho \dfrac{\partial u }{\partial\rho }d\tau+\psi\left(\rho ,\tau \right)ud\tau+u\sum_{i=0}^mh_idB_i ,~~ 0<\rho <1,~ 0< \tau \leq T,
\end{equation*} 
\[ \dfrac{\partial u}{\partial\rho}(0,\tau)=0,~~u \left(1,\tau \right)=0,~~ 0\leq\tau\leq T,~u \left(\rho ,0\right)=0,~~ 0\leq\rho \leq 1.\]
Therefore, using the Ito product rule and general Ito formula, one can conclude that
 \[u_1=e^{-\sum_{i=0}^mh_i(\rho)B_i}u\] is the solution of
\[d u_1 =\dfrac{1}{{\rho }^2}\dfrac{\partial }{\partial\rho }\left({\rho }^2\dfrac{\partial u_1 }{\partial\rho }\right)d\tau+\Big(\varphi \left(\tau \right)\rho+2 \sum_{i=0}^m B_i\left(\tau \right) \dfrac{\partial {h_i(\rho)} }{\partial\rho }\Big)\dfrac{\partial u_1 }{\partial\rho }d\tau\]
\[+\Bigg(\psi\left(\rho ,\tau \right)+\varphi\sum_{i=0}^m B_i\left(\tau \right)\rho \dfrac{\partial {h_i(\rho)} }{\partial\rho }+\sum_{i=0}^m B_i\left(\tau \right)\dfrac{2}{\rho} \dfrac{\partial {h_i(\rho)} }{\partial\rho }-\dfrac{1}{2}\sum_{i=0}^m(h_i(\rho))^2\]
\[+e^{-\sum_{i=0}^mh_i(\rho)B_i}\left(\dfrac{\partial^2 e^{\sum_{i=0}^mh_i(\rho)B_i}}{\partial \rho^2}\right)\Bigg)u_1d\tau,\]
\[ 0<\rho <1,~ 0< \tau \leq T,\]
\[ \dfrac{\partial u_1}{\partial\rho}(0,\tau)=0,~~u_1 \left(1,\tau \right)=0,~~ 0\leq\tau\leq T,~u_1 \left(\rho ,0\right)=0,~~ 0\leq\rho \leq 1.\]
Finally, using Lemma \ref{Lemma 1},  one can deduce that $u_1=0$ for almost every $\omega\in  \mathbf{\Omega}$, which results in the uniqueness of the solution of the problem \eqref{ito3}.
\end{proof}

\begin{proof}[Proof of Lemma \ref{Lemma2ito}]  
Let, for almost every $\omega\in \mathbf{\Omega}$, $({\alpha}_1,{\alpha}_2,\alpha_3)$ be the solution of 
\[{d\alpha_1 }+z\left(\rho ,\tau \right)\dfrac{\partial\alpha_1 }{\partial\rho }d\tau=(h_{11}\left(\rho ,\tau \right)-\dfrac{1}{2}\sum_{i=0}^m(h^1_i(\rho))^2-z\left(\rho ,\tau \right)\sum_{i=0}^m\dfrac{\partial h^1_i}{\partial \rho}(\rho)B_i)\alpha_1d\tau\]
\begin{equation}
\label{iton1}
+\Big( e^{\sum_{i=0}^mh^2_i(\rho)B_i-h^1_i(\rho)B_i}h_{12}\left(\rho ,\tau \right)\alpha_2 +e^{\sum_{i=0}^mh^3_i(\rho)B_i-h^1_i(\rho)B_i}h_{13}\left(\rho ,\tau \right)\alpha_3+e^{-\sum_{i=0}^mh^1_i(\rho)B_i}g_1(\rho,\tau)\Big)d\tau,
\end{equation}
\[{d\alpha_2 }+z\left(\rho ,\tau \right)\dfrac{\partial\alpha_2 }{\partial\rho }d\tau=(h_{22}\left(\rho ,\tau \right)-\dfrac{1}{2}\sum_{i=0}^m(h^2_i(\rho))^2-z\left(\rho ,\tau \right)\sum_{i=0}^m\dfrac{\partial h^2_i}{\partial \rho}(\rho)B_i)\alpha_2d\tau\]
\begin{equation}
+\Big( e^{\sum_{i=0}^mh^1_i(\rho)B_i-h^2_i(\rho)B_i}h_{21}\left(\rho ,\tau \right)\alpha_1 +e^{\sum_{i=0}^mh^3_i(\rho)B_i-h^2_i(\rho)B_i}h_{23}\left(\rho ,\tau \right)\alpha_3+e^{-\sum_{i=0}^mh^2_i(\rho)B_i}g_2(\rho,\tau)\Big)d\tau,
\end{equation}
\[{d\alpha_3 }+z\left(\rho ,\tau \right)\dfrac{\partial\alpha_3 }{\partial\rho }d\tau=(h_{33}\left(\rho ,\tau \right)-\dfrac{1}{2}\sum_{i=0}^m(h^3_i(\rho))^2-z\left(\rho ,\tau \right)\sum_{i=0}^m\dfrac{\partial h^3_i}{\partial \rho}(\rho)B_i)\alpha_3d\tau\]
\begin{equation}
\label{iton2}
+\Big( e^{\sum_{i=0}^mh^1_i(\rho)B_i-h^3_i(\rho)B_i}h_{31}\left(\rho ,\tau \right)\alpha_1 +e^{\sum_{i=0}^mh^2_i(\rho)B_i-h^3_i(\rho)B_i}h_{32}\left(\rho ,\tau \right)\alpha_2+e^{-\sum_{i=0}^mh^3_i(\rho)B_i}g_3(\rho,\tau)\Big)d\tau,
\end{equation}
\[~0\le \rho \le 1, \ 0<\tau \le T, \]
\[\alpha_1 \left(\rho ,0\right)=\alpha_0\left(\rho \right),\ \alpha_2 \left(\rho ,0\right)=\beta_0\left(\rho \right),\ \alpha_3 \left(\rho ,0\right)=\gamma_0\left(\rho \right),~\ 0\le \rho \le 1. \] 
In a way similar to the proof of Lemma \ref{Lemma 1ito}, by defining $A_n$ and $\mathcal{F}_n(\tau)$ (see \eqref{adap}) we can show that $\alpha_1$, $\alpha_2$ and $\alpha_3$ are $\mathcal{F}(\tau)-$adapted. By considering the characteristic equation 
\[\dfrac{d\xi(\rho,\tau)}{d \tau}=z(\xi(\rho,\tau),\tau),~~\xi(\rho,0)=\rho,\]
for equations \eqref{iton1}--\eqref{iton2}, one can deduce that there exists positive $\iota_3(\omega)$ and $\iota_4(\omega)$ such that
\begin{equation}\label{iton7}
{\|(\alpha_1 ,\alpha_2 ,\alpha_3) \|}_{L^{\infty }([0,1])}\le \iota_3(\omega)\left({\|({\alpha }_0,{\beta}_0,{\gamma }_0)\|}_{L^{\infty }([0,1])}+\int_0^\tau{\|(g_1,g_2,g_3)\|}_{L^{\infty }([0,1])}ds\right)~a.s.
\end{equation}
and
\begin{equation}\label{iton8}
{\|(\alpha_1 ,\alpha_2 ,\alpha_3) \|}^2_{L^{2 }(Q^*_\tau)}\le \iota_4(\omega)\left({\|({\alpha }_0,{\beta}_0,{\gamma }_0)\|}^2_{L^{2 }(Q^*_\tau)}+\int_0^\tau{\|(g_1,g_2,g_3)\|}^2_{L^{2 }(Q^*_s)}ds\right)~a.s.
\end{equation}
 From the Ito product rule and general Ito formula, we have
\[d(\smallunderbrace{\overbrace{e^{\sum_{i=0}^mh^j_i(\rho)B_i}}^{W'_j}\alpha_j}_{\alpha^1_j})=d(W'_j)\alpha_j+W'_jd\alpha_j\]
\begin{equation*}
=\alpha_jW'_j\sum_{i=0}^mh^j_i(\rho)dB_i+\dfrac{\alpha_j}{2}W'_j\sum_{i=0}^m(h^j_i(\rho))^2dt+W'_jd\alpha_j,~~~j=1,2,3.
\end{equation*}
Therefore, we have
\begin{equation}\label{rrrr12}
{d\alpha_1^1 }+z\left(\rho ,\tau \right)\dfrac{\partial\alpha^1_1 }{\partial\rho }d\tau=\Big(h_{11}\left(\rho ,\tau \right)\alpha^1_1 +h_{12}\left(\rho ,\tau \right)\alpha_2^1 +h_{13}\left(\rho ,\tau \right)\alpha^1_3+g_1(\rho,\tau)\Big)d\tau+\alpha^1_1\sum_{i=0}^mh^1_i(\rho)dB_i, 
\end{equation}
\begin{equation}
{d\alpha^1_2 }+z\left(\rho ,\tau \right)\dfrac{\partial\alpha^1_2 }{\partial\rho }d\tau=\Big(h_{21}\left(\rho ,\tau \right)\alpha^1_1 +h_{22}\left(\rho ,\tau \right)\alpha^1_2 +h_{23}\left(\rho ,\tau \right)\alpha^1_3+g_2(\rho,\tau)\Big)d\tau+\alpha^1_2\sum_{i=0}^mh^2_i(\rho)dB_i,
\end{equation} 
\begin{equation}\label{rrrr122}
{d\alpha^1_3 }+z\left(\rho ,\tau \right)\dfrac{\partial\alpha^1_3 }{\partial\rho }d\tau=\Big(h_{31}\left(\rho ,\tau \right)\alpha^1_1 +h_{32}\left(\rho ,\tau \right)\alpha^1_2 +h_{33}\left(\rho ,\tau \right)\alpha^1_3+g_3(\rho,\tau)\Big)d\tau+\alpha^1_3\sum_{i=0}^mh^3_i(\rho)dB_i, 
\end{equation}
\[~0\le \rho \le 1, \ 0<\tau \le T. \]
So, applying  Lemma \ref{Lemma2} for \eqref{iton1}--\eqref{iton2} 
and using Assumption~\ref{assumpt}, we can conclude that for almost every $\omega\in \mathbf{\Omega}$ the problem \eqref{ito12}--\eqref{ito13} has a continuous solution.  Now, we want to show that the problem \eqref{ito12}--\eqref{ito13} has a unique solution. Let $(\alpha_1,\beta_1,\gamma_1)$ and  $(\alpha_2,\beta_2,\gamma_2)$ be solutions of the problem \eqref{ito12}--\eqref{ito13}. Then $(u_1,u_2,u_3)=(\alpha_1-\alpha_2,\beta_1-\beta_2,\gamma_1-\gamma_2)$ is solution of
\begin{equation*}
{du_1 }+z\left(\rho ,\tau \right)\dfrac{\partial u_1 }{\partial\rho }d\tau=\Big(h_{11}\left(\rho ,\tau \right)u_1 +h_{12}\left(\rho ,\tau \right)u_2 +h_{13}\left(\rho ,\tau \right)u_3\Big)d\tau+u_1 dW_1, 
\end{equation*}
\begin{equation*}
{du_2 }+z\left(\rho ,\tau \right)\dfrac{\partial u_2 }{\partial\rho }d\tau=\Big(h_{21}\left(\rho ,\tau \right)u_1 +h_{22}\left(\rho ,\tau \right)u_2 +h_{23}\left(\rho ,\tau \right)u_3\Big)d\tau+u_2 dW_2,
\end{equation*} 
\begin{equation*}
{du_3 }+z\left(\rho ,\tau \right)\dfrac{\partial u_3 }{\partial\rho }d\tau=\Big(h_{31}\left(\rho ,\tau \right)u_1 +h_{32}\left(\rho ,\tau \right)u_2 +h_{33}\left(\rho ,\tau \right)u_3\Big)d\tau+u_3 dW_3,
\end{equation*} 
\[~0\le \rho \le 1, \ 0<\tau \le T, \]
\[u_1 \left(\rho ,0\right)=0,\ u_2 \left(\rho ,0\right)=0,\ u_3 \left(\rho ,0\right)=0,~\ 0\le \rho \le 1. \] 
Employing the Ito product rule and general Ito formula, it can be shown that $(\mathfrak{u}_1,\mathfrak{u}_2,\mathfrak{u}_3)=$ \newline$(e^{-\sum_{i=0}^mh^1_i(\rho)B_i}u_1,e^{-\sum_{i=0}^mh^2_i(\rho)B_i}u_2,e^{-\sum_{i=0}^mh^3_i(\rho)B_i}u_3)$ is the solution of
\[{d\mathfrak{u}_1 }+z\left(\rho ,\tau \right)\dfrac{\partial\mathfrak{u}_1 }{\partial\rho }d\tau=(h_{11}\left(\rho ,\tau \right)-\dfrac{1}{2}\sum_{i=0}^m(h^1_i(\rho))^2-z\left(\rho ,\tau \right)\sum_{i=0}^m\dfrac{\partial h^1_i}{\partial \rho}(\rho)B_i)\mathfrak{u}_1d\tau\]
\begin{equation}
\label{ito15}
+\Big( e^{\sum_{i=0}^mh^2_i(\rho)B_i-h^1_i(\rho)B_i}h_{12}\left(\rho ,\tau \right)\mathfrak{u}_2 +e^{\sum_{i=0}^mh^3_i(\rho)B_i-h^1_i(\rho)B_i}h_{13}\left(\rho ,\tau \right)\mathfrak{u}_3\Big)d\tau,
\end{equation}
\[{d\mathfrak{u}_2 }+z\left(\rho ,\tau \right)\dfrac{\partial\mathfrak{u}_2 }{\partial\rho }d\tau=(h_{22}\left(\rho ,\tau \right)-\dfrac{1}{2}\sum_{i=0}^m(h^2_i(\rho))^2-z\left(\rho ,\tau \right)\sum_{i=0}^m\dfrac{\partial h^2_i}{\partial \rho}(\rho)B_i)\mathfrak{u}_2d\tau\]
\begin{equation}
+\Big( e^{\sum_{i=0}^mh^1_i(\rho)B_i-h^2_i(\rho)B_i}h_{21}\left(\rho ,\tau \right)\mathfrak{u}_1 +e^{\sum_{i=0}^mh^3_i(\rho)B_i-h^2_i(\rho)B_i}h_{23}\left(\rho ,\tau \right)\mathfrak{u}_3\Big)d\tau,
\end{equation}
\[{d\mathfrak{u}_3 }+z\left(\rho ,\tau \right)\dfrac{\partial\mathfrak{u}_3 }{\partial\rho }d\tau=(h_{33}\left(\rho ,\tau \right)-\dfrac{1}{2}\sum_{i=0}^m(h^3_i(\rho))^2-z\left(\rho ,\tau \right)\sum_{i=0}^m\dfrac{\partial h^3_i}{\partial \rho}(\rho)B_i)\mathfrak{u}_3d\tau\]
\begin{equation}\label{ito14}
+\Big( e^{\sum_{i=0}^mh^1_i(\rho)B_i-h^3_i(\rho)B_i}h_{31}\left(\rho ,\tau \right)\mathfrak{u}_1 +e^{\sum_{i=0}^mh^2_i(\rho)B_i-h^3_i(\rho)B_i}h_{32}\left(\rho ,\tau \right)\mathfrak{u}_2\Big)d\tau,
\end{equation}
\[~0\le \rho \le 1, \ 0<\tau \le T, \]
\[\mathfrak{u}_1 \left(\rho ,0\right)= 0,\ \mathfrak{u}_2 \left(\rho ,0\right)=0,\ \mathfrak{u}_3 \left(\rho ,0\right)=0,~\ 0\le \rho \le 1. \]
From Lemma \ref{Lemma2}, one can deduce that the problem \eqref{ito15}--\eqref{ito14}, for almost every $\omega\in \mathbf{\Omega}$, has the unique solution  $(\mathfrak{u}_1,\mathfrak{u}_2,\mathfrak{u}_3)=(0,0,0)$. Hence, we can conclude that for almost every $\omega\in \mathbf{\Omega}$, the problem \eqref{ito12}--\eqref{ito13} has a unique solution and the inequalities \eqref{khfgt}--\eqref{khfgt1} can be obtained from \eqref{iton7}--\eqref{iton8}.  If we also assume that $h_{ij}(\rho ,\tau )\ (i, j=1, 2, 3)$ and $g_{i}(\rho ,\tau)\ (i= 1, 2, 3)$ are continuously differentiable with respect to $\rho$ and ${\alpha }_0,\ {\beta }_0,\ {\gamma }_0\in \ C^1[0, 1]$, then from \eqref{iton1}--\eqref{iton2}, \eqref{rrrr12}--\eqref{rrrr122} and Lemma \ref{Lemma2}, one can conclude that the solution of the problem \eqref{ito12}--\eqref{ito13} is continuously differentiable with respect to $\rho$, for almost every $\omega\in \mathbf{\Omega}$, such that the inequality \eqref{nemi} can be obtained. 
\end{proof}

\begin{proof}[Proof of Theorem \ref{ncon}] 
Let $\left(c^*,w^*,p^*,q^*,a^*,\eta^*\right)$ and $\left(c^{\epsilon},w^{\epsilon},p^{\epsilon},q^{\epsilon},a^{\epsilon},\eta^{\epsilon}\right)$ be the solutions of  (\ref{YAS})--(\ref{HM345}) corresponding to the controls  $(u_1^*,u_2^*,{\overline{c}}^*,{\overline{w}}^*)\in  \mathcal{C}_{ad}\times\mathcal{W}_{ad}\times\mathcal{C}^1_{ad}\times\mathcal{W}^1_{ad}$ and $(u_1^{\epsilon},u_2^{\epsilon},{\overline{c}}^\epsilon,{\overline{w}}^\epsilon) =(u_1^*+\epsilon u_1^0,u_2^*+\epsilon u_2^0,{\overline{c}}^*+\epsilon {\overline{c}}_0,{\overline{w}}^*+\epsilon {\overline{w}}_0) \in  \mathcal{C}_{ad}\times\mathcal{W}_{ad}\times\mathcal{C}^1_{ad}\times\mathcal{W}^1_{ad}$, respectively, where $\mathcal{C}_{ad}$,  $\mathcal{W}_{ad}$, $\mathcal{C}^1_{ad}$ and  $\mathcal{W}^1_{ad}$ are defined in (\ref{YAK*15555})--(\ref{YAK*255555}) and $\epsilon$  is a positive constant. Also assume that  
\[c^{\epsilon}_1=\dfrac{c^{\epsilon}-c^*}{\epsilon},~w^{\epsilon}_1=\dfrac{w^{\epsilon}-w^*}{\epsilon},
 ~p^{\epsilon}_1=\dfrac{p^{\epsilon}-p^*}{\epsilon},~q^{\epsilon}_1=\dfrac{q^{\epsilon}-q^*}{\epsilon},~a^{\epsilon}_1=\dfrac{a^{\epsilon}-a^*}{\epsilon},~ \eta^{\epsilon}_1=\dfrac{\eta^{\epsilon}-\eta^*}{\epsilon},\]
and $J(u_1^*,u_2^*,{\overline{c}}^*,{\overline{w}}^*)=\min \Big\{J(u_1,u_2,\overline{c},\overline{w}):(u_1,u_2,\overline{c},\overline{w})\in \mathcal{C}_{ad}\times\mathcal{W}_{ad}\times\mathcal{C}^1_{ad}\times\mathcal{W}^1_{ad}\Big\}.$
 Therefore, we have
\[\dfrac{J(u_1^{\epsilon},u_2^{\epsilon},{\overline{c}}^\epsilon,{\overline{w}}^\epsilon)-J(u_1^*,u_2^*,{\overline{c}}^*,{\overline{w}}^*)}{\epsilon}\]
\[=\Big(\int_0^1\int_0^T\rho^2\Big(\lambda_1(p^*+p^{\epsilon})p^{\epsilon}_1+\lambda_2(q^*+q^{\epsilon})q^{\epsilon}_1+(u_1^*+u_1^{\epsilon}-2r_1^*)u_1^0+(u_2^*+u_2^{\epsilon}-2r_2^*)u_2^0\Big)d\rho d\tau\]
\begin{equation}\label{AA*222}
+\int_0^T({\overline{c}}^*+{\overline{c}}^\epsilon-2r_3^*){\overline{c}}_0+({\overline{w}}^*+{\overline{w}}^\epsilon-2r_4^*){\overline{w}}_0d\tau\Big)\geq 0~~a.s.,
\end{equation}
where $ c^{\epsilon}_1=\dfrac{c^{\epsilon}-c^*}{\epsilon}$ is the solution of the following parabolic equation obtained from \eqref{YAS}--\eqref{opb1}:
\[{d c^{\epsilon}_1}-D_{1}\dfrac{1}{\rho^2}\dfrac{\partial }{\partial\rho }\left({\rho }^{2}\dfrac{\partial c^{\epsilon}_1}{\partial\rho }\right)d\tau\]
\begin{equation*}
=u^{\epsilon}\left(1, \tau \right)\rho \dfrac{\partial c^{\epsilon}_1}{\partial\rho }d\tau+u^{\epsilon,1}\left(1, \tau \right)\rho \dfrac{\partial c^*}{\partial\rho }d\tau-
\dfrac{({\eta^{\epsilon}})^2f\left(c^{\epsilon},p^{\epsilon},q^{\epsilon}\right)-({\eta^*})^2f\left(c^*,p^*,q^*\right)}{\epsilon}d\tau+u^0_{1}d\tau+c_1^\epsilon\displaystyle\sum_{i=0}^mh_i^1dB_i,
\end{equation*}
\[ {\rm \ 0}<\rho <1,\ \tau {\rm >}0,\]
\begin{equation*}
\dfrac{\partial c^{\epsilon}_1}{\partial\rho }\left(0,\tau \right){\rm =0, }~c^{\epsilon}_1\left({\rm 1,}\tau \right){\rm =}{\overline{c}}_0(\tau), ~~\ \tau {\rm >}0,
\end{equation*} 
\begin{equation*}
c^{\epsilon}_1\left(\rho ,0\right){\rm =}0,~{\rm  \ 0}\le \rho \le {\rm 1};
\end{equation*}
 and $w^{\epsilon}_1=\dfrac{w^{\epsilon}-w^*}{\epsilon}$ is the solution of the following parabolic equation  obtained from \eqref{YAS676767}--\eqref{YAS1}:
\[d w^{\epsilon}_1-D_{2}\dfrac{1}{{\rho }^{2}}\dfrac{\partial }{\partial\rho }\left({\rho }^{2}\dfrac{\partial w^{\epsilon}_1}{\partial\rho }\right)d\tau\]
\begin{equation*}
=u^{\epsilon}\left(1, \tau \right)\rho \dfrac{\partial w^{\epsilon}_1}{\partial\rho }d\tau+u^{\epsilon,1}\left(1, \tau \right)\rho \dfrac{\partial w^*}{\partial\rho }d\tau-
\dfrac{({\eta^{\epsilon}})^2g\left(w^{\epsilon},p^{\epsilon},q^{\epsilon}\right)-({\eta^*})^2g\left(w^*,p^*,q^*\right)}{\epsilon}d\tau+u^0_{2}d\tau+w_1^\epsilon\displaystyle\sum_{i=0}^mh_i^2dB_i, 
\end{equation*}
\[{\rm \ 0}<\rho <1,\ \tau {\rm >}0,\]
\begin{equation*}
\dfrac{\partial w^{\epsilon}_1}{\partial\rho }\left(0,\tau \right){\rm =0, }~w^{\epsilon}_1\left({\rm 1,}\tau \right)={\overline{w}}_0(\tau), ~~\ \tau {\rm >}0,
\end{equation*}
\begin{equation*}
w^{\epsilon}_1\left(\rho ,0\right){\rm =}0,~{\rm \ \ 0}\le \rho \le {\rm 1};
\end{equation*}
and $p^{\epsilon}_1=\dfrac{p^{\epsilon}-p^*}{\epsilon}$, $q^{\epsilon}_1=\dfrac{q^{\epsilon}-q^*}{\epsilon}$ and $a^{\epsilon}_1=\dfrac{a^{\epsilon}-a^*}{\epsilon}$ are the solutions of the following hyperbolic equations  obtained from \eqref{pqd}--\eqref{YAS2}:
\begin{eqnarray*}
&{d p^{\epsilon}_1}{\rm +}v^{\epsilon}\dfrac{\partial p^{\epsilon}_1}{\partial\rho }d\tau+{\rm }v^{\epsilon}_1\dfrac{\partial p^*}{\partial\rho }d\tau{\rm =}
\dfrac{ G_1(\eta^{\epsilon},c^{\epsilon},w^{\epsilon},p^{\epsilon},q^{\epsilon},a^{\epsilon})-G_1(\eta^*,c^*,w^*,p^*,q^*,a^*)}{\epsilon}d\tau+p_1^\epsilon\displaystyle\sum_{i=0}^mh_i^3dB_i,&\\  
&{d q^{\epsilon}_1}{\rm +}v^{\epsilon}\dfrac{\partial q^{\epsilon}_1}{\partial\rho }d\tau+{\rm }v^{{\epsilon}}_1\dfrac{\partial q^*}{\partial\rho }d\tau{\rm =}
\dfrac{ G_2(\eta^{\epsilon},c^{\epsilon},w^{\epsilon},p^{\epsilon},q^{\epsilon},a^{\epsilon})-G_2(\eta^*,c^*,w^*,p^*,q^*,a^*)}{\epsilon}d\tau+q_1^\epsilon\displaystyle\sum_{i=0}^mh_i^4dB_i,~&\\  
&{d a^{\epsilon}_1}{\rm +}v^{\epsilon}\dfrac{\partial a^{\epsilon}_1}{\partial\rho }d\tau+{\rm }v^{\epsilon}_1\dfrac{\partial a^*}{\partial\rho }d\tau{\rm =}
\dfrac{ G_3(\eta^{\epsilon},c^{\epsilon},w^{\epsilon},p^{\epsilon},q^{\epsilon},a^{\epsilon})-G_3(\eta^*,c^*,w^*,p^*,q^*,a^*)}{\epsilon}d\tau+a_1^\epsilon\displaystyle\sum_{i=0}^mh_i^5dB_i,&  
\end{eqnarray*}
\[{\rm 0}\le \rho \le {\rm 1,\ }\tau {\rm >}0,\]
\begin{equation*}
~~~~~~~~~~~~~~~~~~~p^\epsilon_1\left(\rho,0\right)=0,~~~q^\epsilon_1\left(\rho,0\right)=0,~~~ a^\epsilon_1\left(\rho,0\right)=0,~~~~{\rm 0}\le \rho \le {\rm 1\ };
\end{equation*}
while the following equations are  obtained from \eqref{HM345}:
\begin{eqnarray*}  
&\dfrac{{\rm 1}}{{\rho }^{{\rm 2}}}\dfrac{\partial }{\partial\rho }\left({\rho }^{{\rm 2}}u^{\epsilon,1}\right){\rm =}\dfrac{({\eta^{\epsilon}})^2 h\left(c^{\epsilon},w^{\epsilon},p^{\epsilon},q^{\epsilon},a^{\epsilon}\right)-({\eta^*})^2 h\left(c^*,w^*,p^*,q^*,a^*\right)}{\epsilon},~~~{\rm  0}< \rho \le {\rm 1,\ }\tau {\rm >}0,&\nonumber\\
&\dfrac{{\rm 1}}{{\rho }^{{\rm 2}}}\dfrac{\partial }{\partial\rho }\left({\rho }^{{\rm 2}}u^{\epsilon}\right){\rm =}({\eta^{\epsilon}})^2h\left(c^{\epsilon},w^{\epsilon},p^{\epsilon},q^{\epsilon},a^{\epsilon}\right),~~~{\rm  0}< \rho \le {\rm 1,\ }\tau {\rm >}0,&\nonumber\\
&{d\eta^{\epsilon}_1 (\tau)}=\eta^{\epsilon}_1(\tau)u^\epsilon( 1,\tau )d\tau+\eta^* (\tau)u^{\epsilon,1}( 1,\tau)d\tau+\eta_1^\epsilon dB_*,~~~\tau>0,&\\
& \eta^{\epsilon}_1(0)=0,~~u^{\epsilon,1}\left(0,\tau \right){\rm =0,\ \ }~~u^{\epsilon}\left(0,\tau \right){\rm =0,\ \ }~\tau>0,&\nonumber\\
&v^{\epsilon}_1\left(\rho ,\tau \right){\rm =}u^{\epsilon,1}\left(\rho ,\tau \right){\rm -}\rho u^{\epsilon,1}\left({\rm 1,}\tau \right),~~v^{\epsilon}\left(\rho ,\tau \right){\rm =}u^{\epsilon}\left(\rho ,\tau \right){\rm -}\rho u^{\epsilon}\left({\rm 1,}\tau \right),~~{\rm  0}\le \rho \le {\rm 1,\ }\tau {\rm >}0.&\nonumber 
\end{eqnarray*}
Therefore, using (\ref{YAS*1e1})--(\ref{HM345**ee}),  \textbf{A}, Theorem \ref{??}  and Lemmas \ref{Lemma 1ito} and \ref{Lemma2ito}, we deduce that
\[\int_0^T\int_0^1\rho^2 \Big(c^\epsilon_1{d z_c}+z_c{d c^\epsilon_1}+w^\epsilon_1{d z_w}+z_w{d w^\epsilon_1}+p^\epsilon_1{d z_p}+z_p{d p^\epsilon_1}+q^\epsilon_1{d z_q}+z_q{d q^\epsilon_1} +a^\epsilon_1{d z_a}+z_a{d a^\epsilon_1}\Big) d\rho \]
\[+\int_0^T\int_0^1\rho^2 \Big(dc^\epsilon_1{d z_c}+dw^\epsilon_1{d z_w}+dp^\epsilon_1{d z_p}+dq^\epsilon_1{d z_q} +da^\epsilon_1{d z_a}\Big) d\rho \]
\[+\int_0^T(\eta^\epsilon_1{d z_\eta}+z_\eta{d \eta^\epsilon_1}+d\eta^\epsilon_1{d z_\eta}) \]
\begin{equation}\label{correneed1}
=-\int_0^T\int_0^1\rho^2\Big(\lambda_1p^*p^\epsilon_1+\lambda_2q^*q^\epsilon_1-z_cu_1^0- z_wu_2^0\Big)d\rho d\tau+\varpi(\epsilon,\omega,\rho, T)-\int_0^T(D_1\dfrac{\partial z_c}{\partial\rho}(1,\tau)\overline{c}_0+D_2\dfrac{\partial z_w}{\partial\rho}(1,\tau)\overline{w}_0)d\tau,
\end{equation}
where, for almost every $\omega\in \mathbf{\Omega}$, $\displaystyle{\lim_{\epsilon\rightarrow 0}}\varpi(\epsilon,\omega,\rho, T)=0$. Thus we have
\[\int_0^T\int_0^1\rho^2\Big(\lambda_1p^*p^\epsilon_1+\lambda_2q^*q^\epsilon_1-z_cu_1^0- z_wu_2^0\Big)d\rho d\tau-\varpi(\epsilon,\omega,\rho, T)+\int_0^T(D_1\dfrac{\partial z_c}{\partial\rho}(1,\tau)\overline{c}_0+D_2\dfrac{\partial z_w}{\partial\rho}(1,\tau)\overline{w}_0)d\tau\]
\begin{equation}
\label{AA*e}
=-\int_0^T\int_0^1\rho^2 d\left({ z_cc^\epsilon_1}+ z_ww^\epsilon_1+ z_pp^\epsilon_1+ z_qq^\epsilon_1+\ z_aa^\epsilon_1\right)d\rho -\int_0^Td( z_\eta\eta^\epsilon_1) =0.
\end{equation}
So, using (\ref{AA*222}) and  (\ref{AA*e}), for almost every $\omega\in \mathbf{\Omega}$, we arrive at
\[\int_0^1\int_0^T\rho^2\Big(( z_c+u_1^*-r_1^*)u_1^0+( z_w+u_2^*-r_2^*)u_2^0\Big)d\rho d\tau+\int_0^T\left(\overline{c}^*-r_3^*- D_1\dfrac{\partial z_c}{\partial\rho}(1,\tau)\right){\overline{c}}_0(\tau)d\tau\]
\begin{equation}\label{correneed3}
+\int_0^T\left(\overline{w}^*-r_4^*- D_2\dfrac{\partial z_w}{\partial\rho}(1,\tau)\right){\overline{w}}_0(\tau)d\tau \geq 0.
\end{equation}
Therefore, using tangent-normal cone techniques (see Subsection 5.3 in \cite{25.}) and  (\ref{YAK*15555})--(\ref{YAK*255555}), for almost every $\omega\in \mathbf{\Omega}$, $(u_1^*,u_2^*,{\overline{c}}^*,{\overline{w}}^*)$  is as follows:
\[
u_1^*=\mathcal{F}_1(-z_c+r_1^*),~u_2^*=\mathcal{F}_2(-z_w+r_2^*),~
{\overline{c}}^*=\mathcal{F}_3\left(D_1\dfrac{\partial z_c}{\partial\rho}(1,\tau)+r_3^*\right),~{\overline{w}}^*=\mathcal{F}_4\left(D_2\dfrac{\partial z_w}{\partial\rho}(1,\tau)+r_4^*\right).
\]
The proof is complete.
\end{proof}

\begin{proof}[Proof of Theorem \ref{YKHFKH}]
We prove this theorem for almost every $\omega\in\mathbf{\Omega}$. Using Theorem \ref{??},  we conclude that for each 
\[(u_1,u_2,\overline{c},\overline{w})\in\overline{\mathcal{C}_{ad}\times\mathcal{W}_{ad}\times\mathcal{C}^1_{ad}\times\mathcal{W}^1_{ad}}\setminus\mathcal{C}_{ad}\times\mathcal{W}_{ad}\times\mathcal{C}^1_{ad}\times\mathcal{W}^1_{ad},\]
where $\overline{\mathcal{C}_{ad}\times\mathcal{W}_{ad}\times\mathcal{C}^1_{ad}\times\mathcal{W}^1_{ad}}$ is the completion of ${\mathcal{C}_{ad}\times\mathcal{W}_{ad}\times\mathcal{C}^1_{ad}\times\mathcal{W}^1_{ad}}$ with respect to the following norm
\[{L^4_2([0, 1]\times [0, T])}:=L^2([0, 1]\times [0, T])\times L^2([0, 1]\times [0, T])\times L^2([0, T])\times L^2 ([0, T]),\]
 there exists a unique value ${J^*}(u_1,u_2,\overline{c},\overline{w})$ such that for every sequence $\Big\{(u^n_1,u^n_2,\overline{c}^n,\overline{w}^n)\Big\}_{n=0}^{\infty}$ in $\mathcal{C}_{ad}\times\mathcal{W}_{ad}\times\mathcal{C}^1_{ad}\times\mathcal{W}^1_{ad}$, which converges to $(u_1,u_2,\overline{c},\overline{w})$ in ${L^4_2([0, 1]\times [0, T])}$,  ${J}(u^n_1,u^n_2,\overline{c}^n,\overline{w}^n)$ converges to ${J^*}(u_1,u_2,\overline{c},\overline{w})$. Using the Lebesgue Dominated Convergence Theorem, we deduce that the function
\begin{equation*}
\mathfrak{J}(u_1,u_2,\overline{c},\overline{w})=
\begin{cases}
{J}(u_1,u_2,\overline{c},\overline{w}),~~~~(u_1,u_2,\overline{c},\overline{w})\in{\mathcal{C}_{ad}\times\mathcal{W}_{ad}\times\mathcal{C}^1_{ad}\times\mathcal{W}^1_{ad}},\\
{J^*}(u_1,u_2,\overline{c},\overline{w}),~~~~(u_1,u_2,\overline{c},\overline{w})\in\overline{\mathcal{C}_{ad}\times\mathcal{W}_{ad}\times\mathcal{C}^1_{ad}\times\mathcal{W}^1_{ad}}\setminus {\mathcal{C}_{ad}\times\mathcal{W}_{ad}\times\mathcal{C}^1_{ad}\times\mathcal{W}^1_{ad}},\\
+\infty,~~~~~~~~~~~~~~~~~(u_1,u_2,\overline{c},\overline{w})\notin\overline{\mathcal{C}_{ad}\times\mathcal{W}_{ad}\times\mathcal{C}^1_{ad}\times\mathcal{W}^1_{ad}},\\
\end{cases}
\end{equation*}
is lower semicontinuous with respect to $(u_1,u_2,{\overline{c}},{\overline{w}})$ in ${L_1^{4,2}([0, 1]\times [0, T])}$, where
\[{L^{4,2}_1([0, 1]\times [0, T])}:=\{(f_1,f_2,f_3,f_3): (\rho^2f_1,\rho^2f_2,f_3,f_3)\in L^{4}_1([0, 1]\times [0, T])\},\]
and 
\[{L^4_1([0, 1]\times [0, T])}:=L^1([0, 1]\times [0, T])\times L^1([0, 1]\times [0, T])\times L^1([0, T])\times L^1 ([0, T]).\]
 Using Theorem \ref{v principle},   we  conclude that for each positive $\epsilon$ there exists $u_\epsilon$ such that
\[\mathfrak{J}(\smallunderbrace{u_1^\epsilon,u^\epsilon_2,{\overline{c}}^\epsilon,{\overline{w}^\epsilon}}_{u_\epsilon})<\mathfrak{J}(u_1,u_2,{\overline{c}},{\overline{w}})\]
\begin{equation*}
+\epsilon^{\frac{1}{2}}\Big(\|\rho^2(u_1-u_1^\epsilon)\|_{L^1([0, 1]\times [0, T])}+\|\rho^2(u_2-u_2^\epsilon)\|_{L^1([0, 1]\times [0, T])}+\|\overline{c}-\overline{c}^\epsilon\|_{L^1 ([0, T])}+\|\overline{w}-\overline{w}^\epsilon\|_{L^1 ([0, T])}\Big),
\end{equation*}
\[\forall(u_1,u_2,\overline{c},\overline{w})\not=(u_1^\epsilon,u^\epsilon_2,\overline{c}^\epsilon,\overline{w}^\epsilon).\]
Also  each sequence $\Big\{(u_1^n,u^n_2,\overline{c}^n,\overline{w}^n)\Big\}_{n=0}^{\infty}$ in $\mathcal{C}_{ad}\times\mathcal{W}_{ad}\times\mathcal{C}^1_{ad}\times\mathcal{W}^1_{ad}$,
 which converges to $(u^\epsilon_1,u^\epsilon_2,\overline{c}^\epsilon,\overline{w}^\epsilon)$  in $L_2^4([0, 1]\times [0, T])$, is a Cauchy sequence in $L_2^4([0, 1]\times [0, T])$. Also assume that for each $n\in\mathbb{N}_0$, $(c_{n},w_{n},p_{n},q_{n},a_{n},\eta_{n})$ be the solution of   (\ref{YAS})--(\ref{HM345}) corresponding to $(u_1^{n},u^{n}_2,\overline{c}^{n},\overline{w}^{n})$. Therefore, according to Theorem \ref{??}, the sequence $\Big\{(c_{n},w_{n},p_{n},q_{n},a_{n},\eta_{n})\Big\}_{n=0}^{\infty}$ is a Cauchy sequence in
  \[ L^2(Q^1_T)\times L^2(Q^1_T)\times L^2(Q^*_T)\times L^2(Q^*_T) \times L^2(Q^*_T)\times L^2([0, T]).\] 
  So, using the Lebesgue Dominated Convergence Theorem, we derive that  for each $p>5$ there exist subsequences $\Big\{(u_1^{n_k},u^{n_k}_2,\overline{c}^{n_k},\overline{w}^{n_k})\Big\}_{k=0}^{\infty}$ and $\Big\{(c_{n_k},w_{n_k},p_{n_k},q_{n_k},a_{n_k},\eta_{n_k})\Big\}_{k=0}^{\infty}$, which are Cauchy in 
\[L_p^4([0, 1]\times [0, T]):= L^p([0, 1]\times [0, T])\times L^p([0, 1]\times [0, T])\times L^p([0, T])\times L^p ([0, T]),\]
and
   \[ L^p(Q^1_T)\times L^p(Q^1_T)\times L^p(Q^*_T)\times L^p(Q^*_T) \times L^p(Q^*_T)\times L^p([0, T]),\]  
respectively. Thus, using Lemma \ref{Lemma 1ito}, it is easy to see that the sequences 
\[\Big\{(e^{-\sum_{i=0}^mh_i^1B_i}c_{n_k},e^{-\sum_{i=0}^mh_i^2B_i}w_{n_k})\Big\}_{k=0}^{\infty},\]
and
\[\Big\{(e^{\sum_{i=0}^mh_i^1B_i}z_{c_{n_k}},e^{\sum_{i=0}^mh_i^2B_i}z_{w_{n_k}})\Big\}_{k=0}^{\infty},\]
 are Cauchy in 
$W^{2,1,2}_p(Q^1_T):=W^{2,1}_p(Q^1_T)\times W^{2,1}_p(Q^1_T),$
where for each $k\in\mathbb{N}_0$ \[(z_{c_{n_k}},z_{w_{n_k}},z_{p_{n_k}},z_{q_{n_k}},z_{a_{n_k}},z_{\eta_{n_k}}),\] is the solution of  (\ref{YAS*1e1})--(\ref{HM345**ee})  corresponding to  $(u_1^{n_k},u^{n_k}_2,\overline{c}^{n_k},\overline{w}^{n_k})$.
Therefore, there exist sequences
\[\Big\{\smallunderbrace{(c_{k_l},w_{k_l},p_{k_l},q_{k_l},a_{k_l},\eta_{k_l},z_{c_{k_l}},z_{w_{k_l}},z_{p_{k_l}},z_{q_{k_l}},z_{a_{k_l}},z_{\eta_{k_l}})}_{A_l}\Big\}_{l=0}^{\infty},~~\Big\{\smallunderbrace{(u_1^{k_l},u^{k_l}_2,\overline{c}^{k_l},\overline{w}^{k_l})}_{\mathfrak{Z}_l}\Big\}_{l=0}^{\infty},\]
and
$A=(c^\epsilon,w^\epsilon,p^\epsilon,q^\epsilon,a^\epsilon,\eta^\epsilon,z^\epsilon_{c},z^\epsilon_{w},z^\epsilon_{p},z^\epsilon_{q},z^\epsilon_{a},z^\epsilon_{\eta})$
 such that
\[\lim_{l\rightarrow\infty}A^*_l=A^*,\]
\[~in~~\smallunderbrace{W^{2,1,2}_p(Q^1_T)\times L^p(Q^*_T)\times L^p(Q^*_T) \times L^p(Q^*_T)\times L^p([0, T])}_{\mathbf{W}_p}\times\mathbf{W}_p,\]
\[~\lim_{l\rightarrow\infty}\mathfrak{Z}_l=(u^\epsilon_1,u^\epsilon_2,\overline{c}^\epsilon,\overline{w}^\epsilon),~in~L^4_p([0, 1]\times [0, T]),\] 
where 
\[A^*=(e^{-\sum_{i=0}^mh_i^1B_i}c^\epsilon,e^{-\sum_{i=0}^mh_i^2B_i}w^\epsilon,p^\epsilon,q^\epsilon,a^\epsilon,\eta^\epsilon,e^{\sum_{i=0}^mh_i^1B_i}z^\epsilon_{c},e^{\sum_{i=0}^mh_i^2B_i}z^\epsilon_{w},z^\epsilon_{p},z^\epsilon_{q},z^\epsilon_{a},z^\epsilon_{\eta}),\]
and
\[A_l^*=(e^{-\sum_{i=0}^mh_i^1B_i}c_{k_l},e^{-\sum_{i=0}^mh_i^2B_i}w_{k_l},p_{k_l},q_{k_l},a_{k_l},\eta_{k_l},e^{\sum_{i=0}^mh_i^1B_i}z_{c_{k_l}},e^{\sum_{i=0}^mh_i^2B_i}z_{w_{k_l}},z_{p_{k_l}},z_{q_{k_l}},z_{a_{k_l}},z_{\eta_{k_l}}).\]
From the $t$-Anisotropic Embedding Theorem (see Theorem 1.4.1 in \cite{21.}), we conclude that  $c^\epsilon, w^\epsilon, z_c^\epsilon,$ and $z_w^\epsilon$ belong to $C({Q^1_T})$. Using the proof of Theorem \ref{ncon} and tangent-normal cone techniques (see Proposition 5.3 in \cite{25.}), for  $\epsilon$ small enough we have
\begin{equation}\label{YA1111}
|u_1^\epsilon-\mathcal{F}_1(-z^\epsilon_c+r_1^*)|\leq \epsilon^{\frac{1}{2}},~|u_2^\epsilon-\mathcal{F}_2(-z^\epsilon_w+r_2^*)|\leq \epsilon^{\frac{1}{2}},~
\end{equation}
\begin{equation}
|{\overline{c}}^\epsilon-\mathcal{F}_3\left(D_1\dfrac{\partial z^\epsilon_c}{\partial\rho}(1,\tau)+r_3^*\right)|\leq\epsilon^{\frac{1}{2}},
\end{equation}
\begin{equation}\label{YA1***}
~|{\overline{w}}^\epsilon-\mathcal{F}_4\left(D_2\dfrac{\partial z^\epsilon_w}{\partial\rho}(1,\tau)+r_4^*\right)|\leq\epsilon^{\frac{1}{2}}.
\end{equation}
 From  Lemmas   \ref{lem} and \ref{Lemma 1ito}  and Theorem \ref{???}, it is clear that  the function 
\[\mathcal{H}:\mathcal{C}_{ad}\times\mathcal{W}_{ad}\times\mathcal{C}^1_{ad}\times\mathcal{W}^1_{ad}\longrightarrow\mathcal{C}_{ad}\times\mathcal{W}_{ad}\times\mathcal{C}^1_{ad}\times\mathcal{W}^1_{ad},\]
\begin{equation*} \mathcal{H}(u_1,u_2,\overline{c},\overline{w})=\Big(\mathcal{F}_1(-z_c+r_1^*),\mathcal{F}_2(-z_w+r_2^*),\mathcal{F}_3(D_1\dfrac{\partial z_c}{\partial\rho}(1,\tau)+r_3^*),\mathcal{F}_4( D_2\dfrac{\partial z_w}{\partial\rho}(1,\tau)+r_4^*)\Big),
\end{equation*}
where $z_c$ and $z_w$ are adjoint states corresponding to $(u_1,u_2,\overline{c},\overline{w})$,  has a unique fixed point $(u_1^*,u_2^*,{\overline{c}}^*,{\overline{w}}^*)$. Also, using  (\ref{YA1111})--(\ref{YA1***}), Lemma \ref{Lemma 1ito}, and Theorems  \ref{??} and \ref{???}, it is easy to show that  the sequence $\{(u_1^\epsilon,u_2^\epsilon,\overline{c}^\epsilon,\overline{w}^\epsilon)\}$ has a subsequence which converges to $(u_1^*,u_2^*,\overline{c}^*,\overline{w}^*)$ in $  L^2(Q^1_T)\times L^2(Q^1_T) \times L^2([0,T])\times L^2([0, T])$, as $\epsilon$ converges to zero. Therefore, we have
\[{J}(u_1^*,u_2^*,{\overline{c}}^*,{\overline{w}}^*)\leq{J}(u_1,u_2,{\overline{c}},{\overline{w}}),~~\forall(u_1,u_2,{\overline{c}},{\overline{w}})\in \mathcal{C}_{ad}\times\mathcal{W}_{ad}\times\mathcal{C}^1_{ad}\times\mathcal{W}^1_{ad}.\]
Moreover, from    \eqref{YA1**e5e}--\eqref{YA2**} and Theorem \ref{AMHZ}, one can deduce that $u_1^*$, $u_2^*$, ${\overline{c}}^*$, ${\overline{w}}^*$ are $\mathcal{F}(\tau)-$adapted.  
\end{proof}
\end{appendix}



\begin{thebibliography}{99}
	
\bibitem{2.}  
Zhao, J. 
A parabolic-hyperbolic free boundary problem modeling tumour growth with drug application. 
Electron. J. Differ. Eq. 2010; 2010:1--18.
\url{https://doi.org/10.1155/2010/620459}
	
\bibitem{7.} 
Cui, S., Wei, X. 
Global existence for a parabolic-hyperbolic free boundary problem modelling tumour growth. 
Acta Math. Appl. Sinica (Eng Ser). 2005;21(4):597--614.
\url{https://doi.org/10.1007/s10255-005-0268-1}
	
\bibitem{8.1}  
Esmaili, S.,  Eslahchi, M. R. 
Application of collocation method for solving
a parabolic-hyperbolic free boundary problem 
which models the growth of tumour with drug application. 
Math. Meth. Appl. Sci. 2017;40(5):1711--1733.
\url{https://doi.org/10.1002/mma.v40.5}
	
\bibitem{8.2} 
Esmaili, S., Eslahchi, M.R. 
Optimal control for a parabolic-hyperbolic free
boundary problem modeling the growth of tumour with drug application. 
J. Optim. Theory. Appl. 2017;173(3):1013--1041.
\url{https://doi.org/10.1007/s10957-016-1037-4}

\bibitem{4.} 
Tao, Y., Chen, M. 
An elliptic-hyperbolic free boundary problem modelling cancer therapy. 
Nonlinearity 2006;19(2):419--440.
\url{https://doi.org/10.1088/0951-7715/19/2/010}

\bibitem{8.} 
Tao, Y. 
A free boundary problem modeling the cell cycle 
and cell movement in multicellular tumour spheroids. 
J. Diff. Eq. 2009;247(1):49--68.
\url{https://doi.org/10.1016/j.jde.2009.04.005}

\bibitem{11.}  
Khaitan, D., Chandna, S., Arya, M.B., Dwarakanath, B.S. 
Establishment and characterization of multicellular spheroids 
from a human glioma cell line: implications for tumour therapy. 
J. Trans. Med. 2006;4(1):12--25.
\url{https://doi.org/10.1186/1479-5876-4-12}

\bibitem{stoch1}  
Lima, E.A.B.F.,  Almeida, R.C., Oden, J.T. 
Analysis and numerical solution of stochastic 
phase-field models of tumour growth. 
Numer. Methods Partial Differ. Equ. 2015;31(2):552--574.
\url{https://doi.org/10.1002/num.v31.2}

\bibitem{stoch2} 
Albano, G. and Giorno, V. 
A stochastic model in tumour growth. 
J. Theor. Biol. 2006;242(2):329--336.
\url{https://doi.org/10.1016/j.jtbi.2006.03.001}

\bibitem{stoch3} 
Albano, G., Giorno, V., Rom\'an-Rom\'an, P., Torres-Ruiz, F. 
Inference on a stochastic two-compartment model in tumour growth. 
Comput. Stat. Data Anal. 2012;56(6):1723--1736.
\url{https://doi.org/10.1016/j.csda.2011.10.016}

\bibitem{im1} 
Mueller-Klieser, W., Schreiber-Klais, S., Walenta, S., Kreuter, M.H.  
Bioactivity of well-defined green tea extracts in multicellular tumour spheroids. 
Int. J. Oncol. 2002;21:1307--1315.
\url{https://doi.org/10.3892/ijo.21.6.1307}

\bibitem{im2} 
Schaller, G., Meyer-Hermann, M.  
Continuum versus discrete model: a comparison for multicellular tumour spheroids. 
Phil. Trans. R. Soc. A 2006;364(1843):1443--1464.
\url{https://doi.org/10.1098/rsta.2006.1780}
	
\bibitem{im3} 
Chandrasekaran, S., King, M.R. 
Gather round: in vitro tumour spheroids as improved models of in vivo tumours. 
J. Bioeng. Biomed. Sci. 2012;2(04):e109.
\url{https://doi.org/10.4172/2155-9538}

\bibitem{30.} 
de Araujo, A.L.A., de Magalh\~aes, P.M.D. 
Existence of solutions and optimal control 
for a model of tissue invasion by solid tumours. 
J. Math. Anal. Appl. 2015;421(1):842--877.
\url{https://doi.org/10.1016/j.jmaa.2014.07.038}

\bibitem{31.} 
Anderson, A.R.A. 
A hybrid mathematical model of solid tumour invasion: 
The importance of cell adhesion. 
Math. Med. Biol. 2005;22(2):163--186.
\url{https://doi.org/10.1093/imammb/dqi005}

\bibitem{17.} 
Calzada, M.C., Fern\'andez-Cara, E.,  Mar\'in, M. 
Optimal control oriented to therapy 
for a free-boundary tumour growth model. 
J. Theor. Biol. 2013;325:1--11.
\url{https://doi.org/10.1016/j.jtbi.2013.02.004}

\bibitem{27.} 
Greenspan, H.P. 
Models for the growth of a solid tumour by diffusion. 
Stud. Appl. Math. 1972;51(4):317--340.
\url{https://doi.org/10.1002/sapm1972514317}

\bibitem{28.} 
Greenspan, H.P. 
On the growth and stability of cell cultures and solid tumours. 
J. Theor. Biol. 1976;56(1):229--242.
\url{https://doi.org/10.1016/S0022-5193(76)80054-9}

\bibitem{29.} 
Byrne, H.M., Chaplain, M.A.J. 
Growth of nonnecrotic tumours in the presence and absence of inhibitors. 
Math. Biosci. 1995;130(2):151--181.
\url{https://doi.org/10.1016/0025-5564(94)00117-3}

\bibitem{29.5}  
Oke, S.I.,  Matadi, M.B., Xulu, S.S. 
Optimal control of breast cancer:
investigating estrogen as a risk factor.
In: Recent Advances in Mathematical and Statistical Methods; 2018:451--463.
\url{https://doi.org/10.1007/978-3-319-99719-3_41}.

\bibitem{29.4} 
Camacho, A., Jerez, S. 
Bone metastasis treatment modeling via optimal control. 
J. Math. Biol. 2019;78(1-2):497--526.
\url{https://doi.org/10.1007/s00285-018-1281-3}

\bibitem{Stability and Complexity in Model Ecosystems} 
May, R. M. 
Stability and Complexity in Model Ecosystems. 
Princeton: Princeton Univ. Press; 1973.

\bibitem{the color of} 
Vasseur, D. A., Yodzis P. 
The color of environmental noise. 
Ecology. 2004;85(4):1146--1152.
\url{https://doi.org/10.1890/02-3122}

\bibitem{Bashkirtseva} 
Bashkirtseva, I.,  Ryashko, L. 
Analysis of noise-induced phenomena 
in the nonlinear tumour-immune system. 
Physica A 2020;549:123923.
\url{https://doi.org/10.1016/j.physa.2019.123923}

\bibitem{29.3} 
Aboulaich, R.,  Darouichi, A.,  Elmouki, I., Jraifi, A. 
A stochastic optimal control model for BCG
immunotherapy in superficial bladder cancer. 
Math. Model. Nat. Phenom. 2017;12(5):99--119.
\url{https://doi.org/10.1051/mmnp/201712507}

\bibitem{29.1} 
Tannenbaum, A., Georgiou, T., Deasy, J., Norton, L.  
Control and the analysis of cancer growth models; 2018. 
bioRxiv. 
\url{https://doi.org/10.1101/244301}

\bibitem{29.2}  
Jerez, S., Cant\'o, J.A. 
A stochastic model for the evolution of bone metastasis: 
Persistence and recovery. 
J. Comput. Appl. Math. 2019;347:12--23.
\url{https://doi.org/10.1016/j.cam.2018.07.047}

\bibitem{stochtumour1}  
Li, S.,  Huang, Y. 
Mean first-passage time of a tumour cell growth system 
with time delay and colored cross-correlated noises excitation. 
J. Low. Freq. Noise V. A. 2018;37(2):191--198.
\url{https://doi.org/10.1177/1461348417725948}

\bibitem{stochtumour2} 
Xu, W.,  Hao, M., Gu, X., Yang, G. 
Stochastic resonance induced by L\'evy noise 
in a tumour growth model with periodic treatment. 
Mod. Phys. Lett. B 2014;28:1450085, 12~pp.
\url{https://doi.org/10.1142/S0217984914500857}

\bibitem{stochtumour3}  
Ai, B.Q.,  Wang, X.J., Liu, G.T., Liu, L.G. 
Correlated noise in a logistic growth model. 
Phys Rev E 2003;67(2): 022903.
\url{https://doi.org/10.1103/PhysRevE.67.022903}

\bibitem{stochtumour4}  
d'Onofrio, A.: 
Bounded-noise-induced transitions in a tumour-immune system interplay. 
Physical Review E 2010;81:021923.
\url{https://doi.org/10.1103/PhysRevE.81.021923}

\bibitem{8.3} 
Esmaili, S., Eslahchi, M. R. 
Application of fixed point-collocation method for solving 
an optimal control problem of a parabolic-hyperbolic free 
boundary problem modeling the growth of tumour with drug application. 
Comput. Math. Appl. 2018;75(7):2193--2216.
\url{https://doi.org/10.1016/j.camwa.2017.11.005}
	
\bibitem{8.4} 
Esmaili, S., Eslahchi, M. R. 
Numerical solution of optimal control problem 
for a model of tumour growth with drug application. 
Int. J. Control. 2019;92(11):2712--2736.
\url{https://doi.org/10.1080/00207179.2018.1458159}
	
\bibitem{21.}   
Wu, Z., Yin, J., Wang, C. 
Elliptic and Parabolic Equations. 
Singapore: World Scientific; 2006.
	
\bibitem{24.} 
Barbu, V. 
Mathematical Methods in Optimization of Differential Systems. 
Dordrecht: Kluwer Academic Publishers; 1994.

\bibitem{22.} 
Ladyzenskaja, O.A., Solonnikov, V.A.,  Ural$'$ceva, N.N. 
Linear and Quasi-Linear Equations of Parabolic Type. 
Providence: American Mathematical Society; 1968.

\bibitem{brnotioncal} 
Wiersema U.F. 
Brownian Motion Calculus. 
Chichester: John Wiley \& Sons; 2008.

\bibitem{25.} 
Barbu, V., Iannelli,  M. 
Optimal control of population dynamics. 
J. Optimiz. Theory App. 1999;102(1):1--14.
\url{https://doi.org/10.1023/A:1021865709529}

\bibitem{measure} 
Strichartz, R.S. 
The Way of Analysis.  
Boston: Jones and Bartlett; 2000.

\bibitem{steve345} 
Shreve, S.E. 
Stochastic Calculus for Finance II: Continuous-Time Models. 
New York: Springer; 2004.
	
\bibitem{oksendal}  
{\O}ksendal, B. 
Stochastic Differential Equations: An Introduction with Applications. 2th ed.
Berlin: Springer; 1985.
	
\end{thebibliography}
\end{document}